\documentclass[10pt]{amsart}
\include{macro}

\newcommand{\suchthat}{\ | \ }

\newcommand{\marked}{\mathbb{M}}

\newcommand{\orbset}{\mathbb{O}}
\newcommand{\surf}{(\Sigma,\marked,\orbset)}

\newcommand{\KQ}[1]{\Bbbk \langle #1\rangle}

\newcommand{\moduss}[1]{\operatorname{\modu_{ #1-ss}}}
\newcommand{\stau}[1]{\operatorname{s\tau-tilt}{#1}}
\newcommand{\rstau}[1]{\operatorname{rs\tau-tilt}{#1}}
\newcommand{\ab}[1]{\operatorname{Ab}({#1})}

\theoremstyle{remark}

\title[Gentle algebras from surfaces with orbifold points I]{Gentle algebras arising from surfaces with orbifold points of order 3, Part I: scattering diagrams}
\author{Daniel Labardini-Fragoso}
\address{Daniel Labardini-Fragoso\newline
Instituto de Matem\'aticas, UNAM, Mexico, and\newline
Mathematisches Institut, Universit\"at zu K\"oln, Germany}
\email{labardini@im.unam.mx, dlabardi@uni-koeln.de}

\author{Lang Mou}
\address{Lang Mou\newline
Department of Pure Mathematics and Mathematical Statistics,\newline
University of Cambridge, CB3 0WB, United Kingdom}
\email{lmou.math@gmail.com}
\date{}

\begin{document}
\maketitle

\begin{abstract}
To each triangulation of any surface with marked points on the boundary and orbifold points of order three, we associate a quiver (with loops) with potential whose Jacobian algebra is finite dimensional and gentle. We study the stability scattering diagrams of such gentle algebras and use them to prove that the Caldero--Chapoton map defines a bijection between $\tau$-rigid pairs and cluster monomials of the generalized cluster algebra associated to the surface by Chekhov and Shapiro.
\end{abstract}

\tableofcontents

\section{Introduction}

The interplay between Fomin--Zelevinsky theory of cluster algebras \cite{fomin2002cluster} and Teichm\"uller theory of surfaces with marked points originated in the works of Fock--Goncharov \cite{fock2007dual}, Fomin--Shapiro--Thurston \cite{fomin2008cluster}, Gekhtman--Shapiro--Vainshtein \cite{gekhtman2005cluster} and Penner \cite{penner1987decorated}. By allowing the cluster exchange relations to be dictated by polynomials more general than just binomials --thus defining \emph{generalized cluster algebras}, Chekhov--Shapiro \cite{chekhov2011teichmller} have extended this interplay to surfaces with orbifold points. Generalized cluster algebras thus play the role of coordinate rings of decorated Teichm\"uller spaces of such orbifolds, with each triangulation providing a full coordinatization via \emph{lambda lengths}, and with Chekhov--Shapiro's polynomials describing the change of coordinates corresponding to a flip of triangulations.

Since the invention of cluster algebras, numerous authors have successfully developed representation-theoretic approaches, that have resulted, for instance, in expressions of cluster variables as Caldero--Chapoton functions of quiver representations \cite{caldero2006cluster, caldero2006triangulated, derksen2010quivers, geiss2018quivers, palu2008cluster, plamondon2011cluster, demonet2011categorification}.
For generalized cluster algebras, however, so far there have been only a couple of works providing such representation-theoretic expressions for cluster variables, cf. \cite{paquette2019group,labardini2019on}.

In \cite{labardini2019on}, Labardini--Velasco associate quivers with potential to the triangulations of a polygon with one orbifold point of order 3, and prove that the cluster monomials in the (coefficient-free) generalized cluster algebra of the latter are Caldero--Chapoton functions of $\tau$-rigid pairs over the Jacobian algebras of those QPs. In the present note, we extend this to the framework of surfaces with marked points on the boundary and orbifold points of order 3, thus providing a vast generalization of Labardini--Velasco's~work. 

A bit more precisely, our construction possesses a few special features. Namely, the Jacobian algebras of the QPs we define are finite-dimensional and gentle, and the underlying quiver has loops arising naturally from the orbifold points. Furthermore, the representation theories of the QPs associated to triangulations related by a flip are closely related through adjoint pairs of reflection functors, which we construct for these QPs. These functors allow us to show that the Bridgeland stability scattering diagrams \cite{bridgeland2016scattering} of the QPs satisfy what Gross--Hacking--Keel--Kontsevich \cite{gross2018canonical} call \emph{mutation invariance}. This fact leads to a cluster-like wall-crossing structure in the real vector space of stability conditions, whose walls we decorate with a representation-theoretic incarnation of Chekhov--Shapiro's generalized cluster transformations. As an upshot, we obtain our main result, namely, that for surfaces with marked points on the boundary and orbifold points of order $3$, the generalized cluster monomials are precisely the Caldero--Chapoton functions of reachable $\tau$-rigid pairs, see \Cref{thm: main theorem} below.

It would be interesting to find out how our work is related to Paquette--Schiffler's approach \cite{paquette2019group}, which provides certain generalized cluster exchange relations in \emph{orbit cluster algebras} arising from quivers with potential with admissible group actions. Our methods are direct, and different from those in \cite{paquette2019group}, in the sense that we do not resort to covers of surfaces, quivers or algebras. 

Let us give a more detailed summary of our results.

\subsection{Generalized cluster algebras}

In \Cref{section: generalized cluster algebra} we recall some of the basic facts surrounding Chekhov--Shapiro's generalized cluster algebras. One starts with certain cluster exchange data which includes a skew-symmetrizable matrix $B$, a tuple of positive divisors of the columns of $B$ and a set of exchange polynomials. This data is used to define the notion of \emph{generalized seed mutation}. The main difference with Fomin--Zelevinsky's seed mutation is that now the exchange polynomials are allowed to have more than just two terms. The \emph{generalized cluster algebra} $\mathcal{A}(B)$ is defined to be the subalgebra generated inside the rational function field by all the clusters in the seeds obtained by applying all possible finite sequences of generalized seed mutations to a given initial seed.

Afterwards, in \Cref{section: surfaces} we describe how the triangulations $\kappa$ of surfaces with marked points and orbifold points $\surf$ give rise to generalized cluster algebras. Each triangulation $\kappa$ has an associated skew-symmetrizable matrix $B(\kappa)$ with rows and columns labeled by the arcs in $\kappa$. The essential observation is that triangulations $\kappa$ and $\mu_k(\kappa)$ related by a flip at $k\in \kappa$ have matrices related by the corresponding matrix mutation of Fomin--Zelevinsky.
The exchange polynomials are binomials or trinomials, depending on whether the corresponding arc is \emph{pending} or not. The associated generalized cluster algebra $\mathcal{A}\surf=\mathcal{A}(B(\kappa))$ essentially depends only on the surface with marked points and orbifold points, not on the triangulation $\kappa$. As mentioned above, $\mathcal{A}\surf$ plays a relevant geometric role as coordinate ring of the decorated Teichm\"uller space of $\surf$.

\subsection{Gentle algebras associated to orbifolds}
Let $\mathbb S = \surf$ be a surface $\Sigma$ with non-empty boundary, a set of marked points $\mathbb M\subset \partial \Sigma$ and a set of interior points $\mathbb O\subset \Sigma\setminus \partial \Sigma$ which we call \emph{orbifold points of order 3}. A \emph{triangulation} of $\mathbb S$ is a maximal collection $\kappa$  of compatible arcs with endpoints in $\mathbb M$ up to isotopy relative to $\mathbb M \cup \mathbb O$. To such a triangulation $\kappa$, we associate a quiver (with loops) with potential $(Q(\kappa),S(\kappa))$ whose Jacobian algebra $\mathcal P(\kappa)$ is finite dimensional and gentle; see \Cref{def: quiver with potential of a triangulation}. In the case without any orbifold points, this construction gives a quiver with potential whose flip/mutation dynamics was established in \cite{labardini2009quivers}, and whose representation theory was studied in \cite{assem2010gentle}.

\subsection{Stability conditions and $\tau$-tilting theory}
We show that the category $\modu \mathcal P(\kappa)$ of finite dimensional modules over $\mathcal P(\kappa)$ is informative for understanding $\mathcal A(\kappa) = \mathcal A(B(\kappa))$ in the sense of \Cref{thm: main theorem}. The connection is made through scattering diagrams or wall-crossing structures originated in \cite{gross2011real} and \cite{kontsevich2006affine, kontsevich2014wall}. We study Bridgeland's stability scattering diagram \cite{bridgeland2016scattering} of the algebra $\mathcal P(\kappa)$, which in the case without orbifold points are shown in \cite{mou2019scattering} to have the same cluster structure as the \emph{cluster scattering diagram} of Gross, Hacking, Keel and Kontsevich \cite{gross2018canonical}. We study the following function
\[
\Phi_\kappa \colon \mathbb R^n \rightarrow \mathcal W(\mathcal P(\kappa)),\quad \theta\mapsto \moduss{\theta} \mathcal P(\kappa),
\]
which maps a King stability condition $\theta$ to the wide subcategory of $\theta$-semistable modules over $\mathcal P(\kappa)$. The structure of the level sets of this function reveals the cluster structure of $\mathcal A(\kappa)$. In particular, the so-called $\mathbf g$-vector fan of $\mathcal A(\kappa)$ (whose rays are in bijection with cluster variables) can be seen from the support of $\Phi_\kappa$; see \Cref{thm: g-vectors are rays of stability chambers}.

We develop BGP-type \cite{bernstein1973coxeter} reflection functors $F_k^\pm \colon \modu \mathcal P(\kappa)\rightarrow \modu \mathcal P(\mu_k(\kappa))$ to relate $\Phi_\kappa$ with $\Phi_{\mu_k(\kappa)}$ where $k$ does not have to be a sink or source. This key tool enables explicit investigations on $\Phi_\kappa$ as well as any $\Phi_\sigma$ for $\sigma$ obtained by a sequence of flips from $\kappa$. In the case without orbifold points, these functors can be viewed as the degree zero cohomology of Keller--Yang's derived equivalences \cite{keller2011derived} which so far are only available for quivers without loops. Via the correspondence between $\tau$-tilting theory and stability conditions of Br\"ustle--Smith--Treffinger \cite{brustle2019wall}, the $\tau$-tilting theory of $\mathcal P(\kappa)$ can be related to the cluster structure of $\mathcal A(\kappa)$ through $\Phi_\kappa$. Our main result is the following.

\begin{theorem}[{\Cref{thm: cc function}, \Cref{thm: cc function induces bijection and isomorphism}}]\label{thm: main theorem}
Sending a $\tau$-rigid pair $\mathcal M$ to its Caldero--Chapoton function $CC(\mathcal M)$ gives a bijection between 
\[
\{\text{reachable indecomposable $\tau$-rigid pairs of $\mathcal P(\kappa)$}\} \longleftrightarrow \{\text{cluster variables of $\mathcal A(\kappa)$}\}
\]
which induces an isomorphism
\[
\mathbf E(\rstau{\mathcal P(\kappa)}) \cong \mathbf E(\mathcal A(\kappa))
\]
between the mutation graph of reachable support $\tau$-tilting pairs of $\modu \mathcal P(\kappa)$ and the exchange graph of unlabeled seeds of $\mathcal A(\kappa)$.
\end{theorem}

\begin{remark}
Sota Asai has informed us that Fu--Geng--Liu--Zhou \cite{fu2021support} recently proved that for all finite-dimensional gentle algebras, every indecomposable $\tau$-rigid pair is reachable. Thus, the adjective ``reachable'' can be removed from the statement of \Cref{thm: main theorem} above.
\end{remark}

\subsection{Motivic Hall algebras and Caldero--Chapoton functions}
To prove \Cref{thm: main theorem}, we make heavy use of Bridgeland's Hall algebra scattering diagram as a function $\phi_\kappa^{\mathrm{Hall}}\colon \mathbb R^n \rightarrow \hat H(\modu \mathcal P(\kappa))$. It can be seen as turning the category $\Phi_\kappa(\theta) = \moduss{\theta} \mathcal P(\kappa) $ into an invertible element $\phi_\kappa^{\mathrm{Hall}}(\theta) = 1_{\mathrm{ss}}(\theta)$ in the motivic Hall algebra $\hat H(\modu \mathcal P(\kappa))$; see \Cref{section: scattering diagram}. Another important ingredient is Joyce's integration map (\Cref{thm: integration map}) which turns $1_{\mathrm{ss}}(\theta)$ further into a formal exponential of a Poisson derivation. The integration map exists simply because our algebra $\mathcal P(\kappa)$ is the Jacobian algebra of a quiver with potential. After applying the integration map, we obtain a function $\phi_\kappa\colon \mathbb R^n \rightarrow \hat G$ with values in Poisson automorphisms some of which precisely describe generalized cluster transformations. The function $\phi_\kappa$ is what we call the stability scattering diagram of $\mathcal P(\kappa)$.

The above machinery makes it practical to express a non-initial cluster variable $x_{i;t} \in \mathcal A(\kappa)$ in terms of representation theoretic information, in fact as the Caldero--Chapoton function of some indecomposable $\tau$-rigid module $M_{i;t}$. Our calculation is largely motivated by Nagao's work \cite{nagao2013donaldson} in the skew-symmetric case of ordinary cluster algebras. However the characterizations of the representation theoretic object corresponding to a cluster variable are very different as we rely on $\tau$-tilting theory \cite{adachi2014tau} rather than Keller--Yang's derived equivalences \cite{keller2011derived}.

\subsection{Structure of the paper}
In \Cref{section: generalized cluster algebra}, we review the definitions of cluster algebras and their generalizations. In \Cref{section: surfaces} and \Cref{section: gentle algebra}, we explain the construction of gentle algebras $\mathcal P(\kappa)$ associated to surfaces with orbifold points. In \Cref{section: reflection functor} and \Cref{section: stability}, we define reflection functors and use them to study stability conditions of $\mathcal P(\kappa)$, realizing cluster chamber structures. In \Cref{section: hall algebras}, we review the motivic Hall algebras of quivers with relations. In \Cref{section: scattering diagram}, \Cref{section: tau tilting}, and \Cref{section: caldero chapoton}, we study the scattering diagrams associated to $\mathcal P(\kappa)$, relate them with the $\tau$-tilting theory of $\modu \mathcal P(\kappa)$ and prove our main theorem \Cref{thm: main theorem}. In Section \ref{section:example} we present in an explicit example in affine type $\widetilde{C}_2$, how the Caldero-Chapoton functions of support $\tau$-tilting pairs related by an AIR-mutation satisfy a generalized exchange relation as stated in Theorem \ref{thm: main theorem} (see Theorem \ref{thm: cc function induces bijection and isomorphism} too). In the final \Cref{section: final}, we summarize and outline the plan of the sequel \cite{LMpart2}.

\section*{Acknowledgements}

DLF received support from UNAM's \emph{Dirección General de Asuntos del Personal Académico} through its \emph{Programa de Apoyos para la Superación del Personal Académico}, and from  UNAM's Physics Institute via a \emph{Cátedra Marcos Moshinsky}. LM~is supported by the Royal Society through the Newton International Fellowship NIF\textbackslash R1\textbackslash 201849.

The last part of the paper was completed during a visit of both authors to Sibylle Schroll at the Mathematisches Institut of the Universität zu Köln. We are grateful for the great working conditions and the hospitality. 

The authors would like to thank the Isaac Newton Institute for Mathematical
Sciences, Cambridge, for support and hospitality during the programme Cluster
Algebras and Representation Theory where work on this paper was undertaken. This
work was supported by EPSRC grant no EP/K032208/1. We thank the Organizers of this programme, Karin Baur, Bethany Marsh, Ralf Schiffler and Sibylle Schroll for a wonderful working atmosphere.

Part of this project was carried out during the junior trimester program
New Trends in Representation Theory, held at the Hausdorff Research Institute for Mathematics, Bonn, and organized by Gustavo Jasso. We are grateful for the support and hospitality received. 

\section{Generalized cluster algebras}\label{section: generalized cluster algebra}
Here, we review Chekhov--Shapiro's definition of generalized cluster algebras \cite{chekhov2011teichmller}.
We only introduce the coefficient-free case which will suffice for our purposes.

\begin{definition} Let $\mathcal{F}$ be the field of rational functions in $n$ indeterminates with coefficients in $\mathbb{C}$.
	A~\emph{labeled seed} in $\mathcal{F}$ is a pair  $(\mathbf x, B)$ consisting of:
	\begin{itemize}
		\item an ordered $n$-tuple $\mathbf x = (x_1, \dots, x_n)$ of elements of $\mathcal{F}$ algebraically independent over $\mathbb{C}$ that generate $\mathcal{F}$ as a field;
		\item a skew-symmetrizable matrix $B\in \operatorname{Mat}_{n\times n}(\mathbb Z)$, i.e.,  a matrix for which there exists a diagonal matrix $D = \operatorname{diag}(d_1,\ldots,d_n)\in \operatorname{Mat}_{n\times n}(\mathbb Z)$ with positive diagonal entries such that $DB + (DB)^T = 0$.
	\end{itemize}
	The tuple $\mathbf{x}$ is called the \emph{labeled cluster} of the labeled seed $(\mathbf{x},B)$, the elements $x_1,\ldots,x_n$, are the \emph{cluster variables} of the labeled seeds.
\end{definition}

To define Chekhov--Shapiro's generalized mutation rule of labeled seeds, we first fix an $n$-tuple $(r_1, \dots, r_n)$ of positive integers such that for each $j$, the integer $r_j$ divides the $j$-th column of $B$, i.e., $b_{ij}/r_j\in \mathbb Z$ for any $i$ and $j$. This property of $(r_1, \dots, r_n)$ is preserved under Fomin--Zelevinsky's matrix mutation. Denote the matrix $\overline{B} = (\overline{b}_{ij})\coloneqq (b_{ij}/r_j)$. For each $i$ fix also a polynomial
\[
\theta_i(u,v) = \sum_{\ell = 0}^{r_i} c_{i,\ell} u^{\ell} v^{r_i - \ell} \in \mathbb C[u,v]
\]
which is palindromic ($\theta_i(u,v) = \theta_i(v,u)$) and monic $c_{i,0} = c_{i,r_i} = 1$.

\begin{definition}\label{def:generalized-labeled-seed-mutation} For each $k\in\{1,\ldots,n\}$, the \emph{generalized mutation} of $(\mathbf{x},B)$ in direction $k$ with respect to $(r_1, \dots, r_n)$ and $(\theta_1,\ldots,\theta_n)$ is the labeled seed $\mu_k(\mathbf{x},B)=((x_1,\ldots,x_{k-1},x_k',x_{k+1},\ldots,x_n),\mu_k(B))$, where
$\mu_k(B)$ is the $k^{\operatorname{th}}$ Fomin--Zelevinsky matrix mutation of $B$, and
\begin{equation}\label{eq:generalized-cluster-mutation}
x_kx'_k=\sum_{\ell=0}^{r_k}c_{k,\ell}\left(\prod_{j:b_{jk}>0} x_j^{b_{jk}/r_k} \right)^{\ell}\left(\prod_{j:b_{jk}<0} x_j^{-b_{jk}/r_k}\right)^{r_k-\ell}
\end{equation}
\end{definition}

Notice that the right hand side of \eqref{eq:generalized-cluster-mutation} is the evaluation $\theta_k\left(\prod_{j=1}^n x_j^{[\overline{b}_{jk}]_+},\prod_{j=1}^n x_j^{[-\overline{b}_{jk}]_+} \right)$.

\begin{definition} The (coefficient-free) \emph{generalized cluster algebra} associated to the \emph{initial seed} $(\mathbf{x},B)$ with respect to the fixed data $(r_1, \dots, r_n)$ and $(\theta_1,\ldots,\theta_n)$ is the subring $\mathcal{A}(\mathbf{x},B)$ of $\mathcal{F}$ generated by the set of all cluster variables that appear in the labeled seeds that can be obtained from $(\mathbf{x},B)$ by applying arbitrary finite sequences of generalized cluster mutations.
\end{definition}

If one takes $r_1=\ldots=r_n=1$ and $\theta_1(u,v)=\ldots=\theta_n(u,v)=u+v$, then \eqref{eq:generalized-cluster-mutation} gives the usual cluster mutation of Fomin--Zelevinsky \cite{fomin2002cluster} and $\mathcal{A}(\mathbf{x},B)$ is the usual (coefficient-free) cluster algebra.

We say that two labeled seeds $(\mathbf x', B')$ and $(\mathbf x, B)$ are \emph{mutation equivalent}, and write $(\mathbf x', B') \sim (\mathbf x, B)$, if they can be obtained by a finite sequence of mutations from each other.

In \cite{fomin2007cluster}, Fomin and Zelevinsky introduce the notions of \emph{$\mathbf{g}$-vector} and \emph{$F$-polynomial} in cluster algebras. Nakanishi \cite{nakanishi2015structure} has extended these notions to generalized cluster algebras. Let $\mathbb{T}_n$ be the rooted $n$-regular tree. We denote its root as $t_0$, and assume that each edge of $\mathbb{T}_n$ has been labeled with a number from $\{1,\ldots,n\}$, in such a way that whenever two distinct edges are incident at a vertex of $\mathbb{T}_n$, their labels are different. This allows us to assign to each vertex $t\in\mathbb{T}_n$ a labeled seed $((x_{1;t}^{B;t_0},\ldots,x_{n;t}^{B;t_0}),B(t))$ in such a way that $((x_{1;t_0}^{B;t_0},\ldots,x_{n;t_0}^{B;t_0}),B(t_0))\coloneqq(\mathbf{x},B)$, and for each edge $t\frac{k}{\qquad}t'$, the labeled seeds attached to $t$ and $t'$ are related by the $k^{\operatorname{th}}$ generalized mutation.

Taking $(r_1,\ldots,r_n)$ and $(\theta_1,\ldots,\theta_n)$ as fully fixed data, and $B$ as our initial matrix (attached to the root $t_0$ of $\mathbb{T}_n$), to each vertex $t$ of $\mathbb{T}_n$ we associate $n$-tuples $(\mathbf{c}_{1;t}^{B;t_0},\ldots,\mathbf{c}_{n;t}^{B;t_0})$, $(\mathbf{g}_{1;t}^{B;t_0},\ldots,\mathbf{g}_{n;t}^{B;t_0})$, of vectors in $\mathbb{Z}^n$, and an $n$-tuple $(F_{1;t}^{B;t_0},\ldots,F_{n;t}^{B;t_0})$ of elements of the rational function field $\mathbb{Q}(y_1,\ldots,y_n)$. Whenever it is clear which matrix $B$ is being attached to $t_0$, we will take the liberty of possibly omitting the superscript $B;t_0$ from the notation. These tuples are defined recursively as follows. For each $i=1,\ldots,n$,
\begin{center}
\begin{tabular}{cccc}
$\mathbf{c}_{i;t_0}^{B;t_0}\coloneqq\mathbf{e}_i$ &  $\mathbf{g}_{i;t_0}^{B;t_0}\coloneqq\mathbf{e}_i$ & and & $F_{i;t_0}^{B;t_0}\coloneqq 1$
\end{tabular}
\end{center}
($\mathbf{e}_i$ is the $i^{\operatorname{th}}$ standard basis vector). Furthermore, whenever we have an edge $t\frac{k}{\quad\quad}t'$ in $\mathbb{T}_n$, we have
\[
\mathbf c^{B;t_0}_{i;t'} = \begin{cases}
-\mathbf c^{B;t_0}_{i;t} \quad & i = k\\
\mathbf c^{B;t_0}_{i;t} + (r_k/r_i)b_{ki}(t) [-\mathbf c^{B;t_0}_{k;t}]_+ \quad &i\neq k,\ b^\mathbf k_{ik}>0\\
\mathbf c^{B;t_0}_{i;t} + (r_k/r_i)b_{ki}(t) [+\mathbf c^{B;t_0}_{k;t}]_+ \quad &i\neq k,\ b^\mathbf k_{ik}\leq 0\\
\end{cases}
\]
\[
\mathbf g^{B;t_0}_{i;t'} = \begin{cases}
\mathbf g^{B;t_0}_{i;t} \quad & i \neq k\\
-\mathbf g^{B;t_0}_{k;t_0} + \sum\limits_{j = 1}^n[b_{jk}(t)]_+ \mathbf g^{B;t_0}_{j;t} - \sum\limits_{j =1}^n (r_k/r_j)[c^{B;t_0}_{k,j;t}]_+ \mathbf b_{j}(t) \quad & i = k
\end{cases}
\]
and\[
F_{k;t'}\cdot F_{k;t} = \theta_k\left(\prod_i y_i^{[c^{B,t_0}_{k,i;t}]_+} F_{i;t}^{[\overline{b}_{ik}(t)]_+},\ \prod_i y_i^{[-c^{B;t_0}_{k,i;t}]_+} F_{i;t}^{[-\overline{b}_{ik}(t)]_+}  \right).
\]
where for $t\in\mathbb{T}_n$, $b_{ki}(t)$ is the $(k,i)$-th entry of the matrix $B(t)$, $\mathbf b_j(t)$ denotes the $j$-th column of $B(t)$, and $c^{B;t_0}_{k,j;t}$ is the $j^{\operatorname{th}}$ entry of $\mathbf c^{B;t_0}_{k;t}$.

For the next theorem, we set
\begin{equation}\label{eq:generalized-yhat}
\hat y_j \coloneqq \prod_{i = 1}^n x_i^{\overline{b}_{ij}}
\qquad \text{and} \qquad
x^{\mathbf g_{j;t}} = \prod_{i = 1}^n x_i^{g_{j,i;t}^{B;t_0}}
\end{equation}

\begin{theorem}[\cite{nakanishi2015structure}]\label{thm:F-pols-are-pols-and-expansion-of-gen-cluster-vars}
For every vertex $t\in\mathbb{T}_n$ and every $j=1,\ldots,n$, the rational function $F_{j;t}^{B;t_0}$ is actually a polynomial in $y_1,\ldots,y_n$, with integer coefficients, and the generalized cluster variable $x_{j;t}^{B;t_0}$ can be written as
\begin{equation}\label{eq:generalized-cluster variable by g-vector and f-poly}
x_{j;t}^{B;t_0} = x^{\mathbf g_{j;t}} \cdot F_{j;t}^{B;t_0} (\hat y_1, \dots, \hat y_n).
\end{equation}
\end{theorem}

\begin{remark}\ 
\begin{enumerate}
    \item For cluster algebras, that is, for the choice $r_1=\ldots=r_n=1$ and $\theta_1(u,v)=\ldots=\theta_n(u,v)=u+v$, Theorem \ref{thm:F-pols-are-pols-and-expansion-of-gen-cluster-vars} was first proved by Fomin--Zelevinsky in \cite{fomin2007cluster}.
    \item In \cite{fomin2007cluster} and \cite{nakanishi2015structure}, the definition of $\mathbf{g}$-vectors and $F$-polynomials is given rather in terms of their properties in a (generalized) cluster algebra with principal coefficients. Under this approach, the fact that they can be obtained through recursions becomes a theorem instead of a definition.
\end{enumerate}
\end{remark}

For a vertex $t\in\mathbb{T}_n$ and a permutation $\sigma$ of the index set $\{1,\ldots,n\}$, set
\begin{equation}\label{eq:permutation-inside-a-labeled-seed}
\sigma\cdot((x_{1;t}^{B;t_0},\ldots,x_{n;t}^{B;t_0}),B(t))\coloneqq ((x_{\sigma^{-1}(1);t}^{B;t_0},\ldots,x_{\sigma^{-1}(n);t}^{B;t_0}),\sigma\cdot (B(t))),
\end{equation}
where the $(i,j)^{\operatorname{th}}$ entry of the matrix $\sigma\cdot(B(t))\in\mathbb{Z}^{n\times n}$ is defined to be $b_{\sigma^{-1}(i)\sigma^{-1}(j)}(t)$.
If $t$ and $s$ are vertices of $\mathbb{T}_n$ such that $(\mathbf{x}_{s}^{B;t_0},B(s))=\sigma\cdot(\mathbf{x}_t^{B;t_0},B(t))$, then
for every $k\in\{1,\ldots,n\}$ we have the equality of labeled seeds
$$
\mu_{\sigma(k)}(\mathbf{x}_s^{B;t_0},B(s))=\sigma\cdot\mu_{k}(\mathbf{x}_t^{B;t_0},B(t)),
$$
which means that the automorphism of $\mathbb{T}_n$ that sends $t$ to $s$
and the edge $t\frac{k}{\qquad}\star $ to the edge $s\frac{\sigma(k)}{\qquad}*$ is compatible with mutations of labeled seeds. Let $G(B)$ be the group of graph automorphisms of $\mathbb{T}_n$ that arise this way.

\begin{definition}\label{def:exhchange-graph-unlabeled-seeds}
The \emph{exchange graph of unlabeled seeds of the generalized cluster algebra $\mathcal A = \mathcal{A}(\mathbf{x},B)$} is the simple graph $\mathbf{E}(\mathcal A)$ obtained as the quotient $\mathbb{T}_n/G(B)$.
\end{definition}

\section{Surfaces with orbifold points of order 3}\label{section: surfaces}
Chekhov--Shapiro \cite{chekhov2011teichmller} have shown that surfaces with marked points and orbifold points give rise to generalize cluster algebras. In this paper we will restrict our attention to the situation where all marked points are contained in the boundary of the surface and all orbifold points have order 3.

\begin{definition}\label{def:unpunct-surf-with-orb-pts-order3}
	An \emph{unpunctured surface with marked points and orbifold points of order 3} is a triple $\surf$ consisting of:
	\begin{itemize}
		\item a compact connected oriented two-dimensional real differentiable manifold $\Sigma$ with non-empty boundary $\partial \Sigma$;
		\item a finite subset $\mathbb M\subset \partial \Sigma$ containing at least one element from each connected component of $\partial \Sigma$;
		\item a finite subset $\mathbb O\subset \Sigma\setminus \partial \Sigma$.
	\end{itemize}
	The elements of $\mathbb{M}$ will be called \emph{marked points}, whereas the elements of $\mathbb{O}$ will be called \emph{orbifold points} (of order 3).
\end{definition}

\begin{definition}
	An arc $i$ on $\surf$ is a curve $i\colon [0,1]\rightarrow \Sigma\setminus \mathbb O$ such that
	\begin{itemize}
		\item the image $i([0,1])$ must and only intersect the boundary $\partial \Sigma$ at its endpoints $i(0)$ and $i(1)$;
		\item $i$ has no self intersections except possibly at endpoints;
		\item if $i$ cuts a monogon, then such a monogon contains exactly one orbifold point. In this case, we call $i$ a \emph{pending arc}.
	\end{itemize} 
\end{definition}

We consider arcs up to isotopy relative to $\mathbb{M}\cup\mathbb{O}$. Two arcs are called \emph{compatible} if there are representatives in their isotopy classes whose images as curves do not intersect in $\Sigma\setminus\mathbb{M}$. Whenever two arcs are compatible, we will directly take the corresponding representatives.

\begin{definition}
	A \emph{triangulation} of $\surf$ is a maximal collection of pairwise compatible arcs. If $\kappa$ is a triangulation of $\surf$, we define the \emph{flip} at any $i\in \kappa$ to be the unique triangulation 
	\[
	\mu_i(\kappa) \coloneqq (\kappa \setminus \{i \}) \cup \{i'\}
	\]
	with the unique arc $i'\neq i$.
\end{definition}

It is easy to see from the definition that flips are involutive.

\begin{definition}
The \emph{flip graph} $\mathbf{E}\surf$ has the triangulations of $\surf$ as vertices; two triangulations are connected by an edge precisely when they are related by a flip.
\end{definition}

For the next definition, consider the subset of $\Sigma$ obtained as the union of the arcs in a triangulation~$\kappa$. The complement of this union in $\Sigma$ is a disjoint union of connected components.

\begin{definition}\label{def:triangles-of-triangulation}
Suppose $\kappa$ is a triangulation of $\surf$.
\begin{enumerate}
    \item The topological closure in $\Sigma$ of any of the connected components mentioned in the preceding paragraph will be called a \emph{triangle of $\kappa$};
    \item a triangle not containing orbifold points will be called \emph{non-singular triangle};
    \item a triangle containing an orbifold point will be called \emph{singular triangle}.
\end{enumerate}
\end{definition}

Thus, a non-singular triangle always has three distinct \emph{sides}, and a singular triangle always has exactly one \emph{side}. See Figure \ref{Fig_nonSing_and_sing_triangles}.
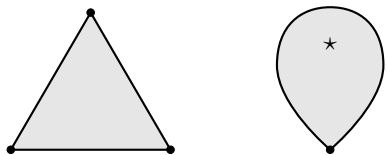
\begin{figure}[!h]
                \centering
                \begin{tikzpicture}[scale=0.70]

\draw[fill=gray!20, thick] (0, 0) -- (3, 0) -- (1.5, 2.598) -- cycle;
\filldraw[black] (0, 0) circle (2pt);
\filldraw[black] (3, 0) circle (2pt);
\filldraw[black] (1.5, 2.598) circle (2pt);

\draw [fill = gray!20, thick] plot [smooth, tension = 1] coordinates { (6,0) (5,1.6) (6,2.7) (7,1.6) (6,0)};
\filldraw[black] (6, 0) circle (2pt);
\node[] at (6,2) {$\star$};

\end{tikzpicture}
                \caption{The two types of triangles of a triangulation. Left: a non-singular triangle. Right: a singular triangle.}\label{Fig_nonSing_and_sing_triangles}
        \end{figure}
Each \emph{side} of a non-singular triangle is either an arc (even possibly a pending one) or a segment of the boundary of $\Sigma$. Furthermore, at most two sides of a non-singular triangle can be pending arcs. The unique side of a singular triangle is pending.

Recall that a \emph{quiver} is a quadruple $Q = (Q_0, Q_1, h, t)$ consisting of 
\begin{itemize}
    \item a set $Q_0$ of \emph{vertices};
    \item a set $Q_1$ of \emph{arrows};
    \item a map $h\colon Q_1\rightarrow Q_0$ giving the \emph{head} of each arrow;
    \item a map $t\colon Q_1\rightarrow Q_0$ giving the \emph{tail} of each arrow.
\end{itemize}
Recall also that we draw any $a\in Q_1$ as $t(a)\xrightarrow{a}h(a)$ or $a\colon t(a) \rightarrow h(a)$.

\begin{definition} 
 Following \cite{assem2010gentle} and \cite[Definition 3.2]{Geuenich-LF2}, we define $\overline{Q}(\kappa)$ to be the quiver whose vertices are the arcs in $\kappa$, that is, $\overline{Q}_0(\kappa)\coloneqq \kappa$, and whose arrows of are induced by the non-singular triangles of $\kappa$ through the orientation of $\Sigma$. More precisely, for each non-singular triangle $\triangle$ of $\kappa$ and every pair $i, j \in\kappa$ of arcs contained in $\triangle$ such that $i$ immediately follows $j$ in $\triangle$ with respect to the clockwise sense in $\triangle$ induced by the orientation of $\Sigma$, we draw a single arrow from $j$ to $i$.
\end{definition}

Notice that we are not including the boundary segments in the quiver $\overline{Q}(\kappa)$.

Take a triangulation $\kappa$, and fix a labeling of the arcs in $\kappa$ by the numbers $1,\ldots,n \coloneqq |\kappa|$. We say then, that $\kappa$ is a \emph{labeled triangulation}.
We associate to $\kappa$ a skew-symmetrizable matrix $B(\kappa)$ as follows.
The quiver $\overline{Q}(\kappa)$ has an associated skew-symmetric matrix $\overline B = (\overline{b}_{ij})_{(i,j)\in\kappa\times\kappa}$ whose entries are given by the rule
\[
\overline{b}_{ij} \coloneqq |\{a\in \overline Q_1(\kappa) \mid \text{$t(a) = i$ and $h(a) = j$}\}| - |\{a\in \overline Q_1(\kappa) \mid \text{$t(a) = j$ and $h(a) = i$}\}|.
\]
Consider the matrix $D = \mathrm{diag}(d_i\suchthat i\in\kappa)$ defined by
$$
d_i \coloneqq \begin{cases} 2 & \text{if $i$ is a pending arc};\\ 
1 & \text{otherwise}.
\end{cases}
$$
We define $B(\kappa) = (b_{ij}(\kappa)) \coloneqq \overline B D$. Notice that $DB(\kappa)$ is skew-symmetric, thus $B(\kappa)$ is skew-symmetrizable. The following lemma is easy to check. 
\begin{lemma}
If $\kappa$ and $\sigma$ are two triangulations, and $\sigma$ is the flip $\mu_i(\kappa)$ of $\kappa$ at $i\in \kappa$, then we have
\[
\mu_i(B(\kappa)) = B(\sigma).
\]
\end{lemma}

We associate a generalized cluster algebra to $B(\kappa)$ by taking
\begin{equation}\label{eq:gen-mut-rul-in-our-particular-setting}
    \begin{tabular}{cccc}
    $r_j=d_j$ & and & $\theta_j=\sum_{\ell=0}^{r_j}u^\ell v^{r_j-\ell}$ & for every $j=1,\ldots,n$.
\end{tabular}
\end{equation}
By construction, for each $j=1,\ldots,n$,  $r_j$  divides all entries of the  $j^{\operatorname{th}}$ column of $B(\kappa)$.

Combining \cite{felikson2012cluster} and \cite{chekhov2011teichmller}, we have

\begin{theorem}\label{thm: FST and CS}
Let $\surf$ be as in Definition \ref{def:unpunct-surf-with-orb-pts-order3} and let $\kappa$ be a triangulation of $\surf$. Denote by $\mathcal{A}(\kappa)$ the generalized cluster algebra $\mathcal{A}(\mathbf{x},B(\kappa))$ constructed with respect to the fixed data \eqref{eq:gen-mut-rul-in-our-particular-setting}.
	There is a bijection
	\[
	\{\text{arcs of $\surf$}\} \longleftrightarrow \{ \text{cluster variables of $\mathcal A(\kappa)$}\}
	\]
	under which the arcs in $\kappa$ correspond to the initial cluster variables. Furthermore, this bijection induces a graph isomorphism
	\[
	\mathbf{E}\surf \longleftrightarrow \mathbf{E}(\mathcal A(\kappa))
	\]
	between the flip graph of $\surf$ and the exchange graph of unlabeled seeds of $\mathcal{A}(\kappa)$, making the triangulation $\kappa$ correspond to the seed $(\mathbf{x},B(\kappa))$. Thus, denoting by $x_j$ the cluster variable corresponding to each arc $j$ on $\surf$, for every triangulation $\sigma$ and every arc $k\in\sigma$, if $k'$ denotes the unique arc on $\surf$ such that $(\sigma\setminus\{k\})\cup\{k'\}$ is the flip $\mu_k(\sigma)$ at $k\in \sigma$, then the cluster variables $x_j\in\mathcal{A}(\kappa)$, $j\in\sigma\cup\{k'\}$, satisfy
	\begin{equation}\label{def:exchange-formula-for-cluster-vars-in-terms-of-arcs}
	    x_kx_{k'}=\begin{cases}
	    \left(\underset{j:b_{jk}(\sigma)<0}{\prod} x_j^{-b_{jk}(\sigma)}\right)+\left(\underset{j:b_{jk}(\sigma)>0}{\prod} x_j^{b_{jk}(\sigma)} \right) & \text{if $k$ is not pending;}\\
	    \left(\underset{j:b_{jk}(\sigma)<0}{\prod} x_j^{2}\right)
	    +\left(\underset{j:b_{jk}(\sigma)>0}{\prod} x_j \right)\left(\underset{j:b_{jk}(\sigma)<0}{\prod} x_j\right)
	    +\left(\underset{j:b_{jk}(\sigma)>0}{\prod} x_j^{2} \right) 
	    & \text{if $k$ is pending}.
	    \end{cases}
	\end{equation}
	\end{theorem}

\begin{remark}
The role that the orbifold points play in this paper is overwhelmingly combinatorial. They play a geometric role, although very subtly, as follows. Chekhov--Shapiro \cite{chekhov2011teichmller} have shown that  within the ring of real-valued functions on the decorated Teichm\"{u}ller space of a surface with marked points and orbifold points of arbitrary orders, the \emph{$\lambda$-lengths} of arcs satisfy an exchange rule \eqref{eq:generalized-cluster-mutation}, hence the subring generated by them is (isomorphic to) a generalized cluster algebra (with so-called \emph{boundary coefficients}). When the orders of the orbifold points are all taken to be equal to three, the exchange relation \eqref{eq:generalized-cluster-mutation} obeyed by the $\lambda$-lengths takes the specific form \eqref{def:exchange-formula-for-cluster-vars-in-terms-of-arcs}. One of our main results in this paper, is that it is this specific form that the Caldero--Chapoton functions of indecomposable $\tau$-rigid pairs satisfy whenever one mutates a basic support $\tau$-tilting pair.
\end{remark}

\section{The gentle algebra associated to a triangulation}\label{section: gentle algebra}

Let $Q$ be a quiver. Denote by $\Bbbk \langle Q \rangle$ the path algebra of $Q$ over a field $\Bbbk$. Denote by $e_i$ the primitive idempotent of vertex $i\in Q_0$. A \emph{potential} $S$ of $Q$ is an element in $\Bbbk \langle Q \rangle/[\Bbbk \langle Q \rangle, \Bbbk \langle Q \rangle]$. So it can be viewed as a linear sum of oriented cyclic paths (called \emph{cycles}) in $Q$ and each cycle is considered up to cyclic permutations.

Assume that $\Bbbk$ is algebraically closed of characteristic zero. We define the \emph{Jacobian ideal} $\langle \partial S \rangle$ to be the (two-sided) ideal in $\Bbbk \langle Q \rangle$ generated by all cyclic derivatives $\partial S \coloneqq \{\partial_a S \mid a\in Q_1\}$ (see \cite[Definition 3.1]{derksen2008quivers}).

\begin{definition}\label{def: quiver with potential of a triangulation}
Let $\mathbb S = \surf$ and $\kappa$ be an ideal triangulation of $\mathbb S$. The quiver $Q(\kappa)$ is defined by adding an arrow $\varepsilon_j \colon j\rightarrow j$ (also called a \emph{loop}) for each pending arc $j\in\kappa$. We define a potential $S(\kappa)\in\KQ{Q(\kappa)}/[\KQ{Q(\kappa)},\KQ{Q(\kappa)}]$ according to the formula
$$
S(\kappa)\coloneqq\sum_{\triangle}\alpha_{\triangle}\beta_{\triangle}\gamma_{\triangle}+\sum_{j \ \text{pending}}\varepsilon_{j}^3,
$$
where the first sum runs over all internal non-singular triangles of $\kappa$ and the second sum runs over all pending arcs of $\kappa$.
\end{definition}

\begin{lemma}\label{lemma: fd and gentle}
The Jacobian algebra $\mathcal P(\kappa) \coloneqq \mathcal P(Q(\kappa), S(\kappa)) \coloneqq \KQ{Q(\kappa)}/ \langle \partial S(\kappa) \rangle$ is finite dimensional and gentle.
\end{lemma}

\begin{proof}
The cyclic derivatives are either of the form $3 \varepsilon_j^2$ for $j$ pending or $a\cdot b$ for $a$ and $b$ two arrows of $Q(\kappa)$ in the same internal triangle of $\kappa$. It is then clear that there exists some $N\in \mathbb N$ such that every path of length greater than $N$ vanishes modulo the relations generated by the cyclic derivatives. Thus the algebra $\mathcal P(\kappa)$ is of finite dimension over $\Bbbk$.

Recall that an algebra $\Bbbk \langle Q \rangle/I$ is called \emph{gentle} if
\begin{enumerate}
    \item each vertex of $Q$ is incident with at most two incoming and at most two outgoing arrows;
    \item the ideal $I$ is generated by paths of length two;
    \item for each $b\in Q_1$ there is at most one $a\in Q_1$ and at most one $c\in Q_1$ such that $ab\in I$ and $bc\in I$;
    \item for each $b\in Q_1$ there is at most one $a\in Q_1$ and at most one $c\in Q_1$ such that $ab\notin I$ and $bc\notin I$.
\end{enumerate}
The conditions (1) and (2) hold for $Q(\kappa)$ and $I = \partial S(\kappa)$ from the constructions of the quiver and the potential. Conditions (3) and (4) are easily checked by studying all possible local configurations of a triangulation.
\end{proof}

For such a Jacobian algebra $\mathcal P(\kappa)$, to each vertex $k\in \kappa$, we associate the subalgebra $H_k \coloneqq \Bbbk \langle \varepsilon_k, e_k \rangle$ (thus isomorphic to $\Bbbk[\varepsilon]/(\varepsilon^2)$) if $k$ is pending and otherwise $H_k \coloneqq \Bbbk e_k$. 

\begin{remark}
\begin{enumerate}
\item In the absence of orbifold points, the gentle algebras $\mathcal{P}(\kappa)$ were introduced in \cite{assem2010gentle} and \cite{labardini2009quivers}.
\item Our definition of the algebras $\mathcal{P}(\kappa)$ in terms of triangulations of $\surf$ can be seen as a special case of the much more general geometric construction of gentle algebras in terms of dissections of surfaces. See, for instance, \cite{opper2018geometric}.
\item The canonical inclusion of $\Bbbk\langle Q(\kappa)\rangle$ in the complete path algebra $\Bbbk\langle\langle Q(\kappa)\rangle\rangle$ induces an algebra isomorphism between $\mathcal{P}(\kappa)$ and the quotient of $\Bbbk\langle\langle Q(\kappa)\rangle\rangle$ by the $\mathfrak{m}$-adic topological closure of the two-sided ideal generated by $\partial S(\kappa)$.
\end{enumerate}
\end{remark}

\section{Reflection functors}\label{section: reflection functor}
In this section, we define BGP-type reflection functors \cite{bernstein1973coxeter} between module categories of the gentle algebras associated to triangulations related by a flip.

Denote by $\modu A$ the category of finitely generated left modules of a $\Bbbk$-algebra $A$. Let $\kappa$ and $\sigma$ be two labeled triangulations related by a flip at $k\in \kappa$. We will define two functors
\[
F^\pm_k \colon \modu \mathcal P(\kappa) \rightarrow \modu \mathcal P(\sigma).
\]
As $\mathcal P(\kappa) = \Bbbk \langle Q(\kappa) \rangle /\langle \partial S(\kappa) \rangle$ is finite dimensional (\Cref{lemma: fd and gentle}), the category $\modu \mathcal P(\kappa)$ is well-known to be equivalent to $\operatorname{rep}(Q(\kappa), \partial S(\kappa))$, the category of finite dimensional $\Bbbk$-linear representations of $Q(\kappa)$ satisfying relations in $\partial S(\kappa)$. Therefore in the following we treat any $\mathcal P(\kappa)$-module equivalently as an object in $\operatorname{rep}(Q(\kappa), \partial S(\kappa))$. For a representation $M$ of a quiver $Q$, denote by $M_i$ the vector space associated to $i\in Q_0$ and by $M_a \colon M_{t(a)} \rightarrow M_{h(a)}$ the linear map associated to $a\in Q_1$. All tensor products $\otimes$ in this section will be of vector spaces over $\Bbbk$.

Flipping $\kappa$ at $k$ will result an arrow $a^*\colon k\rightarrow i$ in $\overline Q_1(\sigma)$ for any $a\colon i\rightarrow k$ in $\overline Q_1(\kappa)$, similarly an arrow $b^*\colon j\rightarrow k$ for any $b\colon k\rightarrow j$. As we have $\mu_k^2(\kappa) = \kappa$, we identify $a^{**}$ with $a$ for any $a$ incident to $k$ in $\overline Q_1(\kappa)$. See the diagrams in Case 1 and Case 2 in \Cref{section: F_k^+}.

\subsection{Sink reflection $F^+_k$}\label{section: F_k^+} Let $M$ be in $\modu \mathcal P(\kappa)$. Consider the $\Bbbk$-linear map
\[
\alpha \colon  \bigoplus_{\substack{a\in \overline Q_1(\kappa)\\t(a) = i,\, h(a)=k}} H_k \otimes M_i  \rightarrow M_k
\]
where in each component $y\otimes x$ is mapped to $y\cdot M_a(x)$ for $y\in H_k$ and $x\in M_i$. Here $y$ acts on $M_k$ by the $H_k$-module structure of $M_k$. By construction, the map $\alpha$ is an $H_k$-module homomorphism. Define $M'_k \coloneqq \ker \alpha$ as an $H_k$-module. For each $a \colon i \rightarrow k$, there is an induced map
\[
f_a \colon M_k' \rightarrow  H_k \otimes M_i
\]
by the natural inclusion of the kernel composed by the projection to the component $H_k \otimes M_i$. 

If $k$ pending, we see $H_k\otimes M_i = (\Bbbk \cdot \varepsilon_k \otimes M_i) \oplus (\Bbbk \cdot 1 \otimes M_i)$ as a $\Bbbk$-linear space and identify it with $M_i\oplus M_i$ such that $\varepsilon_k(y,x) = (x,0)$. Then the map $f_a$ has two components $f_{1,2}\colon M_k' \rightarrow M_i$ satisfying $f_1(\varepsilon_k x) = f_2(x)$ for any $x\in M_k'$. In this case, we define the map $M'_{a^*}\colon M_k' \rightarrow M_i$ to be $f_2$.

If $k$ is not pending, as $H_k\otimes M_i \cong M_i$ by identifying $1\otimes x$ with $x$, we define $M'_{a^*}$ to be $f_a\colon M_k'\rightarrow M_i$.

To define $M' = F^+_k(M)\in \modu \mathcal P(\kappa)$, we first need to specify $M'_i$ for every $i$. We define
\[
M'_i \coloneqq \begin{cases}
	M_i \quad &\text{if $i\neq k$}\\
	M'_k \quad &\text{if $i = k$}.
\end{cases}
\]
We still need to determine the actions of the arrows in $Q_1(\sigma)$ for $F^+_k(M)$. For the arrows shared by $Q_1(\sigma)$ and $Q_1(\kappa)$ (except $\varepsilon_k$), it is easy to see that their domains and targets do not change. We keep the actions of these arrows the same as before.

For each $a \colon i \rightarrow k$ in $\overline Q_1(\kappa)$, we define the action of $a^*$ by $M'_{a^*}\colon M'_k\rightarrow M_i$.

For each $b \colon k \rightarrow j$ in $\overline Q_1(\kappa)$, there are two cases. If $b$ does not belong to an internal triangle, then we define $M'_{b^*} = 0$. Otherwise $b$ in an internal triangle has the local configuration
\[
i \xrightarrow{a} k \xrightarrow{b} j \xrightarrow {c} i.
\]
Note that $a$ and $c$ satisfy $M_a\circ M_c = 0$. So $M_c$ sends $M_j$ to $\ker a\subset \ker \alpha = M_k'$ where the inclusion is via $\Bbbk \cdot 1 \otimes M_i \subset H_k \otimes M_i$. We define $M'_{b^*} \coloneqq M_c \colon M_j \rightarrow M'_k$.

For any newly created $i \xrightarrow{d} j$ in $\sigma$ associated to $i\xrightarrow{a} k \xrightarrow{b} j$ in $\kappa$, we define $M'_d\colon M_j \rightarrow M_i$ to be $M_b\circ M_{\varepsilon_k} \circ M_a$ if $k$ is pending and otherwise $M_b\circ M_a$.

Now it is clear that $F^+_k(M)$ is a representation of the quiver $Q(\sigma)$. To show it is actually a module over $\mathcal P(\sigma)$ we need to check that the relations $\partial S(\sigma)$ are satisfied.

In fact, as any relation in $\partial S(\sigma)$ are local to each internal triangle or pending arc, we only need to check at the triangles which contain the arc $k$ where the flip $\mu_k$ takes place. There are essentially two cases of local configurations.

\textbf{Case 1: $k$ is not pending.}

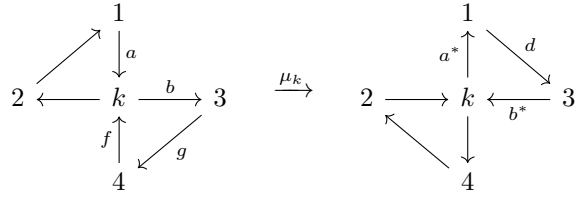
\[
\begin{tikzcd}
& 1 \arrow[d, "a"] \\
2 \arrow[ru] & k \ar[l] \ar[r, "b"] & 3 \arrow[ld, "g"] \\
& 4 \arrow[u, "f"]
\end{tikzcd} \quad \xrightarrow{\mu_k} \quad
\begin{tikzcd}
& 1 \arrow[dr, "d"] \\
2 \arrow[r] & k \ar[u, "a^*"] \ar[d] & 3 \arrow[l, "b^*"] \\
& 4 \arrow[ul]	
\end{tikzcd}
\]
Note that some vertices of $\{1,2,3,4\}$ may be identical or pending, in which case we are not drawing the loops in the diagram. We check the following relations. We use the label of an arrow to represent its action for simplicity.
\begin{enumerate}
    \item $d\circ a^* = 0$. Consider the following commutative diagram.
    \[
    \begin{tikzcd}
       M_k' \ar[d, hook] \ar[r, "a^*"] & M_1 \ar[rd, "a"]\ar[rr, "d"] &  & M_3\\
       M_1 \oplus M_4 \ar[ur, swap, "{(\mathrm{id}, 0)}"] && M_k\ar[ur, "b"]
    \end{tikzcd}
    \]
    Let $(x, y)$ be in $M_k' \subset M_1 \oplus M_4$. Then it must satisfy $a(x) + f(y) = 0$. Thus we have (since $b\cdot f \in \partial S(\kappa)$) that
    $$ b\circ a \circ a^*(x,y) = b\circ a(x) =  -b\circ f (y) = 0. $$
    \item $a^*\circ b^* = 0$. This follows from the diagram below as $a^*\circ b^* = (\mathrm{id}, 0)\circ (0, g)^T = 0$.
    \[
    \begin{tikzcd}
       M_3 \ar[r, "b^*"] \ar[rd, swap, "{(0, g)^T}"] & M_k' \ar[r, "a^*"] \ar[d, hook]& M_1\\
       & M_1 \oplus M_4 \ar[ru, swap, "{(\mathrm{id}, 0)}"]
    \end{tikzcd}
    \]
    \item $b^*\circ d = 0$. This is clear from the following diagram and that $g\circ b = 0$.
    \[
    \begin{tikzcd}
       M_1 \ar[rr, "d"]\ar[rd, "a"] & & M_3 \ar[r, "b^*"] \ar[rd, swap, "{(0, g)^T}"] & M_k' \ar[d, hook]\\
       &M_k \ar[ru, "b"]&& M_1 \oplus M_4
    \end{tikzcd}
    \]
\end{enumerate}

\textbf{Case 2: $k$ is pending.}

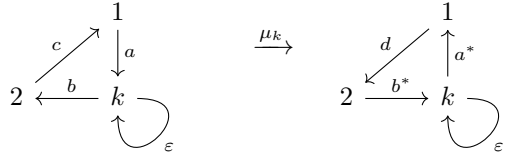
\[
\begin{tikzcd}
	& 1 \ar[d, "a"] \\
	2 \ar[ru, "c"] & k \ar[l, swap, "\ b"] \ar[out=0, in = 270, loop, "\varepsilon"]
\end{tikzcd}
\quad \xrightarrow{\mu_k} \quad
\begin{tikzcd}
	& 1 \ar[dl, swap, "d"] \\
	2 \ar[r, "\ b^*"] & k \ar[u, swap, "a^*"] \ar[out=0, in = 270, loop, "\varepsilon"]
\end{tikzcd}
\]

Note that (exactly) one of the vertices 1 and 2 may also be pending. Denote $\varepsilon = \varepsilon_k$ and simplify the notation $\Bbbk \cdot \varepsilon \otimes M_1 = \varepsilon \otimes M_1$. We check the following relations.

\begin{enumerate}
    \item $d\circ a^* = 0$. Consider the following commutative diagram.
	\[
    \begin{tikzcd}
       M_k' \ar[d, hook] \ar[rr, "a^*"] && M_1 \ar[d, "a"]\ar[r, "d"] &   M_2\\
       M_1 \oplus (\varepsilon \otimes M_1) \ar[urr, swap, "{(x, \varepsilon\otimes y)\mapsto y}", sloped] && M_k\ar[r, "\varepsilon \cdot"] & M_k \ar[u, "b"]
    \end{tikzcd}
    \]
	Let $(x, \varepsilon \otimes y)$ be in $M_k'$. Then it must satisfy $a(x) + \varepsilon a(y) = 0.$ So we have (since $ba\in \partial S(\kappa)$)
	\[
	d\circ a^*((x, \varepsilon \otimes y)) = b(\varepsilon a(y)) = b (- a(x)) = 0.
	\]

	\item $a^*\circ b^* = 0$. This follows from the diagram

	\[
	\begin{tikzcd}
		M_2 \ar[dr, swap, "{(c,0)}"] \ar[r, "b^*"] & M_k' \ar[d, hook] \ar[r, "a^*"] & M_1 \\
		& M_1 \oplus (\varepsilon \otimes M_1)  \ar[ru, swap, "{(x,\varepsilon\otimes y)\mapsto y}"]
	\end{tikzcd}.
	\]

	\item $b^* \circ d = 0$. It follows from the diagram below and $c\circ a = 0$.
	\[
	\begin{tikzcd}
	   M_1 \ar[d, "a"] \ar[r, "d"] & M_2 \ar[r, "b^*"] \ar[rd, swap, "{(c,0)}"] & M_k' \ar[d, hook]\\
	   M_k \ar[r, "\varepsilon \cdot"] & M_k\ar[u, "b"] & M_1 \oplus (\varepsilon \otimes M_1)
	\end{tikzcd}
	\]
	\end{enumerate}

\subsection{Source reflection $F^-_k$} As before, let $M$ be in $\modu \mathcal P(\kappa)$. Consider the $\Bbbk$-linear map
\[
\beta \colon M_k \rightarrow \bigoplus_{\substack{b\in \overline Q_1(\kappa)\\t(b)=k,\, h(b)=j}} H_k \otimes M_j.
\]
defined as follows.

If $k$ is not pending, then each component of $\beta = (\beta_b)_b$ is given by $M_b$ by identifying $1\otimes x\in H_k\otimes M_j$ with $x\in M_j$. If $k$ is pending, then we define $\beta(m) = (1\otimes M_b(\varepsilon_k m)+ \varepsilon_k\otimes M_b(m))_b$ for $m\in M_k$. We check that
\[
\beta(\varepsilon_k m) = (\varepsilon_k\otimes M_b(\varepsilon_k m))_b = \varepsilon_k\cdot\beta(m).
\]
Thus $\beta$ is an $H_k$-homomorphism.
Let $M'_k \coloneqq \coker \beta$ with induced morphisms for each component
\[
f_b \colon H_k\otimes M_j \rightarrow M'_k.
\]
Then $f_b$ further restricts to a linear map on $\Bbbk \cdot 1\otimes M_j \subset H_k \otimes M_j$, which we denote by $M'_{b^*}\colon M_j \rightarrow M'_k$.

We now describe the module $M' = F_k^-(M)$. For $i\neq k$, set $M'_i \coloneqq M_i$; for $i = k$, $M'_k \coloneqq \coker \beta$.

For each $b\colon k\rightarrow j$ in $\overline Q_1(\kappa)$, define the action of $b^*$ by $M'_{b^*}\colon M_j\rightarrow M_k'$.

For each $a\colon i\rightarrow k$ in $\overline Q_1(\kappa)$, there are again two cases. If it does not belong to any internal triangle, set $M'_{a^*} = 0$. Otherwise it has the local configuration
\[
i \xrightarrow{a} k \xrightarrow{b} j \xrightarrow {c} i.
\]
If $k$ is not pending, the relation $c\cdot b = 0$ implies that $M_c\circ M_b \colon M_k \rightarrow M_i$ is $0$ which induces a unique map $M'_{a^*}\colon \coker \beta = M_k' \rightarrow M_i$. If $k$ is pending, let $\widetilde M_c\colon H_k\otimes M_j \rightarrow M_i$ denote the linear map such that $\varepsilon_k\otimes m \mapsto M_c(m)$ and $1\otimes m \mapsto 0$ for $m\in M_j$. Now the local configuration at $k$ is depicted as in \textbf{Case 2} in the last subsection. Then $\widetilde M_c\circ \beta \colon M_k \rightarrow M_i$ vanishes and thus $\widetilde M_c$ factors through $\coker \beta = M_k'$ by a unique map from $\coker \beta$ to $M_i$ which we denote by $M'_{a^*}\colon M_k' \rightarrow M_i$.

The rest of the actions of arrows are defined in the same way as for the case $F^+_k$.

As in the case of $F^+_k$, we then need to check that the actions satisfy the relations given by $\partial S(\sigma)$ so that $F^-_k(M)$ becomes a module over $\mathcal P(\sigma)$. We leave the details in this case to the reader.

\begin{remark}
If $k$ is a sink (resp. source) of $\overline Q_1(\kappa)$ and not pending, then $F_k^+$ (resp. $F_k^-$) is (defined in the same way as) the classical BGP reflection functor \cite{bernstein1973coxeter}. In the case where $\overline Q(\kappa)$ is acyclic and $k$ is a sink or source (possibly pending), $F_k^\pm$ are the same as the reflection functors of \cite{geiss2017quivers} in the corresponding situations. When there is no orbifold point, the functors $F_k^\pm$ fit into a more general setting of quivers with non-degenerate potentials as in \cite{mou2019scattering}. Although definitely our $F_k^\pm$ are motivated by the mutation of decorated representations of Derksen--Weyman--Zelevinsky \cite{derksen2008quivers}, the $F_k^\pm$ here are partial versions of the corresponding DWZ-mutations in the existence of loops which we shall develop in \cite{LMpart2}.
\end{remark}

\subsection{Adjunction between $F^+_k$ and $F^-_k$}

Let $k\in \kappa$ be a vertex. Consider the two reflection functors
\[
F^+_k \colon \modu \mathcal P(\kappa) \rightarrow \modu \mathcal P(\sigma)\quad \text{and} \quad F^-_k \colon \modu \mathcal P(\sigma) \rightarrow \modu \mathcal P(\kappa).
\]

\begin{proposition}\label{prop: adjunction}
The functors $F^+_k$ and $F^-_k$ are adjoint to each other, i.e., there is a bifunctorially isomorphism
\[
\Hom_{\mathcal P(\sigma)}(N, F^+_k(M)) \cong \Hom_{\mathcal P(\kappa)} (F^-_k(N), M) 
\]
for any $N \in \modu \mathcal P(\sigma)$ and $M \in \modu \mathcal P(\kappa)$.
\end{proposition}

\begin{proof}
There are two cases on whether $k$ is pending.

\textbf{Case 1: $k$ is not pending.}

\[
\begin{tikzcd}
& 1 \arrow[d, "a"] \\
2 \arrow[ru, "c"] & k \ar[l, "e"] \ar[r, "b"] & 3 \arrow[ld, "g"] \\
& 4 \arrow[u, "f"]
\end{tikzcd} \quad \xrightarrow{\mu_k} \quad
\begin{tikzcd}
& 1 \arrow[dr, "d"] \\
2 \arrow[r, "e^*"] & k \ar[u, "a^*"] \ar[d, "f^*"] & 3 \arrow[l, "b^*"] \\
& 4 \arrow[ul, "h"]	
\end{tikzcd}
\]

For any two representations $M$ and $N$ of a quiver $Q$ and a \emph{path} $a$ from $i$ to $j$ (i.e. a composition of arrows from $i$ to $j$), denote by $S_a(M,N)$ the equation
\[
x_j \circ M_a = N_a \circ x_i
\]
for $(x_i)_{i\in Q_0}$ a tuple of linear maps $x_i\colon M_i \rightarrow N_i$. The space of morphisms in $\modu \mathcal P(\kappa)$ has the following description 
\[
\Hom_{\mathcal P(\kappa)}(F^-_k(N), M) = \{(x_i)_{i\in \kappa} \mid x_i \colon F^-_k(N)_i \rightarrow M_i \ \text{s.t. $S_y(F^-_k(N), M)$ for any $y\in Q_1(\kappa)$} \}.
\]
Note that for $N' = F^-_k(N)$, the map $\alpha \colon N'_1 \oplus N'_4 \rightarrow N'_k$ is surjective. Thus $x_k$ is determined by $(x_i)_{i\neq k}$. So the space $\Hom_{\mathcal P(\kappa)}(N', M)$ is set of tuples $(x_i)_{i\neq k}$ satisfying
\begin{enumerate}
    \item $S_y(N', M)$ for $y\in Q_1(\kappa)$ not incident to $k$;
    \item $S_{e f}(N', M)$ and $S_{b a}(N', M)$;
    \item For any $(n_1, n_4)\in N'_1\oplus N'_4$ such that $N'_a(n_1) + N'_f(n_4) = 0$, $M_a(x_1(n_1)) + M_f(x_4(n_4)) = 0$, which is equivalent to that for any $n\in N_k$, $M_a\circ x_1 \circ N_{a^*}(n) + M_f\circ x_4 \circ N_{f^*}(n) = 0$.
\end{enumerate}

On the other hand, we can express $\Hom_{\mathcal P(\sigma)}(N, M')$ (writing $M' = F^+_k(M)$) as the set of tuples of linear maps $(z_i\colon N_i\rightarrow M'_i)_{i\neq k}$ such that
\begin{enumerate}
    \item $S_y(N, M')$ for $y\in Q_1(\sigma)$ not incident to $k$;
    \item $S_{a^* e^*}(N, M')$ and $S_{f^* b^*}(N, M')$;
    \item For any $n\in N_k$, we have $(z_1(N_{a^*}(n)), z_4(N_{f^*}(n)))\in M'_k\subset M_1'\oplus M_4'$, which is equivalent to that for any $n\in N_k$, $M_a\circ z_1 \circ N_{a^*}(n) + M_f\circ z_4 \circ N_{f^*}(n) = 0$.
\end{enumerate}

By the constructions of $N'$ and $M'$, there are following equivalences of equations
\begin{align*}
S_{e f}(N', M) & \Leftrightarrow S_{h}(N, M'),\\
S_{b a}(N', M) & \Leftrightarrow S_{d}(N, M'),\\
S_{a^* e^*}(N,M') &\Leftrightarrow S_c (N', M),\\
S_{f^* b^*}(N,M') &\Leftrightarrow S_g (N', M),
\end{align*}
and $S_y(N, M') \Leftrightarrow S_y(N, M')$ for any other $y$ shared by $Q_1(\kappa)$ and $Q_1(\sigma)$. Then the two spaces $\Hom_{\mathcal P(\kappa)}(N', M)$ and $\Hom_{\mathcal P(\sigma)}(N, M')$ are canonically identified and this identification is bifunctorial.

\textbf{Case 2: $k$ is pending.}

\[
\begin{tikzcd}
	& 1 \ar[d, "b"] \\
	2 \ar[ru, "c"] & k \ar[l, swap, "\ a"] \ar[out=0, in = 270, loop, "\varepsilon_k"]
\end{tikzcd}
\quad \xrightarrow{\mu_k} \quad
\begin{tikzcd}
	& 1 \ar[dl, swap, "d"] \\
	2 \ar[r, "\ a^*"] & k \ar[u, swap, "b^*"] \ar[out=0, in = 270, loop, "\varepsilon_k"]
\end{tikzcd}
\]

We keep the same notations such as $N'$ and $M'$ as in Case 1. Similarly to Case 1, the space $\Hom_{\mathcal P(\kappa)}(N', M)$ consists of tuples of linear maps $(x_i\colon N'_i\rightarrow M_i)_{i\neq k}$ such that
\begin{enumerate}
    \item $S_y(N', M)$ for $y\in Q_1(\kappa)$ not incident to $k$;
    \item $S_{a\varepsilon_k b}(N',M)$;
    \item $\alpha_M \circ x_1 \circ \beta_N (N_k) = 0$,
\end{enumerate}
and hence $x_k\colon N'_k\rightarrow M_k$ can be well-defined by the commutative diagram
\[
\begin{tikzcd}
N_k \ar[r, "\beta_N"] & H_k\otimes N_1' \ar[r, two heads, "\alpha_{N'}"] \ar[d, "x_1"]& N_k' = \coker \beta_{N} \ar[d, "x_k"] \\
& H_k\otimes M_1 \ar[r, "\alpha_M"] & M_k
\end{tikzcd}.
\]
Analogously, the space $\Hom_{\mathcal P(\sigma)}(N, M')$ is the set of tuples $(z_i)_{i\neq k}$ such that
\begin{enumerate}
    \item $S_y(N, M')$ for $y\in Q_1(\sigma)$ not incident to $k$;
    \item $S_{b^*\varepsilon_k a^*}(N',M)$;
    \item $\alpha_M \circ x_1 \circ \beta_N (N_k) = 0$.
\end{enumerate}
Now that $S_c(N', M) \Leftrightarrow S_{b^*\varepsilon_k a^*} (N, M')$, $S_d(N, M') \Leftrightarrow S_{a\varepsilon_k b} (N', M)$, and $S_y(N', M) \Leftrightarrow S_y(N, M')$ for any other $y\neq \varepsilon_k$ shared by $Q_1(\kappa)$ and $Q_1(\sigma)$, these two homomorphism spaces are canonically identified bifunctorially.

Combining Case 1 and Case 2, we have proven the adjunction between $F^+_k$ and $F^-_k$.
\end{proof}

For a labeled triangulation $\kappa$, denote by $S_k$ the simple module in $\modu \mathcal P(\kappa)$ corresponding to $k\in \kappa$. We define the following two full subcategories
\[
S_k^\perp \coloneqq \{M\in \modu \mathcal P(\kappa) \mid \Hom(S_k, M) = 0\},\quad ^\perp S_k \coloneqq \{M\in \modu \mathcal P(\kappa) \mid \Hom(M, S_k) = 0\}.
\]

\begin{proposition}\label{cor: left perp and right perp}
The functor $F^+_k$ restricts to an equivalence from $^\perp S_k\subset \modu \mathcal P(\kappa)$ to $S_k^\perp \subset \modu \mathcal P(\sigma)$ with $F_k^-$ being its inverse equivalence.
\end{proposition}

\begin{proof}
By the adjunction \Cref{prop: adjunction}, there are two natural transformations the counit $$\varepsilon_M \colon F^-_kF^+_k(M) \rightarrow M$$ for $M\in \modu \mathcal P(\kappa)$ and the unit $$\eta_N \colon N \rightarrow F^+_kF^-_k(N)$$ for $N\in \modu \mathcal P(\sigma)$. A characterization of $M\in {^\perp S_k}$ (resp. $S_k^\perp$) is the map $\alpha$ being surjective (resp. $\beta$ being injective). Thus $S_k^\perp$ (resp. $^\perp S_k$) is the image of $F_k^+$ (resp. $F_k^-$). It is also clear that for $M\in {^\perp S_k}$ (resp. $N\in S_k^\perp$), the counit $\varepsilon_M$ (resp. the unit $\eta_N$) is an isomorphism.
\end{proof}

\section{Stability conditions under reflections}\label{section: stability}

In this section we show that the reflection functors $F^\pm_k$ relate stability conditions for the abelian categories $\modu \mathcal P(\sigma)$ and $\modu \mathcal P(\kappa)$ in a certain way when $\sigma$ and $\kappa$ are related by a flip.

\subsection{Stability conditions}

\begin{definition}[King \cite{king1994moduli}]\label{def: stability condition}
 Let $A$ be a finite dimensional $\Bbbk$-algebra and $\theta \colon K_0(\modu A)\rightarrow \mathbb R$ an additive function on the Grothendieck group of $\modu A$. An object $M\in \modu A$ is called $\theta$-semistable (resp. $\theta$-stable) if $\theta(M) = 0$ and for any non-zero proper submodule $M'\subset M$, $\theta(M') \geq 0$ (resp. $>0$).
\end{definition}

The full subcategory of all $\theta$-semistable modules is denoted by $\moduss{\theta} A$. It is standard that $\moduss{\theta} A$ is closed under kernels, cokernels and extensions. We will call a full subcategory of $\modu A$ satisfying those properties a \emph{wide subcategory}.

We will from now on call such $\theta$ a (King's) \emph{stability condition} of $\modu A$. Let $\Bbbk$ be algebraically closed. Then it is well-known that $\modu A$ is equivalent to the module category of a \emph{basic algebra}. In other words, we could assume that $A$ comes from a quiver $Q$ with admissible relations, i.e. $A = \Bbbk \langle Q \rangle/I$ where $I\subset \Bbbk \langle Q \rangle$ is an ideal spanned by linear combinations of paths of length at least two. Now the category $\modu A$ is equivalent to $\rep(Q, I)$ the category of finite dimensional $\Bbbk$-linear representations of $Q$ satisfying relations in $I$. 

Identifying $Q_0$ with $\{1,\dots, n\}$ where $n = |Q_0|$, the Grothendieck group $K_0(\modu A)$ is then isomorphic to $\mathbb Z^n$ with the basis of (isoclasses) of simple modules $\{S_i\mid i\in Q_0\}$ identified with the standard basis $\{e_i\mid i = 1, \dots, n\}$ of $\mathbb Z^n$. A stability condition $\theta$ can be expressed as a tuple $(\theta_i)_i\in \mathbb R^n$ such that $\theta(S_i) = \theta_i$.

Recall that we have associated the matrix $B(\kappa) = (b_{ij})$ to a labeled triangulation $\kappa$.
Let $\theta = (\theta_i)_i$ be a stability condition of $\modu \mathcal P(\kappa)$. Fix $k\in \kappa$. Consider the flipped triangulation $\sigma = \mu_k(\kappa)$ and the associated gentle algebra $\mathcal P(\sigma)$. We define two stability conditions $\theta'$ and $\theta''$ on $\modu \mathcal P(\sigma)$ in terms of $\theta$ as follows. Define $\theta' = (\theta'_i)_i$ by
\begin{equation}\label{eq: pl map theta'}
\theta'_i \coloneqq \begin{cases}
        \theta_i + [b_{ik}]_+ \theta_k \quad & \text{if $i \neq k$}\\
        -\theta_k \quad & \text{if $i = k$}
\end{cases}    
\end{equation}
and $\theta'' = (\theta''_i)_i$ by
\begin{equation}\label{eq: pl map theta''}
\theta''_i \coloneqq \begin{cases}
        \theta_i + [-b_{ik}]_+ \theta_k \quad & \text{if $i \neq k$}\\
        -\theta_k \quad & \text{if $i = k$}.
\end{cases}    
\end{equation}

\begin{theorem} \label{thm: stability under mutation}
Let $\kappa$, $k$, $\sigma$ be as above. Let $\theta = (\theta_i)_i$ be a stability condition of $\modu \mathcal P(\kappa)$ such that $\theta_k>0$ (resp. $\theta_k<0$). Then $M\in \modu \mathcal P(\kappa)$ is $\theta$-semistable if and only if $F^+_k(M)\in \modu \mathcal P(\sigma)$ (resp. $F_k^-(M)$) is $\theta'$-semistable (resp. $\theta''$-semistable)
and $M\in {^\perp S_k}$ (resp. $M\in S_k^\perp$).
\end{theorem}

\begin{proof}
Let $M\in \modu \mathcal P(\kappa)$ be $\theta$-semistable with $\theta_k >0$. The $\theta$-semistability implies that $M$ does not have any quotient isomorphic to $S_k$. Thus $M$ belongs to $^\perp S_k$.

Denote by $\mathbf d = (d_i)_i$ the class of $M$ in the Grothendieck group in the basis of $S_i$ and $\mathbf d' = (d_i')_i$ for $M' = F_k^+(M)$. Since $M\in {^\perp S_k}$, by the construction of $F_k^+$ in \Cref{section: F_k^+}, we have that $d_i' = d_i$ for $i\neq k$ and
\[
d'_k = -d_k + \sum_i [b_{ik}]_+d_i.
\]
Thus $\theta'(M') = \theta(M) = 0$. Any submodule $N\subset M$ has a submodule $L_1\in {^\perp S_k}$ generated by the subspace $\bigoplus_{i\neq k} N_i$. The exact sequence $L_1\hookrightarrow M \twoheadrightarrow L_2$ in $^\perp S_k\subset \modu \mathcal P(\kappa)$ is sent by $F_k^+$ to an exact sequence $L_1'\hookrightarrow M' \twoheadrightarrow L_2'$ in $S_k^\perp \subset \modu \mathcal P(\sigma)$ with the identity
\[
\theta(L_1) = \theta'(L_1') = - \theta'(L_2').
\]
The module $M$ is $\theta$-semistable if and only if any such $L_1\hookrightarrow M$ satisfies $\theta(L_1) \geq 0$. Thus it is equivalent to that any such $M'\twoheadrightarrow L_2'$ (in $S_k^\perp$) satisfies $\theta'(L_2') \leq 0$, which is the same condition of $M'$ being $\theta'$-semistable. In fact, here such a quotient $L_2'$ of some $M'\in S_k^\perp$ has the characterization that for a quotient $M'\twoheadrightarrow N$, there is a quotient module $L_2'$ of $N$ cogenerated by the quotient vector space $N/N_k$, which guarantees that $L_2'$ belongs to $S_k^\perp$. Then $M'$ is $\theta'$-semistable if and only if any such quotient $L_2'$ satisfies $\theta'(L_2')\leq 0$.

Now suppose that $F_k^+(M)$ is $\theta'$-semistable and that $M\in {^\perp S_k}$. By \Cref{cor: left perp and right perp}, $M$ is isomorphic to $F_k^-F_k^+(M)$. Then proving the if part for $F_k^+$ amounts to show the only if part for $F_k^-$, which is similar to what we have done for the only if part for $F_k^+$. 

\end{proof}

The following is a direct corollary of \Cref{thm: stability under mutation}.

\begin{corollary}\label{cor: stability after mutation}
For any $\theta$ such that $\theta_k >0$ (resp. $\theta_k <0$), the functor $F^+_k$ (resp. $F^-_k$) induces an equivalence between wide subcategories $\moduss{\theta} \mathcal P(\kappa)$ and $\moduss{\theta^\prime} \mathcal P(\sigma)$ (resp. $\moduss{\theta^{\prime\prime}} \mathcal P(\sigma)$) with $F^-_k$ (resp. $F^+_k$) as its inverse equivalence.
\end{corollary}

Let $A = \Bbbk \langle Q \rangle/I$ as in the beginning of the section. Let $N = K_0(\modu A)$ and $M = \Hom(N, \mathbb Z)$. Denote $e_i = S_i\in N$ for $i\in Q_0$. Define $N^\oplus \coloneqq \{ \sum_i \lambda_i e_i \mid \lambda_i\geq 0\}\subset N$ and $N^+ = N^\oplus \setminus \{0\}$. Denote $M_\mathbb R = M\otimes_\mathbb Z \mathbb R$. Denote by $\mathcal W(A)$ the collection of all wide subcategories of $\modu A$. Consider the function
\[
\Phi_A \colon M_\mathbb R \rightarrow \mathcal W(A), \quad \theta \mapsto \moduss{\theta} A.
\]
We denote the wide subcategory of only zero objects simply by $0$. For example, for $\theta\in M_\mathbb R$ generic, we have $\Phi_A(\theta) = 0$ as $\theta = (\theta_i)_i$ does not satisfy any linear equation $\sum_i \theta_id_i = 0$ for $\mathbf d = (d_i)_i\in \mathbb Z^n$. For $M\in \modu A$, denote by $\ab{M}$ the smallest wide subcategory of $\modu A$ containing $M$.

In the case that $A = \mathcal P(\kappa)$ from some labeled triangulation $\kappa$, denote for simplicity that $\Phi_\kappa = \Phi_{\mathcal P(\kappa)}$. The following proposition regards the situation when $\theta_k = 0$.

\begin{proposition}\label{prop: when theta_k = 0}
Let $\theta = (\theta_i)_i \in M_\mathbb R$ be a stability condition (of both $\modu \mathcal P(\kappa)$ and $\modu \mathcal P(\sigma)$) such that $\theta_k = 0$. Then $\Phi_\kappa(\theta) = \ab{S_k}$ if and only if $\Phi_\sigma(\theta) = \ab{S_k}$.
\end{proposition}

\begin{proof}
By definition any object in $\ab{S_k}$ is $\theta$-semistable. If there exists $M\notin \ab{S_k}\subset \modu \mathcal P(\kappa)$ that is $\theta$-semistable. Let $N$ be the submodule of $M$ generated by $\bigoplus_{i\neq k} M_i$, which by definition is again $\theta$-semistable and belongs to $^\perp S_k$. Then it can be verified as in the proof of \Cref{thm: stability under mutation} that $F_k^+(N)\in \modu \mathcal P(\sigma)$ is $\theta$-semistable. It does not belong to $\ab{S_k}\subset \modu \mathcal P(\sigma)$ because $\bigoplus_{i\neq k}(F_k^+(N))_i = \bigoplus_{i\neq k} N_i \neq 0$.
\end{proof}

Denote by $\mathcal C^+$ the first (closed) quadrant in $M_\mathbb R$ with the basis $(e_i^*)_i$ dual to $(e_i)_i$, i.e., $$\mathcal C^+ \coloneqq \left \{\sum_{i=1}^n \lambda_i e_i^* \mid \lambda_i \geq 0 \right \} \subset M_\mathbb R.$$
The subset $\mathcal C^+$ is a closed polyhedral cone in $M_\mathbb R$. Denote by $C^\circ$ the interior of any closed cone $C$ in $M_\mathbb R$.  

\begin{lemma}\label{lemma: stability at first quadrant}
We have
\begin{enumerate}
    \item $\Phi_A(\theta) = 0$ for any $\theta \in (\mathcal C^+)^\circ$;
    \item $\Phi_A(\theta) = \ab{S_i}$ for any $\theta \in (\mathcal C^+ \cap e_i^\perp)^\circ$;
    \item $\Phi_A(\theta) = \ab{S_i}$ for any generic $\theta \in e_i^\perp$.
\end{enumerate}

\end{lemma}

\begin{proof}
(1) Any $\theta = (\theta_i)_i \in (\mathcal C^+)^\circ$ has every component $\theta_i$ positive. By definition the only $\theta$-semistable object is $0$. (2) Any $\theta \in (\mathcal C^+ \cap e_i^\perp)^\circ $ satisfies $\theta_j > 0$ for $j\neq i$ and $\theta_i = 0$. Thus $\theta$ can only have zero evaluation on $M\in \ab{S_i}$. Any $M\in \ab{S_i}$ is obviously $\theta$-semistable. (3) If $\theta$ avoids the countable union of codimension one subsets of the form $e_i^\perp \cap d^\perp$ for $d\in N$ not proportional to $e_i$, any module $M$ satisfying $\theta(M) = 0$ must be in $\ab{S_i}$, which is $\theta$-semistable.
\end{proof}

Let $\kappa$ and $\sigma$ be two triangulations related by a flip at $k$. Denote by 
\[
T_k^+ = T_{k;\kappa}^+\colon M_\mathbb R\rightarrow M_\mathbb R\quad \text{(resp. $T_k^- = T_{k;\kappa}^-$)}
\]
the linear map sending $\theta$ to $\theta'$ (resp. $\theta''$) as in (\ref{eq: pl map theta'}) and (\ref{eq: pl map theta''}). Define two (open) half spaces in $M_\mathbb R$
\[
H_k^+ \coloneqq \{\theta \in M_\mathbb R \mid \theta_k > 0\}\quad \text{and} \quad H_k^- \coloneqq \{\theta \in M_\mathbb R \mid \theta_k < 0\}.
\]
We define the piecewise linear map
\[
T_k = T_{k;\kappa}\colon M_\mathbb R\rightarrow M_\mathbb R
\]
such that $T_k\vert_{\overline{H}_k^+} = T_k^+$ and $T_k\vert_{\overline{H}_k^-} = T_k^-$. Notice that both $T^\pm_k$ (thus $T_k$) fix points on the hyperplane $e_k^\perp = \{\theta \mid \theta_k = 0\}\subset M_\mathbb R$.

\begin{remark}
   We note that the piecewise linear transformation $T_k$ defined above should not be confused with the one in \cite[Definition 1.22]{gross2018canonical} which we denote by $\tilde T_k$ below. In the current situation, their $\tilde T_k \colon M_\mathbb R \rightarrow M_\mathbb R$ reads as follows. Let $\theta' = (\theta_i')_i = \tilde T_k(\theta)$. Then
    \[
        \theta_i' = \begin{cases}
            \theta_i \quad &\text{for any $i\in I$ if $\theta \in \overline H_k^-$} \\
            \theta_i + b_{ik}\theta_k \quad &\text{for any $i\in I$ if $\theta \in \overline H_k^+$}.
        \end{cases}
    \]
    Thus the maps $T_k$ and $\tilde T_k$ are related by $T_k =  \tilde T_k \circ T_k^+$.
\end{remark}

The following proposition rephrases \Cref{cor: stability after mutation} in terms of the function $\Phi_A$.

\begin{proposition}\label{prop: stability after mutation}
For $\theta\in H_k^+$, we have $\Phi_\sigma (T^+_k (\theta)) = F^+_k(\Phi_\kappa (\theta))$; For $\theta\in H_k^-$, we have $\Phi_\sigma (T^-_k (\theta)) = F_k^-(\Phi_\kappa(\theta))$.
\end{proposition}

\subsection{Cluster chamber structure}\label{subsection: cluster chamber structure}

Applying \Cref{prop: when theta_k = 0}, \Cref{lemma: stability at first quadrant} and \Cref{prop: stability after mutation} iteratively for sequences of mutations, one can build the so-called \emph{cluster chamber structure} on a subset of $M_\mathbb R$ from the function $\Phi_\kappa$ as follows.

Let $\sigma = \kappa(t)$ be the triangulation associated to $t\in \mathbb T_n$ (with $\kappa$ associated to $t_0$). Suppose the unique path from $t_0$ to $t$ takes the sequence $(k_1, k_2, \dots, k_l)$ of flips. Denote $\kappa_j \coloneqq \mu_{k_j}\cdots\mu_{k_2}\mu_{k_1}(\kappa)$ with $\kappa_0 = \kappa$ and $\kappa_l = \sigma$.

\begin{proposition}\label{prop: chamber}
There is a unique choice of signs $(\epsilon_1, \epsilon_2,\dots, \epsilon_l)\in \{+,-\}^l$ with $\epsilon_l = +$ such that for any $j = 2, \dots, l$, the cone
\[
\mathcal C_{j-1} \coloneqq T_{k_j;\kappa_j}^{\epsilon_j}\circ T_{k_{j+1}; \kappa_{j+1}}^{\epsilon_{j+1}}\circ \cdots \circ T_{k_{l}; \kappa_l}^{\epsilon_l}(\mathcal C^+)
\]
is contained in $\overline{H}_{k_{j-1}}^{\epsilon_{j-1}}$.
\end{proposition}

\begin{proof}
We prove the proposition by induction backwards on $j$ from $j=l$. In fact, the cone $T_{k_l;\kappa_l}^+({\mathcal C}^+)$ is contained either in $\overline{H}_k^+$ or $\overline{H}_k^-$ for any $k$. This is because by \Cref{prop: stability after mutation}, for any $\theta\in (\mathcal C^+)^\circ$, we have
\[
\Phi_{\kappa_{l-1}}(T_{k_l;\kappa_l}^+(\theta)) = F_{k_l}^+(\Phi_{\kappa_l}(\theta)) = 0.
\]
Therefore the hyperplane $e_k^\perp$ does not intersect with $T_{k_l;\kappa_l}^+(\mathcal C^+)$ in its interior since a generic point on $e_k^\perp$ supports $\ab{S_k}$ for the function $\Phi_{\kappa_{l-1}}$ (by (3) of \Cref{lemma: stability at first quadrant}). So we determine $\epsilon_{l-1}$ such that ${\mathcal C_{l-1}} = T_{k_l;\kappa_l}^+(\mathcal C^+)$ is contained in $\overline{H}_{k_{l-1}}^{\epsilon_{l-1}}$.

Now suppose for induction that we have determined the sign $\epsilon_j$ such that that $\mathcal C_{j}$ is contained in $\overline{H}_{k_{j}}^{\epsilon_j}$ and that $\Phi_{\kappa_j}(\theta) = 0$ for any $\theta\in \mathcal C_j^\circ$. Consider the reflection functor
\[
F_{k_j}^{\epsilon_j} \colon \modu \mathcal P(\kappa_{j}) \rightarrow \modu \mathcal P(\kappa_{j-1}).
\]
By \Cref{prop: stability after mutation}, for any $\theta \in \mathcal C_{j}^\circ \subset H_{k_j}^{\epsilon_j}$, we have
\[
\Phi_{\kappa_{j-1}}(T_{k_j;\kappa_j}^{\epsilon_{j}}(\theta)) = F_{k_j}^{\epsilon_j}(\Phi_{\kappa_{j}}(\theta)) = 0.
\]
Thus the cone $\mathcal C_{j-1} = T_{k_j;\kappa_j}^{\epsilon_j}(\mathcal C_{j})$ must be contained in either $\overline{H}_{k_{j-1}}^+$ or $\overline{H}_{k_{j-1}}^-$. Then we determine the sign $\epsilon_{j-1}$ such that $\mathcal C_{j-1}\subset H_{k_{j-1}}^{\epsilon_{j-1}}$, which finishes the induction.
\end{proof}

\begin{corollary}\label{cor: stability in a chamber}
Define the cone in $M_\mathbb R$ $$\mathcal C_t = \mathcal C_{t}^{\kappa;t_0} \coloneqq  T_{t}^{\kappa;t_0}(\mathcal C^+) = T_{k_1;\kappa_1}^{\epsilon_1}\circ T_{k_2;\kappa_2}^{\epsilon_2}\circ \cdots \circ T_{k_{l};\kappa_l}^{\epsilon_l}(\mathcal C^+).$$ We have for any $\theta\in \mathcal C_{t}^\circ$ that $\Phi_{\kappa} (\theta) = 0$.
\end{corollary}

\begin{proof}
This is a direct consequence of (the proof of) \Cref{prop: chamber}. In fact, it is clear from the induction that for any $\theta\in \mathcal C_j^\circ$ for $j = 1, 2,\dots, l-1$, we have $\Phi_{\kappa_j}(\theta) = 0$. It is extended to $j=0$ for one more step of the induction.
\end{proof}

Note that each map $T_{k_j;\kappa_j}^{\epsilon_j}$ is an automorphism of $M\cong \mathbb Z^n$. So each $\mathcal C_j$ is again a maximal simplicial cone whose $n$ rays has the enumeration by $\{1, \dots, n\}$ inherited from that of $\mathcal C^+$. We denote the integral generators of these rays by $\mathbf r^j_i\in M$ for $i\in \{1, \dots, n\}$, which form an integral basis of $M$. The $n$ (relatively open) facets $\{\mathfrak d_i^j\mid i=1,\dots, n\}$ of $\mathcal C_j$ also has the enumeration such that $\mathbf r^j_i \notin \bar {\mathfrak d}^j_i$. The normal vectors of these facets dual to $\mathbf r^j_i$ are denoted by $\mathbf n^j_i\in N$. As already did in the proof of \Cref{cor: stability in a chamber}, we continue to denote $\mathcal C_0 = \mathcal C_t$ and $\mathfrak d^0_i = \mathfrak d_{i;t} $. We next consider the function value of $\Phi_\kappa$ in the interior of a facet $\mathfrak d_{i;t}$. 

\begin{construction}\label{con: semistable on a cluster wall}
We construct a subcategory $T_t(\ab{S_i}) = T_t^{\kappa;t_0}(\ab{S_i})$ of $\modu \mathcal P(\kappa)$ iteratively as follows. First consider $\ab{S_i}\subset \modu \mathcal P(\sigma)$, which is the smallest wide subcategory containing $S_i$. We have $\Phi_{\kappa_l}(\theta) = \Phi_{\sigma}(\theta) = \ab{S_i}$ for any $\theta\in \mathfrak d_i^l$. If $i\neq k_l$, we have for any $\theta\in \mathfrak d_i^{l-1}$
\[
\Phi_{\kappa_{l-1}}(\theta) = F_{k_l}^{\epsilon_l}(\ab{S_i}).
\]
If $i = k_l$, we have for any $\theta\in \mathfrak d_{k_l}^{l-1}$
\[
\Phi_{\kappa_{l-1}}(\theta) = \ab{S_{k_l}}\subset \modu \mathcal P({\kappa_{l-1}}).
\]
Assume for some $j$ that for any $i$, the function $\Phi_{\kappa_j}$ is constant on $\mathfrak d^j_i$. For $\mathfrak d^{j-1}_i$, by \Cref{prop: chamber}, there are two cases: (1) $\mathfrak d^{j-1}_i\subset e_{k_j}^\perp$ and (2) $\mathfrak d^{j-1}_i \cap e_{k_j}^\perp = \emptyset$. In the first case, $\mathfrak d_i^{j-1} = \mathfrak d_i^j$ as the transformation $T_{k_j;\kappa_j}^{\epsilon_j}$ fixes the hyperplane $e_{k_j}^\perp$. By the assumption $\Phi_{\kappa_j}(\theta)$ is constantly $\ab{S_{k_j}}$ in $\mathfrak d^{j-1}_i$. Hence by \Cref{prop: when theta_k = 0}, we have $\Phi_{\kappa_{j-1}}(\theta) = \ab{S_{k_j}}\subset \modu \mathcal P(\kappa_{j-1})$ for any $\theta\in \mathfrak d^{j-1}_i$. In the second case, we know from \Cref{prop: stability after mutation} that for any $\theta\in \mathfrak d^{j}_{i}$
\[
\Phi_{\kappa_{j-1}}(T_{k_j;\kappa_j}^{\epsilon_j}(\theta)) = F_{k_j}^{\epsilon_j}(\Phi_{\kappa_j}(\theta)) \subset \modu \mathcal P(\kappa_{j-1}).
\]
In either case, we see that $\Phi_{\kappa_{j-1}}(\mathfrak d^{j-1}_i)$ is equivalent to  $\Phi_{\kappa_j}(\mathfrak d^j_i)$. 

The above construction gives an inductive way to describe the subcategory $\Phi_{\kappa_{j}}(\mathfrak d_i^j)$ for $j=0,\dots, l$. We define $T_t(\ab{S_i})\subset \modu \mathcal P(\kappa)$ to be $\Phi_{\kappa}(\mathfrak d_{i; t})$ (corresponding to $j=0$). It is clear from the iterative construction of $T_t(\ab{S_i})$ that it is equivalent to $\ab{S_i}\subset \modu \mathcal P(\sigma)$ which is equivalent to (1) $\modu \Bbbk$ if $i$ is not pending; (2) $\modu \Bbbk[\varepsilon]/\varepsilon^2$ if $i$ is pending. \qed
\end{construction}

Summarizing \Cref{cor: stability in a chamber} and \Cref{con: semistable on a cluster wall}, we have
\begin{proposition}[Chamber structure]\label{prop: function value of phi}
The function $\Phi_\kappa\colon M_\mathbb R \rightarrow \mathcal W({\mathcal P(\kappa)})$ 
\begin{enumerate}
    \item takes value $0$ in the interior of $\mathcal C_t$ for any $t\in \mathbb T_n$, and
    \item is constant in the interior of any facet $\mathfrak d_{i;t}$ of $\mathcal C_t$ with value $T_t(\ab{S_i})$.
\end{enumerate} 
\end{proposition}

Now we relate the cones $\mathcal C_t$ with the $\mathbf g$-vectors and $\mathbf c$-vectors associated to the generalized cluster algebra $\mathcal A(\kappa)$. Namely we will show that the \emph{chamber structure} in \Cref{prop: function value of phi} coincide with the cluster combinatorics of $\mathcal A(\kappa)$, thus bearing the name \emph{cluster chamber structure}. We denote $\mathbf r_{i; t}\in M$ for the integral generators of the rays of $\mathcal C_t$ and $\mathbf n_{i; t}\in N$ the dual normal vectors of the facets $\mathfrak d_{i;t}$.

\begin{theorem}\label{thm: g-vectors are rays of stability chambers}
We have for any $t\in \mathbb T_n$ and and $i\in \{1, \dots, n\}$ the equalities of vectors
\[
\mathbf r_{i; t} = \mathbf g_{i; t}\in M,\quad \mathbf n_{i; t} = \mathbf c_{i; t}\in N.
\]
\end{theorem}

\begin{proof}
This theorem is proved inductively. Initially, $\mathbf r_{i;t_0} = \mathbf g_{i;t_0} = e_i^*\in M$ and $\mathbf n_{i;t_0} = \mathbf c_{i;t_0}= e_i$ for all $i$. We show that the recursive formula for $\mathbf r_{i; t}$ is the same as that of $\mathbf g_{i; t}$ and the one of $\mathbf n_{i; t}$ the same as that of $\mathbf c_{i; t}$. Note that by definition $(\mathbf n_{i;t})_i$ is the dual basis of $(\mathbf r_{i;t})_i$.

A remarkable feature is that for any $\mathbf n_{i; t}$, its coordinates (in the basis $e_i$) always have the same sign (called \emph{sign-coherent}), i.e., they are either all non-negative or non-positive. This is because by \Cref{prop: function value of phi} the corresponding facet $\mathfrak d_{i; t}$ constantly supports the wide subcategory $T_t(\ab{S_i})$ via the function $\Phi_\kappa$ so it must be orthogonal to the dimension vector $d\in N^+$ of the generator of $T_t(\ab{S_i})$. So the normal vector $\mathbf n_{i; t}$ must be sign-coherent.

With the duals $\mathbf n_{i; t}$ known to be sign-coherent, the vectors $\mathbf r_{i;t}$ satisfy the following recursion (for $t' \frac{k}{\quad\quad} t$ in $\mathbb T_n$):
\[
\mathbf r_{i;t'} = \begin{cases}
\mathbf r_{i; t} \quad &\text{if $i \neq k $}\\
 - \mathbf r_{k;t} + \sum_{i = 1}^n [-\operatorname{sgn}(\mathbf n_{k;t}) b_{ik}(t)]_+ \mathbf r_{i; t} \quad &\text{if $i = k$}.
\end{cases}
\]
This result can be extracted from \cite[Corollary 5.5 and 5.9]{gross2018canonical} which only relies on the definitions of $\mathbf r_{i;t}$ as $T_{t}^{\kappa;t_0}(e_i)$. A more direct proof can be found in \cite[Proposition 7.7]{mou2021scattering}. Accordingly the dual vectors $\mathbf n_{i; t}$ have the recursion
\[
\mathbf n_{i; t'} = \begin{cases}
\mathbf n_{i; t} + [\operatorname{sgn}(\mathbf n_{k; t})b_{ik}(t)]_+ \mathbf n_{k; t} \quad &\text{if $i\neq k$}\\
-\mathbf n_{k; t} \quad &\text{if $i = k$}.
\end{cases}
\]
Using the sign-coherence of the vectors $\mathbf n_{i; t}$, one sees their recursion is exactly the same as that of the $\mathbf c$-vectors in \Cref{section: generalized cluster algebra} by induction on the distance from $t$ to $t_0$. This leads to the equality $\mathbf n_{i; t} = \mathbf c_{i; t}$ for any $t$ and $i$. Now with the $\mathbf c$-vectors known to be sign-coherent, it is clear that the $\mathbf g$-vectors defined in \Cref{section: generalized cluster algebra} have the same recursion as the vectors $\mathbf r_{i;t}$.
\end{proof}

The following corollary is an easy consequence of the sign-coherence of $\mathbf c$-vectors showed in the proof of \Cref{thm: g-vectors are rays of stability chambers}.
\begin{corollary}\label{cor: c vectors determine b matrix}
The matrix $B(t) = (b_{ij}(t))$ can be recovered from the $\mathbf c$-vectors by the formula
\[
b_{ij}(t) = \mathbf c_{i;t} \cdot \overline B(\kappa) \cdot (r_j\mathbf c_{j;t}^T).
\]
\end{corollary}

\begin{definition}\label{def: cluster cones}
We define $\Delta^+_\kappa$ to be the set of the cones $\{ \mathcal C_t = \mathcal C^{\kappa;t_0}_t \mid t\in \mathbb T_n\}$. We note that two cones $\mathcal C_t$ and $\mathcal C_s$ are regarded equal if they are the same subset of $M_\mathbb R$.
\end{definition}

\begin{remark}\label{rmk: chamber and intersection}
By \Cref{prop: function value of phi}, the interior of $\mathcal C_t$ has the characterization of being a connected component of $M_\mathbb R\setminus \operatorname{supp}(\Phi_\kappa)$ (thus shall be called a \emph{chamber}). Two such cones $\mathcal C_t$ and $\mathcal C_s$ are hence either identical or their interiors do not intersect. When $t\frac{k}{\quad\quad}t'$ in $\mathbb T_n$, we have
\[
\mathcal C_t \cap \mathcal C_{t'} = \overline{\mathfrak d}_{k;t} = \overline{\mathfrak d}_{k;t'}
\]
from their constructions. In this case, we say that they are related by mutation at their $k$-th facet. There is then an $n$-regular graph $\mathbf E(\Delta_\kappa^+)$ whose vertex set is $\Delta_\kappa^+$ such that two cones are joint by an edge if and only if they intersect in codimension one.
\end{remark}

Let $\sigma = \kappa(t)$ be the labeled triangulation associated to $t\in \mathbb T_n$ (with $\kappa$ associated to $t_0\in \mathbb T_n$). Applying the previous constructions, but with $t$ as a new root of $\mathbb T_n$, and with $\sigma$ sitting on this new root, we obtain $\Delta_\sigma^+=\{ \mathcal C^{\sigma;t}_s \mid s\in \mathbb T_n\}$ and the $n$-regular graph $\mathbf E(\Delta_\sigma^+)$. We now establish the concrete relation between this data with $t$ as root, and the data $\Delta_\kappa^+$ and $\mathbf E(\Delta_\kappa^+)$ with $t_0$ as root. Recall the notations at the beginning of \Cref{subsection: cluster chamber structure}. Consider the piecewise linear map
\[
T_{t}^{\kappa;t_0} \coloneqq T_{k_1;\kappa_1} \circ T_{k_2; \kappa_2} \circ \cdots\circ T_{k_l;\sigma} \colon \mathbb R^n\rightarrow \mathbb R^n.
\]

\begin{proposition}\label{prop: isomorphism of graphs of chambers}
The map $T_{t}^{\kappa; t_0}\colon \mathbb R^n \rightarrow \mathbb R^n$ sends the cones in $\Delta_\sigma^+$ to $\Delta_\kappa^+$ bijectively such that for any $s\in \mathbb T_n$, we have
\begin{equation}\label{eq: pull back of chamber}
T_{t}^{\kappa;t_0}(\mathcal C_{s}^{\sigma;t}) = \mathcal C_{s}^{\kappa;t_0}.    
\end{equation}
Moreover, it induces an isomorphism between the graphs 
\[
T_{t}^{\kappa;t_0}\colon \mathbf E(\Delta_\sigma^+) \rightarrow \mathbf E(\Delta_\kappa^+)
\]
\end{proposition}

\begin{proof}
By definition, $\mathcal C_{s}^{\sigma;t} = T_{s}^{\sigma;t}(\mathcal C^+)$ and $\mathcal C_{s}^{\kappa;t_0} = T_{s}^{\kappa;t_0}(\mathcal C^+)$. Since $T_{t}^{\kappa;t_0}\circ T_{s}^{\sigma;t} = T_{s}^{\kappa;t_0}$, the equation (\ref{eq: pull back of chamber}) follows. For the rest of the statement, it suffices to show the case where $\sigma$ and $\kappa$ are related by a single flip. By \Cref{prop: stability after mutation}, the map $T_t^{\kappa;t_0}$ sends chambers to chambers with the inverse being $T_{t_0}^{\sigma;t}$. Here $\Delta_\kappa^+$ is in fact the subset of chambers that are connected to $\mathcal C^+$ via mutations at facets; see \Cref{rmk: chamber and intersection}. Hence the bijection follows. The graph isomorphism follows from that $T_t^{\kappa;t_0}$ is a bijection on vertices and is also a local isomorphism between graphs.
\end{proof}

\section{\texorpdfstring{The motivic Hall algebra of $(Q,I)$}{The motivic Hall algebra of (Q, I)}}\label{section: hall algebras}
In this section, we briefly review the motivic Hall algebra (as introduced by Joyce \cite{joyce2007configurations}) of a quiver with relations following \cite{bridgeland2012introduction} and \cite[Section 5]{bridgeland2016scattering}. We note that quivers may have oriented cycles such as loops and two-cycles.

\subsection{Motivic Hall algebras}
Let $A$ be a finite dimensional algebra of the form $\mathbb C\langle Q\rangle/I$ with $I$ admissible. Recall the notations below \Cref{cor: stability after mutation}. There is an algebraic $\mathbb C$-stack $\mathfrak M$ (for example, as defined in \cite[Section 4.2]{bridgeland2016scattering}) that parametrizes all objects in $\modu A$. The stack $\mathfrak M$ can be decomposed into a disjoint union
$\mathfrak M = \bigsqcup_{d\in N^\oplus} \mathfrak M_d$ where each $\mathfrak M_d$ is an open and closed substack parametrizing objects of dimension vector $d\in N^\oplus$; see \cite[Lemma 4.1]{bridgeland2016scattering}.

There is another algebraic stack $\mathfrak M^{(2)}$ parametrizing short exact sequences $0\rightarrow E_1 \rightarrow E_2 \rightarrow E_3 \rightarrow 0$ in $\modu A$ up to certain equivalence (\cite[Section 4.4]{bridgeland2016scattering}). There are morphisms of stacks $\pi_i \colon \mathfrak M^{(2)}\rightarrow \mathfrak M$ sending a short exact sequence to its constituent $E_i$ for $i = 1, 2 ,3$. In the following diagram, the morphism $(\pi_3, \pi_1)$ is of finite type and $\pi_2$ is representable and proper (\cite[Lemma 4.5]{bridgeland2016scattering}).
\[
\begin{tikzcd}
\mathfrak M^{(2)} \ar[r, "\pi_2"] \ar[d, "{(\pi_3, \pi_1)}"] & \mathfrak M \\
\mathfrak M \times \mathfrak M
\end{tikzcd}
\]

Let $K(\mathrm{St}/\mathfrak M)$ be the relative Grothendieck group of finite type $\mathbb C$-stacks (with affine stabilizers) over $\mathfrak M$ as in \cite[Definition 5.1]{bridgeland2016scattering}. It is naturally a module over the Grothendieck ring $K(\mathrm{St/\mathbb C})$. The elements in $K(\mathrm{St}/\mathfrak M)$ are of the form $[X\xrightarrow{f} \mathfrak M]$ where $X$ is in $\mathrm{St}/\mathbb C$ and $f$ is a morphism of $\mathbb C$-stacks. There is a convolution product $*$ on $K(\mathrm{St}/\mathfrak M)$ defined as
\begin{equation}
[X\xrightarrow{f} \mathfrak M] * [Y\xrightarrow{g} \mathfrak M] = [Z \xrightarrow{\pi_2\circ h}\mathfrak M\times \mathfrak M]    
\end{equation}
where $h$ fits into the Cartesian square
\[
\begin{tikzcd}
Z \ar[d]\ar[r, "h"] & \mathfrak M^{(2)} \ar[r, "\pi_2"] \ar[d, "{(\pi_3, \pi_1)}"] & \mathfrak M\\
\mathfrak M \times \mathfrak M \ar[r] & \mathfrak M \times \mathfrak M
\end{tikzcd}.
\]
This product $*$ is $K(\mathrm{St}/\mathbb C)$-linear and makes $K(\mathrm{St}/\mathfrak M)$ into an associative $K(\mathrm{St}/\mathbb C)$-algebra; see \cite[Section 4]{bridgeland2012introduction}. The algebra $(K(\mathrm{St}/\mathfrak M), *)$ is the \emph{motivic Hall algebra} of $(Q,I)$ which we denote by $H = H(Q,I)$.

The algebra $H$ is $N^\oplus$-graded as $H  =  \bigoplus_{d\in N^\oplus} H_d$ where $H_d$ is the submodule generated by $[X\xrightarrow{f}\mathfrak M]$ such that $f$ factors through the inclusion $\mathfrak M_d\subset \mathfrak M$. The unit is $\mathbf 1 = [\mathfrak M_0 \subset \mathfrak M]$ in degree $0$.

We complete $H$ with respect to $N^+\subset N^\oplus$ obtaining the completed motivic Hall algebra $\hat H = \prod_{d\in N^\oplus} H_d$. Consider the $N^+$-graded Lie subalgebra (under the commutator bracket) $$ \mathfrak g_{\mathrm{Hall}} = H_{>0} \coloneqq \bigoplus_{d\in N^+} H_d\subset H $$ and its completion $\hat {\mathfrak g}_{\mathrm{Hall}} = \prod_{d\in N^+} H_d$. There is a corresponding pro-unipotent group $\hat G_\mathrm{Hall}$ whose elements are formal symbols $\exp(x)$, which are in one-to-one correspondence with $x\in \hat {\mathfrak g}_{\mathrm{Hall}}$. Note that there is an embedding $\hat G_\mathrm{Hall} \hookrightarrow \hat H$ by sending $\exp(x)\in \hat G_\mathrm{Hall}$ to the exponential series $1 + x + x^2/2! + \cdots + x^k/k! + \cdots \in \hat H$. This embedding respects multiplications because the multiplication in $\hat G_{\mathrm{Hall}}$ is defined through the Baker--Campbell--Hausdorff formula.

\subsection{Integration map}
Let $K(\mathrm{Var}/\mathbb C)$ be the Grothendieck ring of varieties over $\mathbb C$. Consider the \emph{regular subring}
\[
K_\mathrm{reg}(\mathrm{St}/\mathbb C)\coloneqq K(\mathrm{Var}/\mathbb C)[\mathbb L^{-1}, [\mathbb P^n]^{-1} \mid n\in \mathbb N] \subset K(\mathrm{St}/\mathbb C)
\]
where $\mathbb L$ denote the class of the affine line $\mathbb A^1$ and $\mathbb P^n$ the projective space of dimension $n$ in $K(\mathrm{St}/\mathbb C)$. We define $H_\mathrm{reg}$ to be the submodule of $H$ generated over $K_\mathrm{reg}(\mathrm{St}/\mathbb C)$ by elements $[X\rightarrow \mathfrak M]$ with $X$ a variety.

\begin{theorem}[{\cite[Theorem 5.1]{bridgeland2012introduction}}]
The submodule $H_\mathrm{reg}$ is closed under the convolution product, thus a $K_{\mathrm{reg}}(\mathrm{St/\mathbb C})$-algebra. Moreover, the quotient $H_\mathrm{reg}/(\mathbb L - 1) H_\mathrm{reg}$ is a commutative.
\end{theorem}

It follows from the above theorem that $H_\mathrm{reg}$ can be equipped with the Poisson bracket
\begin{equation}\label{eq: poisson bracket on reg}
\{f,g\} = (\mathbb L-1)^{-1}(f*g-g*f).    
\end{equation}
Thus $\mathfrak g_\mathrm{reg} \coloneqq (H_\mathrm{reg})_{>0}/(\mathbb L - 1) \subset H$ is a Lie algebra with the commutator bracket. Denote by $\hat {\mathfrak g}_\mathrm{reg}$ its completion and $\hat G_\mathrm{reg}$ the corresponding pro-unipotent group.

Now we turn to the situation where $I = \partial S$ is the Jacobian ideal of some potential $S$ of $Q$. Consider the group algebra $\mathbb C[N]$ with the Poisson bracket $\{ -, - \}$ defined by
\[
\{ y^{n_1}, y^{n_2} \} = \{ n_1, n_2 \} \cdot y^{n_1 + n_2}
\]
where $\{ - , - \}$ on the right also denotes the skew-symmetric bilinear form on $N$ such that $$\{ e_i, e_j \} = |\{a\in Q_1 \mid t(a) = i,\ h(a) = j\}| - |\{a\in Q_1 \mid t(a) = j,\ h(a) = i\}|.$$
Let $\mathfrak g = \mathbb C[N^+]$ be the $N^+$-graded Lie subalgebra of $\mathbb C[N]$ and $\hat {\mathfrak g}$ its completion. Denote by $\hat G$ the pro-unipotent group of $\hat {\mathfrak g}$.

\begin{theorem}[{\cite[Theorem 11.1]{bridgeland2016scattering}}]\label{thm: integration map}
When $I$ is the Jacobian ideal of a potential $S$ which is minimal (i.e. every cycle has length at least 3), there is a Poisson homomorphism (the integration map)
    \[
    I \colon H_\mathrm{reg} \rightarrow \mathbb C[N^\oplus]
    \]
    such that $I([X\xrightarrow {f} \mathfrak M_d]) = \chi (X) \cdot y^d$ where $\chi(\cdot)$ takes the Euler characteristic of a complex variety.
\end{theorem}

\begin{proof}
Note that the quiver $Q$ may have loops or two-cycles which are ruled out in the assumption of \cite[Theorem 11.1]{bridgeland2016scattering}. However, in the presence of loops or two-cycles the proof (see also the argument in \cite[Theorem 5.1]{bridgeland2012introduction}) still works since in this more general situation, for any $M$ and $N$ in $\modu A$, the following identity still holds
\[
\hom(M, N) - \operatorname{ext}^1(M, N) - (\hom(N, M) - \operatorname{ext}^1(N, M)) = -\{ \dim M, \dim N\}.
\]
A direct proof of the identity can be found in \cite[Theorem 7.6]{joyce2012theory}.
\end{proof}

The Poisson homomorphism in the above theorem then induces the following Lie algebra homomorphism
\[
I\colon \hat{\mathfrak g}_\mathrm{reg}\rightarrow \hat {\mathfrak g},\quad [X \xrightarrow{f}\mathfrak M_d]/(\mathbb L - 1) \mapsto \chi(X) \cdot y^d,
\]
which further induces a group homomorphism $I\colon \hat G_\mathrm{reg} \rightarrow \hat G$. By abuse of notation, we use the same letter $I$ for various maps induced by the integration map. However there should not be any ambiguity when the argument of $I$ is specified.

Now consider the tensor product algebra
\[
B \coloneqq \mathbb C[N^\oplus] \otimes_\mathbb C \mathbb C[M]
\]
with a monomial denoted by $y^nx^m$. We extend the previously defined Poisson bracket $\{-,-\}$ on $\mathbb C[N^\oplus]$ to $B$ by
\[
\{y^n, x^m\} = m(n) y^nx^m,\quad \{x^{m_1}, x^{m_2}\} = 0.
\]
The Poisson algebra $B$ is $N^\oplus$-graded where $B_n = y^n\cdot \mathbb C[M]$ for any $n\in N^\oplus$. Denote the corresponding completion by $\hat B \coloneqq \prod_{n\in N^\oplus} B_n$. Any $a\in \hat {\mathfrak g}$ acts on $\hat B$ by derivations $a(b) = \{a, b\}$ for $b\in \hat B$. The group element $\exp (a) \in \hat G$ acts on $\hat B$ by the automorphism $\exp(\{a,-\})\in \Aut(\hat B)$ $$\exp(\{a, -\})(b) \coloneqq b + \{a, b\} + \{a,\{a,b\}\}/2 + \dots + \{a,\{a,\cdots \{a,b\}\cdots\}/k!+\cdots.$$

We define another non-commutative $N^\oplus$-graded algebra
\[
C \coloneqq H_{\mathrm{reg}} \otimes_{\mathbb C} \mathbb C[M] = \bigoplus_{d\in N^\oplus} (H_{\mathrm{reg}})_d\otimes_\mathbb C \mathbb C[M]
\]
extending the multiplicative structures by $$a_d * x^m = \mathbb L^{m(d)} x^m * a_d$$ for $a_d\in H_d$, $d\in N^\oplus$ and $m\in M$. We equip $C$ with the following Poisson bracket
\[
\{a,b\} = [a,b]/(\mathbb L - 1)
\]
extending the one on $H_\mathrm{reg}$. Then the $N^\oplus$-graded Poisson homomorphism (the integration map in \Cref{thm: integration map}) $I\colon H_\mathrm{reg} \rightarrow \mathbb C[N^\oplus]$ extends to $I\colon C\rightarrow B$ (as well as to the completions $I\colon \hat C\rightarrow \hat B$).

For an element $\exp(a)\in \hat G_\mathrm{reg}$, the following lemma describes how $I(\exp(a))\in \hat G$ acts on $\hat B$ as an automorphism. Note that via the embedding $\hat G_\mathrm{Hall}\subset \hat H$, $\exp(a)$ can be viewed as an invertible element in $\hat H$.

\begin{lemma}\label{lemma: Poisson homomorphism}
For any $a\in \hat {\mathfrak g}_\mathrm{reg}$ and $x\in \mathbb C[M]$, we have
\[
\exp(a)*x*\exp(a)^{-1} = \exp\left(\{(\mathbb L-1)a, - \}_{\hat C}\right)(x)\in \hat C
\]
and thus $I(\exp(a))$ acts on $x$ by
\[
\exp(\{I(a),-\}) (x) = I\left(\exp(a)*x*\exp(a)^{-1}\right)\in \hat B.
\]
\end{lemma}

\begin{proof}
First we have $\exp(a)*x*\exp(a)^{-1} = \exp([a,-])(x)$ where $[-,-]$ is the commutator on $\hat H\otimes_{\mathbb C}\mathbb C[M]$. Since $[a,-] = \{(\mathbb L-1)a, -\}_{\hat C}$ for $a\in \hat {\mathfrak g}_\mathrm{reg}$, we have the first equality and in particular it is in $\hat C$.

By definition, $I(\exp(a))$ acts on $\hat B$ by $\exp(\{I(a), -\})$. Let $\bar a = (\mathbb L-1) a \in \hat H_\mathrm{reg}$. Then we can regard $I(a)\in \hat g$ as $I(\bar a)$ where by abuse of notation the later $I$ denotes the integration $I\colon \hat H_\mathrm{reg}\rightarrow \widehat{\mathbb C[N^\oplus]}$. Thus we have
\[
\exp(\{I(\bar a),-\}) (x) = I\left(\exp\left(\{\bar a, -\}_{\hat C}\right)(x)\right)\in \hat B
\]
since $I\colon \hat C \rightarrow \hat B$ is a Poisson homomorphism.
\end{proof}

\section{\texorpdfstring{Scattering diagrams of $\mathcal P(\kappa)$}{Scattering diagrams of P(kappa)}}\label{section: scattering diagram}

Bridgeland \cite{bridgeland2016scattering} introduced the \emph{Hall algebra scattering diagram} for any quiver with relations $(Q,I)$, and in the case where $I$ is a Jacobian ideal the \emph{stability scattering diagram} by applying the integration map (\Cref{thm: integration map}). The aim of this section is to apply this construction to the algebra $\mathcal P(\kappa)$, where we stress that in the presence of loops in the quiver, the construction remains valid.

Let $A = \mathbb C\langle Q\rangle/I$ be as in the last section and $\theta\in M_\mathbb R$ a stability condition. There is an open substack $\mathfrak M_\mathrm{ss}(\theta)\subset \mathfrak M$ parametrizing $\theta$-semistable modules; see \cite[Lemma 6.3]{bridgeland2016scattering}. Denote
\[
1_\mathrm{ss}(\theta) \coloneqq [\mathfrak M_\mathrm{ss}(\theta)\subset \mathfrak M] \in \hat H(Q,I).
\]
Note that the degree zero term of $1_\mathrm{ss}(\theta)$ is the unit $\mathbf 1$ since the zero module is $\theta$-semistable. So the element $1_\mathrm{ss}(\theta)$ belongs to the group $\hat G_{\mathrm{Hall}} = \mathbf 1 + \hat H(Q,I)_{>0}$ where $\hat H(Q,I)_{>0} = \prod_{d\in N^+} H(Q, I)_d$.

\begin{definition}\label{def: hall algebra sd}
Following \cite{bridgeland2016scattering}, we define the \emph{Hall algebra scattering diagram} of $A$ to be the function
\[
\phi_A^\mathrm{Hall} \colon M_\mathbb R \rightarrow \hat G_\mathrm{Hall}, \quad \phi_A^\mathrm{Hall}(\theta) = 1_\mathrm{ss}(\theta) \in \hat G_\mathrm{Hall}.
\]
\end{definition}

\begin{remark}
The function $\phi_A^\mathrm{Hall}$ can be viewed as an algebraic avatar of $\Phi_A\colon M_\mathbb R\rightarrow \mathcal W(A)$. When $\Phi_A(\theta) = 0$, we have $\phi_A^\mathrm{Hall}(\theta) = \mathbf 1$.
\end{remark}

The following \emph{absence of poles} theorem is due to Joyce \cite{joyce2007configurations}. Here we present the form taken in \cite[Theorem 5.3]{bridgeland2016scattering} and \cite[Theorem 6.3]{bridgeland2011hall}.

\begin{theorem}[Joyce]\label{thm: joyce absence of poles} For any $\theta\in M_\mathbb R$, we write $1_\mathrm{ss}(\theta)  = \mathbf 1 + \epsilon$ for some $\epsilon\in \hat H(Q,I)_{>0}$. Then the series
    \[
        \log(1_\mathrm{ss}(\theta)) \coloneqq \epsilon - \epsilon^2/2 + \epsilon^3/3 + \cdots + (-1)^{k-1} \epsilon^k/k + \cdots
    \]
    computes an element in $\hat{\mathfrak g}_\mathrm{reg}$. In other words, $1_\mathrm{ss}(\theta)$ belongs to $\hat G_{\mathrm{reg}}$.
\end{theorem}

Suppose that $I$ is a Jacobian ideal. Hence by \Cref{thm: integration map} the integration maps $I\colon \hat{\mathfrak g}_\mathrm{reg}\rightarrow \mathfrak g$ and $I\colon \hat G_{\mathrm{reg}} \rightarrow \hat G$ are valid. Furthermore, according to Joyce's \Cref{thm: joyce absence of poles}, the range of the Hall algebra scattering diagram
\[
\phi_A^\mathrm{Hall}\colon M_\mathbb R \rightarrow \hat G_\mathrm{Hall},\quad \theta\mapsto 1_\mathrm{ss}(\theta)
\]
is in the subgroup $\hat G_\mathrm{reg}\subset \hat G_\mathrm{Hall}$.

\begin{definition}
We define \emph{stability scattering diagram} for $A$ to be the function
\[
        \phi_A \coloneqq I \circ \phi_A^\mathrm{Hall} \colon M_\mathbb R\rightarrow \hat G,\quad \phi_A(\theta) = I(1_\mathrm{ss}(\theta))\in \hat G.
\]
\end{definition}

To describe the group element $\phi_A(\theta) = I(1_\mathrm{ss}(\theta))$ for some $\theta$, we express it by its Poisson action on the completion $\hat B$ of $B = \mathbb C[N^\oplus] \otimes_\mathbb C \mathbb C[M]$. Let $\exp(a)$ be in $\hat G_d \coloneqq \exp(\prod_{k=1}^\infty \mathfrak g_{kd})$ with $d$ primitive in $N^+$. It acts on $x^m$ by 
\begin{equation}\label{eq: wall-crossing function}
\exp(a)(x^m) = \exp(\{a,-\})(x^m) =  x^m \cdot f(y^d)^{m(d)}\in \hat B    
\end{equation}
for some power series $f(y)$ in a single variable $y$. In fact, it is easy to see that if 
\[
a = \sum_{k=1}^\infty a_k y^{kd}\in \prod_{k=1}^\infty \mathfrak g_{kd},
\]
then $f(y)$ is expressed as
\[
    f(y) = \exp\left( \sum_{k=1}^\infty ka_ky^{k}\right)\in \mathbb C[[y]].
\]
To compute the power series $f(y)$ for $\phi_A(\theta)$, we use the framed stability introduced below.

Let $m\in M$ such that $m(d)\geq 0$ for any $d\in N^\oplus$. Let $\theta\in M_\mathbb R$ be a King's stability condition. Denote by $P^m$ the projective $A$-module
\[
    P^m \coloneqq \bigoplus_{i = 1}^n P_i^{m(e_i)}.
\]

\begin{definition}
An $m$-framed $\theta$-stable $A$-module is $M\in \modu A$ with a morphism $\mathbf{fr}\colon P^m \rightarrow M$ (\emph{the framing}) such that
\begin{enumerate}
    \item $M$ is $\theta$-semistable;
    \item Any proper submodule $M'\subset M$ containing the image of $\mathbf{fr}$ satisfies $\theta(M')>0$.
\end{enumerate}
Two $m$-framed $A$-modules $M$ and $N$ are said to be equivalent if there exists an isomorphism $f\colon M\rightarrow N$ intertwining their framings. 
\end{definition}

There is a fine moduli scheme $F(d, m, \theta)$ whose $\mathbb C$-points parametrize $m$-framed $\theta$-stable $A$-modules of dimension $d$ up to equivalence; see \cite[Section 8.2]{bridgeland2016scattering} for a detailed description. In the case where $I=0$, these moduli schemes were studied by Engel and Reineke earlier in \cite{MR2511752}. Now we present Bridgeland's formula on the wall-crossing action using framed representations. A closely related formula but in a slightly different setting is obtained by Reineke for acyclic quivers in \cite{reineke2010poisson}.

\begin{theorem}[{\cite[Theorem 10.2 and (11.1)]{bridgeland2016scattering}}]\label{thm: bridgeland wall crossing}
Let $m\in M^\oplus$ and $d\in N^+$ be a primitive dimension vector. Let $\theta$ be a general point in $d^\perp\subset M_\mathbb R$ (such that any $n\in N$ orthogonal to $\theta$ is parallel to $d$). Then the action of $\phi_A(\theta)$ on $x^m\in \hat B$ is given by
\[
x^m \mapsto x^m \cdot \sum_{k=1}^\infty \chi(F(kd, m, \theta))y^{kd}.
\]
Here $\chi(X)$ takes the Euler characteristic of $X(\mathbb C)$ (in the analytic topology) for a $\mathbb C$-scheme $X$. 
\end{theorem}

\begin{remark}
Our setting for the above theorem is slightly more generalized than that of \cite{bridgeland2016scattering} as we allow quivers to have loops. However, as long as the ideal $I$ is the Jacobian ideal from a potential, the integration map $I$ is valid as noted in the proof of \Cref{thm: integration map}. The rest then carries over to the present setting.
\end{remark}

Next we apply the above theory to the case when $A = \mathcal P(\kappa) \coloneqq \mathbb C\langle Q(\kappa)\rangle/\langle \partial S(\kappa) \rangle$ coming from a triangulation $\kappa$ of some $\surf$. We write $\phi_\kappa$ for the function $\phi_{\mathcal P(\kappa)}$ for simplicity. The motivic Hall algebra is written as $H(\modu \mathcal P(\kappa))$.

\begin{theorem}\label{thm: cluster complex structure} The stability scattering diagram $\phi_\kappa\colon M_\mathbb R\rightarrow \hat G$ has the following properties.
\begin{enumerate}
    \item Let $\mathcal C$ be any chamber in $\Delta_\kappa^+$. Then the function $\phi_\kappa$ has constant value $\mathbf 1$ in $\mathcal C^\circ$ the interior of $\mathcal C$;
    \item Let $\mathfrak d$ be any (relatively open) facet of $\mathcal C$. Then the power series $f(y^d)$ corresponding to $\phi_\kappa(\theta) = I(1_\mathrm{ss}(\theta))$ for any $\theta \in \mathfrak d$ (as in (\ref{eq: wall-crossing function})) is actually a polynomial
\begin{enumerate}
    \item $1 + y^d$ if $\Phi_\kappa (\theta) \cong \modu \mathbb C$;
    \item $1 + y^d + y^{2d}$ if $\Phi_\kappa (\theta) \cong \modu \mathbb C[\varepsilon]/\varepsilon^2$
\end{enumerate}
where $d\in N^+$ is the dimension vector of the generator of $\Phi_\kappa (\theta)$.
\end{enumerate}
 
\end{theorem}

\begin{proof}
Part (1) is true because $\Phi_\kappa(\theta) = 0$ for any $\theta\in \mathcal C^\circ$ by \Cref{prop: function value of phi}. Thus we have $\phi_\kappa(\theta) = I([\mathfrak M_0\subset \mathfrak M]) = \mathbf 1\in \hat G$.

For part (2), suppose that the facet $\mathfrak d$ is $\mathfrak d_{i;t}$ of $\mathcal C_t$. Then by \Cref{prop: function value of phi}, there is an embedding $\iota\colon \modu H_i \rightarrow \modu \mathcal P(\kappa)$ such that the image is exactly $\Phi_\kappa(\theta)$. Here $H_i$ is the algebra $\mathbb C[\varepsilon]/\varepsilon^2$ if $i$ is pending or $\mathbb C$ if $i$ is not pending. The simple module $S_i\in \modu H_i$ embeds as $\iota(S_i)$ the generator of $\Phi_\kappa(\theta)$. The embedding induces an injective algebra homomorphism $\iota\colon \hat H(\modu H_i) \rightarrow \hat H(\modu \mathcal P(\kappa))$ such that $\iota([\mathfrak M \xrightarrow{\mathrm{id}}\mathfrak M]) = 1_\mathrm{ss}(\theta)$. Meanwhile the algebra $\mathbb C[[y]]$ embeds in $\mathbb C[[N^\oplus]]$ by sending $y$ to $y^d$ where $d\in N^+$ is the dimension vector of $\iota(S_i)$. We further have
\[
I(\log [\mathfrak M\xrightarrow{\mathrm{id}}\mathfrak M])|_{y = y^d} = I(\log (1_\mathrm{ss}(\theta)))\in \hat {\mathfrak g}.
\]
Therefore to compute the power series $f(y^d)$ in $y^d$ for $I(1_\mathrm{ss}(\theta))\in \hat G$, we only need to compute the one of $I([\mathfrak M\xrightarrow{\mathrm{id}} \mathfrak M])$ for $\hat H(\modu H_i)$ as a power series in a single variable $y$, and replace $y$ with $y^d$. Thus it amounts to the calculation for the action of $I([\mathfrak M\xrightarrow{\mathrm{id}}\mathfrak M])$ on the Poisson algebra $\mathbb C[[y]]\otimes_\mathbb C \mathbb C[x]$ with $\{y,x\} = yx$.

By \Cref{thm: bridgeland wall crossing}, the action of $I(1_\mathrm{ss}(\theta))$ on a monomial $x^m$ is given by
\[
x^m \mapsto x^m \cdot \sum_{\theta(d) = 0} \chi (F(d, m, \theta)) \cdot y^{d}.
\]
Now we are in $\modu \mathbb C[\varepsilon]/\varepsilon^2$ or $\modu \mathbb C$, and $[\mathfrak M\xrightarrow{\mathrm{id}} \mathfrak M] = 1_\mathrm{ss}(0)$. Take $m = e_1^*$ and $\theta = 0$ is trivial. The only dimension vectors that can appear are the multiples of $e_1$. In the case where $i$ is not pending, it follows from the exact same computation in \cite[Lemma 11.4]{bridgeland2016scattering} that the action is given by $x\mapsto x (1 + y)$. So in this case $f(y^d) = 1 + y^d$. Suppose that $i$ is pending. Then we have
\begin{enumerate}
    \item For $d=0$, the only $e_1^*$-framed $0$-stable module is the $0$ module (with the framing $\mathbf {fr}$ being the $0$ map), thus $\chi(F(0, e_1^*, 0)) = \chi(\mathrm{pt}) =1$.
    \item For $d=e_1$, up to equivalence there is only the simple module $\mathbb C$ with the framing being the quotient $\mathbf {fr}\colon H_i \rightarrow \mathbb C$, $\varepsilon\mapsto 0$. Thus $\chi(F(e_1, e_1^*, 0)) = \chi(\mathrm{pt}) = 1$.
    \item For $d=2e_2$, up to isomorphism there are two $H_i$-modules $H_i$ and $\mathbb C\oplus\mathbb C$ where the later can never be $e_1^*$-framed $0$-stable because the image of a framing is always proper (thus breaking the stability condition). The former $H_i$ with the framing $\mathbf{fr}\colon H_i\rightarrow H_i$ being an isomorphism is $e_1^*$-framed $0$-stable but every such is equivalent to each other. So we have $\chi(F(2e_1, e_1^*, 0) = \chi(\mathrm{pt}) = 1$.
    \item For all other $ke_1$ for $k\geq 3$, the moduli scheme $F(ke_1, e_1^*, \theta)$ is empty. This is because the image of any framing must be a proper submodule, which however violates the stability.
\end{enumerate}
Now we conclude that in the case where $i$ is pending, the power series $f(y^d)$ is $1+y^d+y^{2d}$.
\end{proof}

We associate an element $\mathfrak p^\mathrm{Hall}_t\in \hat G_\mathrm{Hall}$ to each $t\in \mathbb T_n$ and will see that its integration $\mathfrak p_t \coloneqq I(\mathfrak p_t^\mathrm{Hall})\in \hat G$ is closely related to the mutations of cluster variables from $\mathbf x$ to $\mathbf x_t$. Recall the notations in \Cref{con: semistable on a cluster wall}. Define
\[
\mathfrak p_t^\mathrm{Hall} \coloneqq \phi^\mathrm{Hall}_\kappa(\mathfrak d^{l}_{k_l})^{\epsilon_l} * \cdots * \phi^\mathrm{Hall}_\kappa(\mathfrak d^{1}_{k_1})^{\epsilon_1} \in \hat G_{\mathrm{Hall}}.
\]
where $\epsilon_i \coloneqq \operatorname{sgn}(\mathbf n_{k_i}^i)\in \{\pm\}$ (as these must be sign-coherent).

Let $\theta \in \mathcal C_t^\circ$. We define the following \emph{torsion class}
\[
\mathcal T(\theta) \coloneqq \{ M\in \modu \mathcal P(\kappa)\mid \text{any quotient $N$ of $M$ satisfies $\theta(N)<0$}\}
\]
as a full subcategory of $\modu \mathcal P(\kappa)$. There is an open substack $\mathfrak M_{\mathcal T(\theta)}\subset \mathfrak M$ parametrizing objects in $\mathcal T(\theta)$; see its definition in \cite[Lemma 6.3 and Lemma 6.6]{bridgeland2016scattering}. Denote the element $1_{\mathcal T(\theta)} \coloneqq [\mathfrak M_{\mathcal T(\theta)}\subset \mathfrak M]\in \hat G_{\mathrm{Hall}} \subset \hat H(\modu \mathcal P(\kappa))$.

\begin{lemma}\label{lemma: consistency and torsion class}
We have $1_{\mathcal T(\theta)} = \mathfrak p_{t}^\mathrm{Hall}\in \hat G_\mathrm{Hall}$. Thus the element $\mathfrak p_t^{\mathrm{Hall}}$ only depends on the cone $\mathcal C_t$ and so is $\mathfrak p_t\in \hat G$.
\end{lemma}

\begin{proof}
This lemma completely follows from Bridgeland's theorem \cite[Theorem 6.5]{bridgeland2016scattering} that $\Phi_\kappa^{\mathrm{Hall}}$ is \emph{consistent} in the following sense. For any $\theta_i\in M_\mathbb R$ in the interior of some $\mathcal C_i\in \Delta_\kappa^+$ for $i=1,2$, there is an associated element
\[
\phi_\kappa^{\mathrm{Hall}}(\theta_1, \theta_2) \coloneqq 1_{\mathcal T(\theta_1)} * 1_{\mathcal T(\theta_2)}^{-1}\in \hat G_{\mathrm{Hall}}.
\]
In the meantime, one could consider a path $\gamma$ going from $\mathcal C_1$ to $\mathcal C_2$ by crossing facets of cones in $\Delta_\kappa^+$. Since $\mathbb E(\Delta_\kappa^+)$ is connected, such paths always exist. Consider the product
\[
\mathfrak p_\gamma \coloneqq \phi_\kappa^{\mathrm{Hall}}(\mathfrak d_1)^{\epsilon_1}*\phi_\kappa^{\mathrm{Hall}}(\mathfrak d_2)^{\epsilon_2}* \cdots *\phi_\kappa^{\mathrm{Hall}}(\mathfrak d_l)^{\epsilon_l}\in \hat G_{\mathrm{Hall}}.
\]
Here from left to right of the product, the sequence $\mathfrak d_i$ record the facets that $\gamma$ crosses in order and we determine the signs $\epsilon_i = +1$ if the cross is in the same direction as the positive normal vector of $\mathfrak d_i$ and otherwise $\epsilon_i = -1$. A direct implication of the consistency of $\Phi_\kappa$ is that
\[
\mathfrak p_\gamma = \phi_\kappa^{\mathrm{Hall}}(\theta_1, \theta_2)
\]
and in particular this does not depend on the path $\gamma$ (see \cite[(6.7)]{bridgeland2016scattering}). The current lemma corresponds to the particular situation where $\theta_1 = \theta$ and $\theta_2$ is some interior point in $\mathcal C^+$.

\end{proof}

\begin{proposition}\label{cor: f-polynomial from stability}
We have the identity for any $t\in \mathbb T_n$ and any $i$ 
\[
\mathfrak p_t(x^{\mathbf g_{i;t}}) = x^{\mathbf g_{i;t}}\cdot F_{i;t}(y) \in \mathbb C[N^\oplus]\otimes_\mathbb C \mathbb C[M]
\]
where $\mathbf g_{i;t}$ is the $\mathbf g$-vector associated to $x_{i;t}$ and $F_{i;t}(y)$ is the corresponding $F$-polynomial.
\end{proposition}

\begin{proof}
    Note that the action of $\mathfrak p_t$ on $x^m$ is taken in the algebra $\hat B$. The result follows from the chamber structure and the description of wall-crossings of the stability scattering diagram $\phi_\kappa$ from \Cref{thm: cluster complex structure}. We prove it by induction on the distance from $t$ to $t_0$. Here the Poisson bracket on $\mathbb C[N^\oplus]$ is given by the skew-symmetric matrix $\overline{B} = \overline{B}(\kappa)$.
    
    Suppose the identity is true for $t$ and let $t'\frac{k}{\quad\quad}t$ be an edge in $\mathbb T_n$. Notice that since all wall-crossings act on $x^{\overline{B}n}y^n$ for any $n\in N^\oplus$ trivially, the identity implies for any $i$,
    \[
    \mathfrak p_t(y^{\mathbf c_{i;t}}) = y^{\mathbf c_{i;t}} \prod_{j = 1}^n F_{j;t}^{\overline{b}_{ji}(t)}(y).
    \]
    We show in the following the identity is true for $t'$ and $k$ in the case where $k$ is pending and the vector $\mathbf c_{k;t}$ is positive; other cases are similar. In this case, we have the recursion
    \[
    \mathbf g_{k;t'} = -\mathbf g_{k;t} + \sum_{i\neq k} [-b_{ik}(t)]_+ \mathbf g_{i;t}.
    \]
    Now since $\mathfrak p_{t'} = \phi_\kappa(\mathfrak d_{k;t}) \cdot \mathfrak p_t$, we have
    \[
    \mathfrak p_{t'}(x^{\mathbf g_{k;t'}}) = \mathfrak p_t \left(\phi(\mathfrak d_{k;t}) \left (x^{-\mathbf g_{k;t}}\cdot \prod_{i\neq k} x^{[-b_{ik}(t)]_+ \mathbf g_{i;t}} \right) \right).
    \]
    The action of $\phi(\mathfrak d_{k;t})$ on $x^{\mathbf g_{i;t}}$ for $i\neq k$ is trivial because $\mathbf c_{k;t}$ is orthogonal to any of those $\mathbf g_{i;t}$. Thus we have
    \begin{align*}
    \mathfrak p_{t'}(x^{\mathbf g_{k;t'}}) & = \mathfrak p_t\left(\phi(\mathfrak d_{k;t})(x^{-\mathbf g_{k;t}})\right)\cdot \prod_{i\neq k}\mathfrak p_t(x^{\mathbf g_{i;t}})^{[-b_{ik}^t]_+}\\
    &=\mathfrak p_t\left(x^{-\mathbf g_{k;t}}(1 + y^{\mathbf c_{k;t}} + y^{2\mathbf c_{k;t}})\right) \prod_{i\neq k} x^{[-b_{ik}(t)]_+\mathbf g_{i;t}}F_{i;t}^{[-b_{ik}(t)]_+}\ \text{(by \Cref{thm: cluster complex structure} and induction)}\\
    & = x^{\mathbf g_{k;t'}} \cdot F_{k;t}^{-1}\cdot \left(1 + y^{\mathbf c_{k;t}}\prod_{i\neq k}F_{i;t}^{\overline{b}_{ik}(t)} + y^{2\mathbf c_{k;t}}\prod_{i\neq k}F_{i;t}^{2\overline{b}_{ik}(t)}\right) \cdot \prod_{i\neq k}F_{i;t}^{[-b_{ik}(t)]_+}\\
    & = x^{\mathbf g_{k;t'}}\cdot F_{k;t'}.
    \end{align*}
    The last step uses the recursive definition of $F$-polynomials in \Cref{section: generalized cluster algebra}.
\end{proof}

\section{\texorpdfstring{$\tau$-tilting theory of $\mathcal P(\kappa)$}{tau-tilting theory of P(kappa)}}\label{section: tau tilting}

In this section we briefly review the $\tau$-tilting theory of Adachi--Iyama--Reiten \cite{adachi2014tau} and study the case of $\mathcal P(\kappa)$.

\begin{definition}\label{def: tau-rigid module}
	Let $A$ be a finite dimensional algebra over an algebraically closed field $\Bbbk$. A finitely generated $A$-module is said to be \emph{$\tau$-rigid} if $\Hom_A (M, \tau M) = 0$ where $\tau \colon \modu A \rightarrow \modu A$ is the \emph{Auslander--Reiten translation}.
\end{definition}

In what follows we assume that $A$ is of the form $\Bbbk \langle Q \rangle/I$ with $I$ an admissible ideal. Let $n = |Q_0|$.

\begin{definition}\ 
\begin{enumerate}
	\item A \emph{$\tau$-rigid pair} $(M , P)$ is a $\tau$-rigid module $M$ and a projective $A$-module such that 
	\[
	\Hom_A(P,M) = 0.
	\]
	\item A $\tau$-rigid pair is said to be \emph{support $\tau$-tilting} if $|M| + |P| = n$ where $|\cdot |$ counts the number of non-isomorphic indecomposable direct summands.
	
	\item A $\tau$-rigid pair is said to be \emph{almost complete support $\tau$-tilting} if $|M| + |P| = n-1$.
\end{enumerate}
\end{definition}

We say that an $A$-module $M$ is \emph{basic} if its indecomposable summands are pairwise non-isomorphic. We say a pair $(M,P)$ \emph{basic} if $M$ and $P$ are both basic. Denote by $\stau{A}$ the set of all isomorphism classes of basic support $\tau$-tilting pairs of $A$. It is clear that both $(A, 0)$ and $(0,A)$ are support $\tau$-tilting pairs.

For two $\tau$-rigid pairs $(M_1, P_1)$ and $(M_2, P_2)$, simply write $(M_1, P_1)\oplus (M_2, P_2)$ for $(M_1 \oplus M_2, P_1 \oplus P_2)$. We say $(M, P)$ indecomposable if it is of the form $(M, 0)$ or $(0,P)$ with $M$ or $P$ being indecomposable.

\begin{theorem}[{\cite[Theorem 0.4]{adachi2014tau}}]\label{thm: almost complete pair has two completions}
Any basic almost complete support $\tau$-tilting pair for $A$ is a direct summand of exactly two basic support $\tau$-tilting pairs.
\end{theorem}

The two support $\tau$-tilting pairs in \Cref{thm: almost complete pair has two completions} are said to be \emph{AIR-mutations} of each other. We use the notation $(N,Q)=\mu_k^{\operatorname{AIR}}(M,P)$ to indicate that the support $\tau$-tilting pairs $(M,P)$ and $(N,Q)$ are related by AIR-mutation of their $k^{\operatorname{th}}$ indecomposable summand.  One could regard AIR-mutation as an operation on $\operatorname{s\tau-tilt} A$. For example, write $(M,P)\in \operatorname{s\tau-tilt} A$ as $(M_1\oplus M_2, P)$ such that $M_2$ is indecomposable. Then by \Cref{thm: almost complete pair has two completions}, there exists either a unique $M_3\in \modu A$ non-isomorphic to $M_2$ or a unique $P_1\in \operatorname{proj} A$ such that $(M_1\oplus M_3, P)$ or $(M_1, P\oplus P_1)$ is support $\tau$-tilting.

From now on we will often use a single letter (e.g. $\mathcal M$ instead of $(M,P)$) for a $\tau$-rigid pair for simplicity. We shall call $(0,A)$ the \emph{initial} support $\tau$-tilting pair. Suppose that the vertex set $Q_0$ has been identified with $I = \{1,\dots, n\}$. Then the indecomposable summands of $(0,A)$ are $(0,P_i)$ for $i\in I$, thus indexed by $I$. For any $\mathcal M\in \operatorname{s\tau-tilt} A$ whose indecomposable summands are indexed by $I$, denote by $\mu_k(\mathcal M)$ its mutation at the corresponding summand. Immediately one sees there is an association of $\mathcal M_t\in \operatorname{s\tau-tilt} A$ to a vertex $t\in \mathbb T_n$ such that $\mathcal M_{t_0} = (0,A)$ and for $t\frac{k}{\quad\quad}t'$, $\mathcal M_t = \mu_k(\mathcal M_{t'})$. Denote by $\mathcal M_{i;t}$ the indecomposable summand of $\mathcal M_t$ of index $i$. We call $\mathcal M\in \operatorname{s\tau-tilt} A$ \emph{reachable} if it is isomorphic to some $\mathcal M_t$. The set of isomorphism classes of \emph{reachable} $\tau$-rigid pairs is denoted by $\rstau{A}$. A $\tau$-rigid pair is \emph{reachable} if all its indecomposable summands are that of some reachable $\mathcal M_t$.

Let $M \coloneqq K_0(\operatorname{proj} A)$ be the Grothendieck group of the category of finitely generated projective modules over $A$. It is a lattice with a basis consisting of the classes of the indecomposable projective modules 
\[
P_1,\ P_2,\dots,\ P_n.
\]
The lattice $M$ is naturally dual to $N \coloneqq K_0(\modu A)$ by the pairing
\[
K_0(\operatorname{proj} A) \times K_0(\modu A) \rightarrow \mathbb Z,\quad (P,T) \mapsto \hom (P, T).
\]
The basis $e_i \coloneqq S_i$ for $N$ is then dual to $e_i^* \coloneqq P_i$ for $M$.

For $T\in \modu A$, let $Q_1\rightarrow Q_0\rightarrow T\rightarrow 0$ be its minimal projective presentation. We define the \emph{$\mathbf g$-vector} of $T$ to be
\[
\mathbf g(T) \coloneqq Q_1 - Q_0\in M,
\]
identified with a vector in $\mathbb Z^n$ when expressed in the basis $\{e_i^*\mid i\in I\}$.

To each $\tau$-rigid pair $\mathcal M = (T, P)$, we associate a $g$-vector
\[
\mathbf g(\mathcal M)\coloneqq \mathbf g(T) - \mathbf g(P).
\]
Note that our convention is opposite to that in \cite{adachi2014tau} where they define $\mathbf g(\mathcal M)\coloneqq \mathbf g(P) - \mathbf g(T)$. In our convention, the $\mathbf g$-vector of $(0,P_i)$ is $e_i^*$.

\begin{theorem}[\cite{adachi2014tau}]\label{thm: AIR}\ 
\begin{enumerate}
    \item Any $\tau$-rigid pair is always a direct summand of some $\mathcal M \in \operatorname{s\tau-tilt} A$.
	\item The $\mathbf g$-vectors of indecomposable direct summands of any $\mathcal M \in \operatorname{s\tau-tilt} A$ form a $\mathbb Z$-basis of $M$.
	\item The map $\mathcal M\mapsto \mathbf g(\mathcal M)$ is injective from the set of (the isoclasses of) $\tau$-rigid pairs to $M$.
\end{enumerate}	
\end{theorem}

For $\mathcal M$ any $\tau$-rigid pair, define $\mathcal C_\mathcal M$ to be the closed rational polyhedral cone in $M_\mathbb R$ generated by the $\mathbf g$-vectors of its indecomposable summands. By \Cref{thm: AIR}, the cone $\mathcal C_\mathcal M$ is simplicial of maximal dimension if $\mathcal M$ is support $\tau$-tilting, and the collection of all these cones $\left\{\mathcal C_\mathcal M\mid \mathcal M\in \operatorname{s\tau-tilt} A\right\}$ together with all their faces form a simplicial fan in $M_\mathbb R$ considered in \cite{demonet2019tau}. Two maximal cones $\mathcal C_\mathcal M$ and $\mathcal C_{\mathcal M'}$ intersect at one of their common facet if they are related by a mutation. We denote the dual graph of the reachable component of this simplicial fan by $\mathbf E(\rstau{A})$. In other words, the vertices of $\mathbf E(\rstau{A})$ are reachable objects in $\rstau{A}$ up to isomorphism and two are joint by an edge if and only if they are obtained by a mutation from each other. Thus we call $\mathbf E(\rstau{A})$ the exchange graph of $\rstau{A}$.

The following important theorem of Br\"ustle, Smith and Treffinger relates the cones $\mathcal C_\mathcal M$ with stability conditions of $\modu A$.

\begin{theorem}[{\cite[Proposition 3.13 and Theorem 3.14]{brustle2019wall}}]\label{thm: BST}
The function $\Phi_A\colon M_\mathbb R \rightarrow \mathcal W(A)$ is constant in the interior of $\mathcal C_\mathcal M$ for any $\tau$-rigid pair $\mathcal M = (T,P)$ and has the value
\[
J_\mathcal M \coloneqq T^\perp \cap {^\perp\tau T} \cap P^\perp.
\]
Moreover, the subcategory $J_\mathcal M$ contains exactly $n - |T| - |P|$ non-isomorphic stable modules. In particular, $\Phi_A(\theta) = 0$ for any $\theta$ in the interior of $\mathcal C_\mathcal M$ for $\mathcal M\in \stau{A}$.
\end{theorem}

For a $\tau$-rigid pair $\mathcal M = (T, P)$, define the torsion class
\[
\operatorname{Fac} \mathcal M \coloneqq \{E\in \modu A \mid \text{$E$ is a quotient of $T^{\oplus l}$ for some $l\in \mathbb N$}\} \subset \modu A.
\]
We will need the following proposition in \cite{brustle2019wall} later.

\begin{proposition}[{\cite[Proposition 3.27]{brustle2019wall}}]\label{prop: two definitions of a torsion class}
For any $\tau$-rigid pair $\mathcal M = (T, P)$, the torsion class $\operatorname{Fac} \mathcal M$ is the same as $\mathcal T(\theta)$ for any $\theta$ in the interior of $\mathcal C_\mathcal M$.
\end{proposition}

Let $A$ be the algebra $\mathcal P(\kappa)$ associated to some labeled triangulation $\kappa$ of $\surf$. We consider the $\tau$-tilting theory and especially the graph $\mathbf E(\rstau{\mathcal P(\kappa)})$ of $\modu \mathcal P(\kappa)$. For $t\in \mathbb T_n$ and $i\in I$, denote 
\[
\mathcal M_t = \mathcal M_t^{\kappa;t_0}\quad \text{and} \quad \mathcal M_{i;t} = \mathcal M_{i;t}^{\kappa;t_0}
\]
the corresponding support $\tau$-tilting pair and indecomposable $\tau$-rigid pair.

\begin{theorem}\label{cor: g-vector equality}
For any $t\in \mathbb T_n$ and $i\in I$, the $\mathbf g$-vector of $\mathcal M_{i;t}$ is equal to the $\mathbf g$-vector of the cluster variable $x_{i;t}$ of $\mathcal A(\kappa)$, that is,
\[
\mathbf g(\mathcal M_{i;t}) = \mathbf g_{i;t}
\]
\end{theorem}

\begin{proof}
For $t=t_0$, the equality holds as $\mathbf g(\mathcal M_{i;t_0}) =  \mathbf g(0, P_i)= \mathbf g_{i;t_0} = e_i^*\in M$. We prove for the rest of $\mathbb T_n$ by induction on the distance to $t_0$. Suppose the equality is true for some $t\in \mathbb T_n$ and let $t'\frac{k}{\quad\quad} t$ be in $\mathbb T_n$.

Notice that by \Cref{thm: g-vectors are rays of stability chambers} the vectors $\mathbf g_{i;t}$ are the generators $\mathbf r_{i;t}$ of the rays of the cone $\mathcal C_t$. The equality $\mathbf g(\mathcal M_{i;t}) = \mathbf g_{i;t}$ implies $\mathcal C_{\mathcal M_t} = \mathcal C_t$. To prove $\mathbf g(\mathcal M_{k;t'}) = \mathbf g_{k;t'}$, it suffices to show that $\mathcal C_{\mathcal M_{t'}} = \mathcal C_{t'}$ as they already share a common facet $\overline{\mathfrak d}_{k;t}$ and $\mathbf g(\mathcal M_{k;t'})$ and $\mathbf g_{k;t'}$ are the integral generators of the corresponding rays that are not in $\overline{\mathfrak d}_{k;t}$.

By \Cref{thm: BST}, the interior of $\mathcal C_{\mathcal M_t}$ is a connected component of $\mathcal M_\mathbb R\setminus \operatorname{supp}(\Phi_\kappa)$, and by \Cref{cor: stability in a chamber} so is $\mathcal C_t$. As pointed out in \Cref{rmk: chamber and intersection}, the two cones $\mathcal C_{t'}$ and $\mathcal C_t$ intersect exactly on their $k$-th facet $\overline{\mathfrak d}_{k;t}$. The same is true for $\mathcal C_{\mathcal M_t}$ and $\mathcal C_{\mathcal M_{t'}}$ since by \Cref{thm: AIR} they intersect exactly at $\mathcal C_{\mathcal M_t/\mathcal M_{k;t}} = \mathcal C_{\mathcal M_{t'}/\mathcal M_{k;t'}}$. Therefore $\mathcal C_{t'}^\circ \cap \mathcal C^\circ_{\mathcal M_{t'}} \neq \emptyset$ and thus must be the same connected component $M_\mathbb R\setminus \operatorname{supp}(\Phi_\kappa)$. Hence their closures are equal, i.e., $\mathcal C_{t'} = \mathcal C_{\mathcal M_{t'}}$. This finishes the induction.
\end{proof}

\begin{remark}\label{rmk: perpendicular category}
By \Cref{cor: g-vector equality}, the cone $\mathcal C_{\mathcal M_t/\mathcal M_{i;t}}$ is equal to the cone $\overline{\mathfrak d}_{i;t}$ for any $t$ and $i$. Thus by \Cref{prop: function value of phi} and \Cref{thm: BST}, for any $\theta\in \mathfrak d_{i;t}$, we have
\[
\Phi_\kappa(\theta) = J_{\mathcal M_t/\mathcal M_{i;t}} = T_t(\ab{S_i}).
\]
\end{remark}

Recall that $\mathbf E(\rstau{\mathcal P(\kappa)})$ is the exchange graph of reachable support $\tau$-tilting pairs. The following result is a direct corollary of \Cref{cor: g-vector equality}.

\begin{corollary}\label{cor: isomorphism stability chambers and stau}
The equality $\mathcal C_{\mathcal M_t} = \mathcal C_t$ for any $t\in \mathbb T_n$ induces a bijection between $\Delta_\kappa^+$ and $\rstau{\mathcal P(\kappa)}$ and an isomorphism between graphs $\mathbf E(\Delta_\kappa^+)$ and $\mathbf E(\rstau {\mathcal P(\kappa)})$.
\end{corollary}

Let $\sigma = \kappa(t)$ be the labeled triangulation associated to $t\in \mathbb T_n$ (with $\kappa$ associated to $t_0\in \mathbb T_n$). One could consider all the previous constructions with respect to $\sigma$ and $t\in \mathbb T_n$ as the root.

\begin{remark}\label{eq: isomorphism stau graphs}
Composing the isomorphism \Cref{cor: isomorphism stability chambers and stau} with \Cref{prop: isomorphism of graphs of chambers}, we obtain an isomorphism of graphs
\begin{equation}
T_{t}^{\kappa; t_0} \colon \mathbf E(\rstau{\mathcal P(\sigma)}) \xrightarrow{\sim} \mathbf E(\rstau {\mathcal P(\kappa)})
\end{equation}
such that $T_t^{\kappa;t_0}(\mathcal M_{s}^{\sigma;t}) = \mathcal M_s^{\kappa;t_0}$ for any $s\in \mathbb T_n$.
\end{remark}

\section{Cluster variables as Caldero--Chapoton functions}\label{section: caldero chapoton}

In this section, we prove that the Caldero--Chapoton functions of reachable $\tau$-rigid pairs are cluster monomials of the generalized cluster algebra $\mathcal A(\kappa)$, leading to the main result \Cref{thm: main theorem}.

\subsection{Caldero--Chapoton function}
Let $A$ be a finite dimensional algebra over $\mathbb C$ of the form $\mathbb C\langle Q\rangle/I$ with $I$ admissible and $Q_0 = \{1, 2, \dots, n\}$. An $A$-module $M$ can be equivalently viewed as a finite dimensional representation of $Q$ satisfying relations given by $I$.

\begin{definition}\label{def: quiver grassmannian}
	For $M\in \modu A$ and $\mathbf d = (d_1, d_2, \dots, d_n) \in \mathbb N^n$, we define the \emph{quiver Grassmannian} $\operatorname{Gr}(M, \mathbf d)$ to be the projective $\mathbb C$-scheme whose $\mathbb C$-points parametrize the quotient representations of $M$ of dimension vector $\mathbf d$.
\end{definition}

\begin{definition}\label{def: f-poly}
The \emph{$F$-polynomial} of $M\in \modu A$ is defined to be
\begin{equation}\label{eq: f-poly}
F_M(y_1, \dots , y_n) = \sum_{\mathbf d\in \mathbb N^n} \chi(\mathrm{Gr}(M,\mathbf d)) \cdot y^\mathbf d \in \mathbb Z[y_1, y_2, \dots, y_n]   
\end{equation}
where $y^\mathbf d = \prod_{i=1}^n y_i^{d_i}$. Here $\chi(\cdot)$ takes the Euler characteristic of the analytic topology of $X(\mathbb C)$ of some finite type $\mathbb C$-scheme $X$.
\end{definition}

Now consider the algebra $\mathcal P(\kappa)$ associated to some labeled triangulation $\kappa$ of $\surf$. We choose the field $\Bbbk$ to be $\mathbb C$.

\begin{definition}\label{def: cc function}
Let $(M,P)$ be a $\tau$-rigid pair in $\modu \mathcal P(\kappa)$. We define the \emph{Caldero--Chapoton function} of the pair $(M,P)$ to be
\[
CC_\kappa(M,P) = x^{\mathbf g(M,P)}\cdot F_M(\hat y_1, \dots, \hat y_n) \in \mathbb Z[x_1^\pm, \dots, x_n^\pm].
\]
where $x^m = \prod_{i=1}^nx_i^{m_i}$ for an integral vector $m = (m_1, \dots, m_n)$.
\end{definition}

Consider the generalized cluster algebra $\mathcal A(\kappa)$ with initial cluster variables $x_1, x_2, \dots, x_n$.

\begin{theorem}\label{thm: cc function}
For any $t\in \mathbb T_n$ and $i\in \{1, 2, \dots, n\}$, let $\mathcal M_{i; t}$ be the corresponding indecomposable $\tau$-rigid pair of $\modu \mathcal P(\kappa)$. Then the Caldero--Chapoton function of $\mathcal M_{i;t}$ is equal to the cluster variable $x_{i;t}$, i.e.,
\[
CC_\kappa(\mathcal M_{i;t}) = x_{i;t}\in \mathcal A(\kappa).
\]
\end{theorem}

\begin{proof}
By \Cref{thm:F-pols-are-pols-and-expansion-of-gen-cluster-vars}, a cluster variable $x_{i;t}$ is determined by its $\mathbf g$-vector and $F$-polynomial as
\[
x_{i;t} = x^{\mathbf g_{i;t}}\cdot F_{i;t}(\hat y_1, \dots, \hat y_n).
\]
It is shown in \Cref{cor: g-vector equality} that $\mathbf g(\mathcal M_{i; t}) = \mathbf g_{i;t}$. By \Cref{def: cc function}, it suffices to show that 
\[
    F_{\mathcal M_{i;t}}(y) = F_{i;t}(y) \in \mathbb Z[y_1, \dots, y_n]
\]
(if $\mathcal M_{i;t}$ is of the form $(0,P)$, define $F_{(0,P)}$ to be $1$). Consider the cone $\mathcal C_t \in \Delta^+_\kappa$ and let $\theta$ be in the interior of $\mathcal C_t$. By \Cref{cor: f-polynomial from stability}, we only need to show the identity
\begin{equation}\label{eq: identity to prove}
    \mathfrak p_t(x^{\mathbf g_{i;t}}) = x^{\mathbf g_{i;t}} \cdot F_{\mathcal M_{i;t}}(y).
\end{equation}
By the consistency of $\phi_\kappa^\mathrm{Hall}$ (\Cref{lemma: consistency and torsion class}), we have $\mathfrak p^\mathrm{Hall}_t = 1_{\mathcal T(\theta)}$ in the motivic Hall algebra. Denote $m = \mathbf g_{i;t}\in M$. We compute the action of $\mathfrak p_{t}^\mathrm{Hall}$ on $x^m$
\begin{equation}\label{eq: action of torsion class}
\mathfrak p_{t}^\mathrm{Hall}(x^m) = 1_{\mathcal T(\theta)}^{-1} * x^m * 1_{\mathcal T(\theta)} = 1_{\mathcal T(\theta)}^{-1} * \left( \sum_{\mathbf d\in N^\oplus} \mathbb L^{-m(\mathbf d)} 1_{\mathcal T(\theta), \mathbf d} \right) * x^m.    
\end{equation}

Suppose that $\mathcal M_{i;t}$ is of the form $(0,P)$. Then there is no morphism from the projective module $P$ to any object in $\mathcal T(\theta) = \mathrm{Fac}(\mathcal M_{t})$ (by \Cref{prop: two definitions of a torsion class}) because of the $\tau$-rigidity of $\mathcal M_t$. Thus we have $m(\mathbf d) = \hom (P, T) = 0$ for $T\in \mathcal T(\theta)$ with dimension vector $\mathbf d$. Then $\mathfrak p_t^\mathrm{Hall}(x^m) = x^m$. Applying the integration map and by \Cref{lemma: Poisson homomorphism}, we get $\mathfrak p_t(x^m) = x^m$ as desired in (\ref{eq: identity to prove}). 

Suppose that $\mathcal M_{i;t}$ is of the form $(M_{i;t}, 0)$. We use the following identity provided by \Cref{lemma: hall algebra identity} in the motivic Hall algebra:
\begin{equation}\label{eq: nagao identity}
    \sum_{\mathbf d\in N^\oplus} \mathbb L^{-m(\mathbf d)} 1_{\mathcal T(\theta), \mathbf d} = 1_{\mathcal T(\theta)} * \sum_{\mathbf d\in N^\oplus}  [\operatorname{Gr}(M_{i; t}, \mathbf d)\rightarrow \mathfrak M_\mathbf d].
\end{equation}
Applying the identity to the right side of (\ref{eq: action of torsion class}) gives
\[
\mathfrak p_t^\mathrm{Hall}(x^m) = \left( \sum_{\mathbf d\in N^\oplus}  [\operatorname{Gr}(M_{i; t}, \mathbf d)\rightarrow \mathfrak M_\mathbf d] \right) * x^m.
\]
Finally using the integration map, we have
\[
\mathfrak p_t(x^m) \overset{{\Cref{lemma: Poisson homomorphism}}}{ = } I\left( \mathfrak p_t^\mathrm{Hall}(x^m) \right) = x^m \cdot \left( \sum_{\mathbf d\in N^\oplus} \chi(\mathrm{Gr}(M_{i;t}, \mathbf d))y^\mathbf d\right) = x^m \cdot F_{M_{i;t}}(y).
\]
\end{proof}

We show the following identity in the motivic Hall algebra to complete the proof of \Cref{thm: cc function}.
\begin{lemma}\label{lemma: hall algebra identity}
$\sum\limits_{\mathbf d\in N^\oplus} \mathbb L^{-m(\mathbf d)} 1_{\mathcal T(\theta), \mathbf d} = 1_{\mathcal T(\theta)} * \sum\limits_{\mathbf d\in N^\oplus}  [\operatorname{Gr}(M_{i; t}, \mathbf d)\rightarrow \mathfrak M_\mathbf d]$ where $m = \mathbf g(M_{i;t})$.
\end{lemma}
\begin{proof}
For $M\in \modu A$, denote the disjoint union of quiver Grassmannians of all dimensions by $\mathrm{Gr}(M) \coloneqq \coprod_{\mathbf d\in N^\oplus} \mathrm{Gr}(M,\mathbf d)$. Denote $\mathcal T \coloneqq \mathcal T(\theta)$. So the right-hand side of the identity can be rewritten as $[\mathfrak M_{\mathcal T} \rightarrow \mathfrak M] * [\mathrm{Gr}(M_{i;t}) \rightarrow \mathfrak M] = [Z\xrightarrow{\pi_2\circ h} \mathfrak M]$ where $Z$ fits into the following Cartesian square
\[
\begin{tikzcd}
Z \ar[d]\ar[r, "h"] & \mathfrak M^{(2)} \ar[r, "\pi_2"] \ar[d, "{(\pi_3, \pi_1)}"] & \mathfrak M\\
\mathfrak M_{\mathcal T} \times \operatorname{Gr}(M_{i;t}) \ar[r] & \mathfrak M \times \mathfrak M
\end{tikzcd}
\]
A $\mathbb C$-point of the stack $Z$ by definition represents a short exact sequence in $\modu A$
\[
0\rightarrow N \rightarrow E \rightarrow T \rightarrow 0
\]
where $N$ is a quotient module of $M_{i;t}$ and $T\in \mathcal T$ (which implies $E\in \mathcal T$). The set of morphisms between two such objects are pairs of isomorphisms $(a, b)$ such that the following diagram commutes
\[
\begin{tikzcd}
0 \ar[r] & N \ar[r]\ar[d, "\mathrm{id}"] & E \ar[d, "a"] \ar[r] & T \ar[d, "b"] \ar[r] &0\\
0 \ar[r] & N_1 \ar[r] & E_1 \ar[r] &T_1 \ar[r] &0
\end{tikzcd}.
\]
Notice that here $N = N_1$ are the same quotient module of $M_{i;t}$ (if $N \neq N_1$, the set is empty). Consider another stack $\mathfrak {Hom}(M_{i;t}, \mathcal T)$ (and with a natural map to $\mathfrak M_{\mathcal T}$) parametrizing morphisms from $M_{i;t}$ to $\mathcal T$ (see for example the construction in \cite[Section 6.1(D)]{nagao2013donaldson}). Then there is a morphism of stacks from $Z$ to $\mathfrak {Hom}(M_{i;t}, \mathcal T)$ over $\mathfrak M_{\mathcal T}$ sending a short exact sequence above to the composition $M_{i;t} \rightarrow N \rightarrow E$ which induces an equivalence from the groupoid $Z(\mathbb C)$ to $\mathfrak {Hom}(M_{i;t}, \mathcal T)(\mathbb C)$. Thus we have $[Z\rightarrow \mathfrak M] = [\mathfrak {Hom}(M_{i;t}, \mathcal T) \rightarrow \mathfrak M]$ in the motivic Hall algebra (see \cite[Section 5.2]{bridgeland2016scattering}). 

By \cite[Proposition 2.4]{adachi2014tau}, for any $T\in \mathcal T(\theta)$ of dimension vector $\mathbf d$, we have the formula
\[
-\mathbf g(M_{i;t})(\mathbf d) = \hom{(M_{i;t}, T)} - \hom{(T, \tau(M_{i;t}))}.
\]
However since $T\in \mathcal T = \mathrm{Fac}(\mathcal M_t)$ (by \Cref{prop: two definitions of a torsion class}) and $M_{i;t}$ is $\tau$-rigid, we have $\hom {(T, \tau(M_{i;t}))} = 0$. Thus the fiber at any $\mathbb C$-point in $\mathfrak M_{\mathcal T,d}$ of the map $\mathfrak {Hom}(M_{i;t}, \mathcal T(\theta)) \rightarrow \mathfrak M_\mathcal T$ is isomorphic to $\mathbb C^{-m(d)}$. Therefore we have 
\[
\sum_{\mathbf d\in N^\oplus} \mathbb L^{-m(\mathbf d)} 1_{\mathcal T, \mathbf d} = [\mathfrak {Hom}(M_{i;t}, \mathcal T) \rightarrow \mathfrak M]
\]
which finishes the proof.
\end{proof}

\begin{remark}\label{rmk: cc function of monomials}
The proof of \Cref{thm: cc function} can be easily extended to any reachable $\tau$-rigid pairs in which case we have for $(b_i)_i\in \mathbb N^n$,
\[
CC_\kappa\left(\bigoplus_{i=1}^n \mathcal M_{i;t}^{b_i}\right) = \prod_{i=1}^n x_{i;t}^{b_i} \in \mathcal A(\kappa).
\]
\end{remark}

\subsection{Isomorphism of exchange graphs}

Let us attach $\mathcal M_t\in \rstau{\mathcal P(\kappa)}$ to any $t\in \mathbb T_n$. If $\mathcal M_t \cong \mathcal M_s$, then there exists some permutation of indices $\{1,2,\dots, n\}$ such that their components are isomorphic via this permutation. Then there exists a unique graph automorphism of $\mathbb T_n$ sending $t$ to $s$ that respects the mutations of support $\tau$-tilting pairs. Let $G(\mathcal P(\kappa))\subset \Aut (\mathbb T_n)$ be the subgroup formed by all such automorphisms. We obtain the quotient graph $\mathbb T_n/G(\mathcal P(\kappa))$ isomorphic to $\mathbf E(\rstau{\mathcal P(\kappa)})$. Recall the exchange graph $\mathbf E(B(\kappa)) = \mathbb T_n/G(B(\kappa))$ of the generalized cluster algebra $\mathcal A(\kappa)$ as in \Cref{def:exhchange-graph-unlabeled-seeds}. 

\begin{corollary}\label{prop: cc induces surjective map}
For each $t\in \mathbb T_n$, sending the support $\tau$-tilting pair $\mathcal M_t$ to the labeled seed $$((CC_\kappa(\mathcal M_{1;t}), CC_\kappa(\mathcal M_{2;t}), \dots, CC_\kappa(\mathcal M_{n;t})), B(t))$$ induces a covering of graphs from $\mathbf E(\rstau{\mathcal P(\kappa)})$ to $\mathbf E(B(\kappa))$.
\end{corollary}

\begin{proof}
\Cref{thm: cc function} shows that $((CC_\kappa(\mathcal M_{1;t}), CC_\kappa(\mathcal M_{2;t}), \dots, CC_\kappa(\mathcal M_{n;t})), B(t))$ is actually the labeled seed $(\mathbf x_t, B(t))$ of $\mathcal A(\kappa)$. Furthermore if $\mathcal M_t \cong \mathcal M_s$ (so their indecomposable summands are isomorphic respectively via a unique permutation $\sigma$), we know that $\mathbf x_t = \sigma(\mathbf x_s)$ and also by \Cref{cor: c vectors determine b matrix} that $B(s) = \sigma(B(t))$. Thus the group $G(\mathcal P(\kappa))$ is a subgroup of $G(B(\kappa))$, hence the result. We summarize the result in the following commutating diagram.
\[
\begin{tikzcd}[row sep=large,column sep=huge]
(\mathbb T_n, M_t) \ar[r, "{(CC_\kappa(\cdot), B(\cdot))}"] \ar[d, "{\cdot/G(\mathcal P(\kappa))}"] & (\mathbb T_n, (\mathbf x_t, B(t))) \ar[d, "{\cdot/G(B(\kappa))}"]\\
\mathbf E(\rstau{\mathcal P(\kappa)}) \ar[r, two heads] & \mathbf E(B(\kappa))
\end{tikzcd}
\]
\end{proof}

Let $\mathcal M = \mathcal M_1\oplus \mathcal M_2 \oplus \cdots \oplus \mathcal M_n$ be any reachable support $\tau$-tilting pair. Not like $\mathcal M_t$ which comes with the information that its indecomposable summands are indexed by $\{1,2,\dots, n\}$, there is no preferred labeling on the summands $\mathcal M_i$.

By \Cref{rmk: perpendicular category}, the perpendicular category $J_{\mathcal M/ \mathcal M_i}$ must be equivalent to either $\modu \mathbb C$ or $\modu \mathbb C[\varepsilon]/\varepsilon^2$ where the two cases correspond to $r_i = 1$ and $r_i = 2$ respectively. We define the matrix $B(\mathcal M) = (\varepsilon_{ij})$ by 
\[
\varepsilon_{ij} \coloneqq \mathbf c_{i} \cdot \overline{B}(\kappa) \cdot (r_j \mathbf c_j^T)
\]
where $(\mathbf c_i)_i$ is just the dual basis of $(\mathbf g(\mathcal M_i))_i$. If $\mathcal M$ is isomorphic to some $\mathcal M_t$, then again by \Cref{cor: c vectors determine b matrix}, $B(\mathcal M) = \sigma(B(t))$ for the unique permutation $\sigma$ determined by the isomorphism.

\begin{theorem}\label{thm: cc function induces bijection and isomorphism}
Sending the $\tau$-rigid pair $\mathcal M_{i;t}^{\kappa;t_0}$ to its Caldero--Chapoton function $CC_\kappa(\mathcal M_{i;t}^{\kappa;t_0})$ induces a bijection
	\[
	CC_\kappa \colon \{\text{reachable indecomposable $\tau$-rigid pairs over $\mathcal{P}(\kappa)$}\} \longleftrightarrow \{ \text{cluster variables of $\mathcal A(\kappa)$}\}
	\]
Hence the covering of graphs 
\[
    \mathbf E(\rstau{\mathcal P(\kappa)})\twoheadrightarrow\mathbf E(B(\kappa)), \quad \mathcal M\mapsto \left((CC_\kappa(\mathcal M_i))_{i=1}^n, B(\mathcal M)\right)
\]
is actually an isomorphism. Accordingly, for every $k\in\{1,\ldots,n\}$ we have{\small
\begin{equation}\label{eq:CC-functions-satisfy-Chekhov-Shapiro-equation}
CC_\kappa(\mathcal{M}_k)CC_\kappa(\mathcal{M}_k')=\begin{cases}
	    \underset{j:\varepsilon_{jk}<0}{\prod} CC_\kappa(\mathcal{M}_j)^{-\varepsilon_{jk}}+
	    \underset{j:\varepsilon_{jk}>0}{\prod} CC_\kappa(\mathcal{M}_j)^{\varepsilon_{jk}} & \text{if $r_k=1$;}\\
	    \underset{j:\varepsilon_{jk}<0}{\prod} CC_\kappa(\mathcal{M}_j)^{2}
	    +\underset{j:\varepsilon_{jk}>0}{\prod} CC_\kappa(\mathcal{M}_j) \underset{j:\varepsilon_{jk}<0}{\prod} CC_\kappa(\mathcal{M}_j)
	    +\underset{j:\varepsilon_{jk}>0}{\prod} CC_\kappa(\mathcal{M}_j)^{2}
	    & \text{if $r_k=2$},
	    \end{cases}
	\end{equation}}
where, up to isomorphism, $\mathcal{M}_k'$ is the unique indecomposable $\tau$-rigid pair such that
\[
    \mathcal{M}_1\oplus\cdots\mathcal{M}_{k-1}\oplus\mathcal{M}_k'\oplus\mathcal{M}_{k+1}\oplus\cdots\oplus\mathcal{M}_n=\mu_k^{\operatorname{AIR}}(\mathcal{M}_1\oplus\cdots\mathcal{M}_{k-1}\oplus\mathcal{M}_k\oplus\mathcal{M}_{k+1}\oplus\cdots\oplus\mathcal{M}_n)
\]
\end{theorem}

\begin{proof}
By \Cref{prop: cc induces surjective map}, the map $CC_\kappa$ is surjective since a cluster variable must be some Caldero--Chapoton function. For the injectivity, we need to show that if $\mathcal M_{i;t}$ is not isomorphic to $\mathcal M_{j;s}$, then $CC_\kappa(\mathcal M_{i;t}) \neq CC_\kappa(\mathcal M_{j;s})$. We claim that it is enough to show
\begin{lemma}\label{lemma: cc function of a module is not initial}
If $\mathcal M_{i;t}$ is not of the form $(0,P)$, i.e. is of the form $(M_{i;t},0)$, then $CC_\kappa(\mathcal M_{i;t})$ is different from each of the initial variables $x_1, x_2, \dots, x_n$.
\end{lemma}
Suppose the above lemma is true. Consider the automorphism $\varphi$ of $\mathcal F = \mathbb Q(x_1,\dots, x_n)$ sending $x_{i;t}$ to $x_i$ for $i=1,2,\dots n$. Now it suffices to show  
\[
\varphi(x_{j;s}) = x_{j;s}^{B(t);t} \neq x_i = \varphi(x_{i;t}).
\]
Let $\sigma = \kappa(t)$ be the labeled triangulation associated to $t$ (such that $\kappa$ is associated to $t_0$). Then we have $x_{j;s}^{B(t);t} = CC_{\sigma}(\mathcal M_{j;s}^{\sigma; t})$ by applying \Cref{thm: cc function} to $\sigma$ where the Caldero--Chapoton function is with respect to the algebra $\mathcal P(\sigma)$, thus the subscript. Since $\mathcal M_{i;t}\not\cong \mathcal M_{j;s}$ implies that
$(0, \mathcal P(\sigma){e_i}) = \mathcal M_{i;t}^{\sigma; t} \not \cong \mathcal M_{j;s}^{\sigma; t}$ by \Cref{eq: isomorphism stau graphs}, the \Cref{lemma: cc function of a module is not initial} applied to $\mathcal P(\sigma)$ just asserts that 
\[
x_{j;s}^{B(t);t} = CC_{\sigma}(\mathcal M_{j;s}^{\sigma; t}) \neq x_i.
\]

Now we complete the proof of \Cref{thm: cc function induces bijection and isomorphism} by proving \Cref{lemma: cc function of a module is not initial}.
\begin{proof}[{Proof of \Cref{lemma: cc function of a module is not initial}}]
It is enough to show that in the Laurent expansion of $CC(\mathcal M_{i;t})$, each monomial $$x^{\mathbf g_{i;t}}\cdot x^{\overline{B}\cdot \mathbf d} = \prod_{i=1}^n x_i^{a_i}$$ must have some $a_i<0$. It suffices to show the inner product (as vectors in $\mathbb Z^n$)
\[
\langle \mathbf g_{i;t} + \overline{B}\cdot \mathbf d, \mathbf d\rangle = \langle \mathbf g_{i;t}, \mathbf d\rangle
\]
is strictly negative for any $\mathbf d\in \mathbb N^n$ such that there is a quotient module $N$ of $M_{i;t}$ whose dimension vector is $\mathbf d$. Using the exact sequence in \cite[Proposition 2.4 (a)]{adachi2014tau}, we have
\[
\langle \mathbf g_{i;t}, \mathbf d\rangle = -\hom(M_{i;t},N) + \hom(N, \tau(M_{i;t})).
\]
The second term on the right side vanishes because $M_{i;t}$ is $\tau$-rigid and $N$ is a quotient of $M_{i;t}$, which makes the inner product strictly negative.
\end{proof}

Returning to the proof of \Cref{thm: cc function induces bijection and isomorphism}, now that non-isomorphic $\mathcal M_{i;t}$ and $\mathcal M_{j;s}$ are shown to have distinguished Caldero--Chapoton functions, non-isomorphic objects in $\rstau{\mathcal P(\kappa)}$ must not correspond to the same unlabeled seed. Then the covering of graphs is injective, thus an isomorphism.

Finally, the Caldero--Chapoton functions $CC_\kappa(\mathcal M_k)$ and $CC_\kappa(\mathcal M_k')$ satisfy the exchange relations as stated simply because in the corresponding unlabeled seeds, the cluster variables satisfy such exchange relations.
\end{proof}

\section{\texorpdfstring{An example in affine type $\widetilde{C}_2$}{An example in affine type C2}}\label{section:example} 
Consider the triangulations $\kappa_0 = \kappa,\kappa_1,\kappa_2,\kappa_3$ and $\kappa_4=\sigma$ from Figure \ref{Fig_ExTriangsC2tilde_GoodExample}. Arcs in these triangulations are labeled by $\{1, 2, 3\}$. The sequence of flips from $\kappa$ to $\sigma$ is $(k_1, k_2, k_3, k_4) = (1, 3, 2, 3)$. We say this example in type $\widetilde C_2$ since $B(\kappa)$ is of that type. The choice of signs $(\epsilon_1, \epsilon_2,\epsilon_3, \epsilon_4)\in \{+,-\}^4$ from Proposition \ref{prop: chamber} is $(+,+,+,+)$.
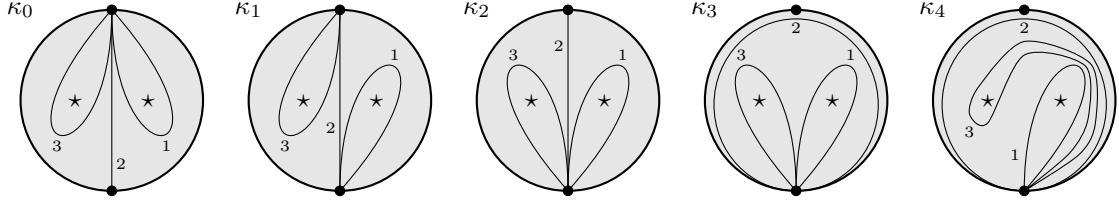
\begin{figure}[h]
                \centering
\begin{tikzpicture}[scale = 1.2]
\filldraw[fill=gray!20, thick](0,0) circle (1);
\draw[] (0,-1) -- (0,1);
\draw[] (0,1) .. controls (-1.5, -0.7) and (0, -1) .. (0,1);
\draw[] (0,1) .. controls (1.5, -0.7) and (0, -1) .. (0, 1);
\node[] at (-0.4, 0) {$\star$};
\node[] at (0.4, 0) {$\star$};
\node[] at (0.6, -0.5) {\tiny $1$};
\node[] at (0.1, -0.7) {\tiny $2$};
\node[] at (-0.6, -0.5) {\tiny $3$};
\filldraw(0, -1) circle (1.5pt);
\filldraw(0, 1) circle (1.5pt);
\node[] at (-1, 1) {$\kappa_0$};

\filldraw[fill=gray!20, thick](2.5,0) circle (1);
\filldraw(2.5, -1) circle (1.5pt);
\filldraw(2.5, 1) circle (1.5pt);
\node[] at (2.1, 0) {$\star$};
\node[] at (2.9, 0) {$\star$};
\draw[] (2.5,-1) -- (2.5,1);
\draw[] (2.5,1) .. controls (1, -0.7) and (2.5, -1) .. (2.5,1);
\draw[] (2.5,-1) .. controls (4, 0.7) and (2.5, 1) .. (2.5, -1);
\node[] at (1.9, -0.5) {\tiny $3$};
\node[] at (2.4, -0.3) {\tiny $2$};
\node[] at (3.1, 0.5) {\tiny $1$};
\node[] at (1.5, 1) {$\kappa_1$};

\filldraw[fill=gray!20, thick](5,0) circle (1);
\filldraw(5, -1) circle (1.5pt);
\filldraw(5, 1) circle (1.5pt);
\node[] at (4.6, 0) {$\star$};
\node[] at (5.4, 0) {$\star$};
\draw[] (5,-1) -- (5,1);
\draw[] (5,-1) .. controls (3.5, 0.7) and (5, 1) .. (5, -1);
\draw[] (5,-1) .. controls (6.5, 0.7) and (5, 1) .. (5, -1);
\node[] at (5.6, 0.5) {\tiny $1$};
\node[] at (4.9, 0.6) {\tiny $2$};
\node[] at (4.4, 0.5) {\tiny $3$};
\node[] at (4, 1) {$\kappa_2$};

\filldraw[fill=gray!20, thick](7.5,0) circle (1);
\filldraw(7.5, -1) circle (1.5pt);
\filldraw(7.5, 1) circle (1.5pt);
\node[] at (7.1, 0) {$\star$};
\node[] at (7.9, 0) {$\star$};
\draw[] (7.5,-1) .. controls (6, 0.7) and (7.5, 1) .. (7.5, -1);
\draw[] (7.5,-1) .. controls (9, 0.7) and (7.5, 1) .. (7.5, -1);
\draw[] (7.5, -0.05) ellipse (0.9 and 0.95);
\node[] at (8.1, 0.5) {\tiny $1$};
\node[] at (7.5, 0.8) {\tiny $2$};
\node[] at (6.9, 0.5) {\tiny $3$};
\node[] at (6.5, 1) {$\kappa_3$};

\filldraw[fill=gray!20, thick](10,0) circle (1);
\filldraw(10, -1) circle (1.5pt);
\filldraw(10, 1) circle (1.5pt);
\node[] at (9.6, 0) {$\star$};
\node[] at (10.4, 0) {$\star$};
\draw[] (10,-1) .. controls (11.5, 0.7) and (10, 1) .. (10, -1);
\draw[] (10, -0.05) ellipse (0.9 and 0.95);
\draw[] plot [smooth] coordinates {(10, -1) (10.6, -0.6) (10.78, -0.25) (10.73, 0.4) (10, 0.65) 
(9.5, 0.15) (9.4, -0.15) (9.6, -0.25) (10, 0.5) (10.67, 0.38) (10.7, -0.2) (10.5, -0.58) (10, -1)};
\node[] at (9.9, -0.6) {\tiny $1$};
\node[] at (10, 0.8) {\tiny $2$};
\node[] at (9.4, -0.35) {\tiny $3$};
\node[] at (9, 1) {$\kappa_4$};

\end{tikzpicture}
\caption{Five triangulations of the digon with two orbifold points of order 3. Adjacent triangulations are related by a flip.}
\label{Fig_ExTriangsC2tilde_GoodExample}
\end{figure}

\begin{table}[ht]
\centering
    \begin{tabular}{|c|c|c|c|}
    \hline
         & $Q(\kappa_j)$ & $B(\kappa_j)$ & $\mathbf{r}_1^j,\mathbf{r}_2^j,\mathbf{r}_3^j$ \\
    \hline
            &&& \\
    $\kappa_4$ & { $\xymatrix{&  3 \ar@(ld,lu) \ar[rr]  & & 1 \ar[dl] \ar@(ru,rd)  & \\ & & 2 \ar[ul] & & }$} & { $\left[\begin{array}{ccc} 0 & 1  & -2 \\ -2 & 0 & 2 \\ 2 & -1 & 0 \end{array}\right]$ } & { $\left[\begin{array}{ccc} 1 &  0 & 0 \\ 0 & 1 & 0\\ 0 & 0 & 1 \end{array}\right]$ }\\
        &&& \\
    $\kappa_3$ & { $\xymatrix{& 3 \ar@(ld,lu) \ar[dr] & & 1 \ar@(ru,rd) \ar[ll]  & \\ & &  2 \ar[ur] & & }$} & { $\left[\begin{array}{ccc} 0 & -1 & 2 \\ 2 & 0 & -2 \\ -2 & 1 & 0 \end{array}\right]$ } & { $\left[\begin{array}{ccc} 1 & 0  & 0 \\ 0 & 1 & 2 \\ 0 & 0 & -1 \end{array}\right]$ } \\
        &&& \\
    $\kappa_2$    & { $\xymatrix{ & 3 \ar@(ld,lu) & & 1 \ar@(ru,rd) \ar[dl] & \\ & & 2 \ar[ul] & & }$}  & {$\left[\begin{array}{ccc} 0 & 1 & 0\\ -2 & 0 & 2 \\ 0 & -1 & 0 \end{array}\right]$ } & { $\left[\begin{array}{ccc} 1 & 0  & 0 \\ 0 & -1 & -2 \\ 0 & 1 & 1\end{array}\right]$ } \\
        &&& \\
    $\kappa_1$     & {  $\xymatrix{ & 3  \ar@(ld,lu) \ar[dr] & & 1  \ar@(ru,rd) \ar[dl] & \\ & & 2   & & }$}   & { $\left[\begin{array}{ccc} 0 & 1 & 0 \\ -2 & 0 & -2 \\ 0 & 1 & 0 \end{array}\right]$}  & { $\left[\begin{array}{ccc} 1 & 0  & 0 \\ 0 & 1 & 0 \\ 0 & -1 & -1\end{array}\right]$ } \\
        &&& \\
    $\kappa_0$ & { $\xymatrix{ & 3  \ar@(ld,lu) \ar[dr] & & 1 \ar@(ru,rd) &\\ & & 2 \ar[ur] & & }$} & { $\left[\begin{array}{ccc} 0 & -1 & 0\\ 2 & 0 & -2 \\ 0 & 1 & 0\end{array}\right]$} & { $\left[\begin{array}{ccc}  -1 & 0  & 0 \\ 0 & 1 & 0 \\ 0 & -1 & -1 \end{array}\right]$ } \\
    \hline     
    \end{tabular}
    \caption{}
    \label{table:r-vectors}
\end{table}

In \Cref{table:r-vectors} we can see the quivers and skew-symmetrizable matrices of the triangulations, as well as the images $\mathbf{r}_1^j$, $\mathbf{r}_2^j$ and $\mathbf{r}_3^j$ of the standard basis vectors under the transformations
\begin{align*}T^+_{k_1;\kappa_1}\circ T^+_{k_2;\kappa_2}\circ T^+_{k_3;\kappa_3}\circ T^+_{k_4;\kappa_4}, & \\  T^+_{k_2;\kappa_2}\circ T^+_{k_3;\kappa_3}\circ T^+_{k_4;\kappa_4}, & \\  T^+_{k_3;\kappa_3}\circ T^+_{k_4;\kappa_4} &\qquad \text{and}\\ T^+_{k_4;\kappa_4}.
\end{align*}
Thus, with the notation of \Cref{prop: chamber}, \Cref{prop: isomorphism of graphs of chambers} tells us that the cone $\mathcal{C}_j=\{\lambda_1\mathbf{r}_1^j+\lambda_2\mathbf{r}_2^j+\lambda_3\mathbf{r}_3^j\suchthat \lambda_1,\lambda_2,\lambda_3\geq 0\}$ should be regarded as sitting inside the $\tau$-tilting fan of the gentle algebra $\mathcal{P}(\kappa_j)$.

According to \Cref{thm: g-vectors are rays of stability chambers,cor: g-vector equality}, there exist indecomposable $\tau$-rigid pairs $\mathcal{M}_{1;j},\mathcal{M}_{2;j},\mathcal{M}_{3;j}$, $0\leq j\leq 4$ (beware of our slight misuse of the subindex $j$), such that $$\mathcal{M}_{1;j}\oplus\mathcal{M}_{2;j}\oplus\mathcal{M}_{3;j}\in\rstau{\mathcal{P}(\kappa_j)} \qquad \text{and}$$
$$
\mathbf{g}^{\mathcal{P}(\kappa_j)}(\mathcal{M}_{1;j})=\mathbf{r}_1^j, \ \mathbf{g}^{\mathcal{P}(\kappa_j)}(\mathcal{M}_{2;j})=\mathbf{r}_2^j, \quad \text{and} \quad \mathbf{g}^{\mathcal{P}(\kappa_j)}(\mathcal{M}_{3;j})=\mathbf{r}_3^j.
$$
Furthermore, again by \Cref{thm: g-vectors are rays of stability chambers,cor: g-vector equality}, for each $j\in \{0,1,2,3,4\}$, the basic support $\tau$-tilting pair $\mathcal{M}_{1;j}\oplus\mathcal{M}_{2;j}\oplus\mathcal{M}_{3;j}\in\rstau{\mathcal{P}(\kappa_j)}$ is actually the one obtained by applying to $(0,\mathcal{P}(\kappa_j))\in \rstau{\mathcal{P}(\kappa_j)}$ the sequence of Adachi--Iyama--Reiten $\tau$-tilting mutations given as follows:
\begin{align*}
    \mathcal{M}_{1;4}\oplus\mathcal{M}_{2;4}\oplus\mathcal{M}_{3;4} &=(0,\mathcal{P}(\kappa_4)) &  \text{in} \quad \rstau{\mathcal{P}(\kappa_4)} \\
    \mathcal{M}_{1;3}\oplus\mathcal{M}_{2;3}\oplus\mathcal{M}_{3;3} &=\mu_{3}^{\operatorname{AIR}}(0,\mathcal{P}(\kappa_3)) &  \text{in} \quad \rstau{\mathcal{P}(\kappa_3)} \\
    \mathcal{M}_{1;2}\oplus\mathcal{M}_{2;2}\oplus\mathcal{M}_{3;2} &=\mu_{3}^{\operatorname{AIR}}\mu_{2}^{\operatorname{AIR}}(0,\mathcal{P}(\kappa_2)) &  \text{in} \quad \rstau{\mathcal{P}(\kappa_2)}\\
    \mathcal{M}_{1;1}\oplus\mathcal{M}_{2;1}\oplus\mathcal{M}_{3;1} &=\mu_{3}^{\operatorname{AIR}}\mu_{2}^{\operatorname{AIR}}\mu_{3}^{\operatorname{AIR}}(0,\mathcal{P}(\kappa_1)) &  \text{in} \quad \rstau{\mathcal{P}(\kappa_1)}\\
    \mathcal{M}_{1;0}\oplus\mathcal{M}_{2;0}\oplus\mathcal{M}_{3;0} &=\mu_{3}^{\operatorname{AIR}}\mu_{2}^{\operatorname{AIR}}\mu_{3}^{\operatorname{AIR}}\mu_{1}^{\operatorname{AIR}}(0,\mathcal{P}(\kappa_0)) &  \text{in} \quad \rstau{\mathcal{P}(\kappa_0)}
\end{align*}
These objects are explicitly computed in \Cref{table: tau rigid pairs} below, where a pair $(0, P)$ is denoted as $P[1]$.

\begin{table}[ht]
    \centering
    \begin{tabular}{|c|c|c|c|}
    \hline
         &  $\mathcal{M}_{1;j}$ & $\mathcal{M}_{2;j}$ & $\mathcal{M}_{3;j}$ \\
    \hline 
         $j = 4$ & $P_1(\kappa_4)[1]$ & $P_2(\kappa_4)[1]$ & $P_3(\kappa_4)[1]$ \\
    
         $j = 3$ & $P_1(\kappa_3)[1]$ & $P_2(\kappa_3)[1]$ & \begin{tikzcd}[row sep = scriptsize, column sep = scriptsize, ampersand replacement=\&]
         \mathbb C^2 \ar[loop above, "{\begin{bsmallmatrix} 0 & 0\\ 1 & 0 \end{bsmallmatrix}}"]\ar[dr] \& \& 0\ar[ll]\ar[loop above]\\
         \& 0 \ar[ur] \&
         \end{tikzcd} \\
    
         $j = 2$ & $P_1(\kappa_2)[1]$ & \begin{tikzcd}[row sep = scriptsize, column sep = scriptsize, ampersand replacement=\&]
         0 \ar[loop above] \& \& 0 \ar[dl] \ar[loop above]\\
         \& \mathbb C \ar[ul] \&
         \end{tikzcd} & \begin{tikzcd}[row sep = scriptsize, column sep = scriptsize, ampersand replacement=\&]
         \mathbb C^2 \ar[loop above, "{\begin{bsmallmatrix} 0 & 0 \\ 1 & 0 \end{bsmallmatrix}}"] \& \& 0 \ar[dl] \ar[loop above]\\
         \& \mathbb C^2 \ar[ul, "{\begin{bsmallmatrix} 1 & 0 \\ 0 & 1 \end{bsmallmatrix}}"] \&
         \end{tikzcd} \\
    
         $j = 1$ & $P_1(\kappa_1)[1]$ & \begin{tikzcd}[row sep = scriptsize, column sep = scriptsize, ampersand replacement=\&]
         \mathbb C^2 \ar[loop above, "{\begin{bsmallmatrix} 0 & 0\\ 1 & 0 \end{bsmallmatrix}}"]\ar[dr, "{\begin{bsmallmatrix} 0 & 1 \end{bsmallmatrix}}", swap] \& \& 0 \ar[dl] \ar[loop above]\\
         \& \mathbb C \&
         \end{tikzcd} & \begin{tikzcd}[row sep = scriptsize, column sep = scriptsize, ampersand replacement=\&]
         \mathbb C^2 \ar[loop above, "{\begin{bsmallmatrix} 0 & 0 \\ 1 & 0 \end{bsmallmatrix}}"] \ar[dr, "{\begin{bsmallmatrix} 1 & 0 \\ 0 & 1 \end{bsmallmatrix}}", swap]  \& \& 0 \ar[dl] \ar[loop above]\\
         \& \mathbb C^2 \&
         \end{tikzcd}\\
    
         $j = 0$ & \begin{tikzcd}[row sep = scriptsize, column sep = scriptsize, ampersand replacement=\&]
         0 \ar[loop above]\ar[dr] \& \& \mathbb C^2 \ar[loop above, "{\begin{bsmallmatrix} 0 & 0\\ 1 & 0 \end{bsmallmatrix}}"]\\
         \& 0 \ar[ur] \&
         \end{tikzcd} & \begin{tikzcd}[row sep = scriptsize, column sep = scriptsize, ampersand replacement=\&]
         \mathbb C^2 \ar[loop above, "{\begin{bsmallmatrix} 0 & 0\\ 1 & 0 \end{bsmallmatrix}}"]\ar[dr, "{\begin{bsmallmatrix} 0 & 1 \end{bsmallmatrix}}", swap] \& \& \mathbb C^2 \ar[loop above, "{\begin{bsmallmatrix} 0 & 0\\ 1 & 0 \end{bsmallmatrix}}"]\\
         \& \mathbb C \ar[ur, "\begin{bsmallmatrix} 1 \\ 0 \end{bsmallmatrix}", swap] \&
         \end{tikzcd} & \begin{tikzcd}[row sep = scriptsize, column sep = scriptsize, ampersand replacement=\&]
         \mathbb C^2 \ar[loop above, "{\begin{bsmallmatrix} 0 & 0\\ 1 & 0 \end{bsmallmatrix}}"]\ar[dr, "{\begin{bsmallmatrix} 1 & 0 \\ 0 & 1 \end{bsmallmatrix}}", swap] \& \& \mathbb C^4 \ar[loop above, "{\begin{bsmallmatrix} 0 & 1 & 0 & 0 \\ 0 & 0 & 0 & 0 \\ 0 & 0 & 0 & 0 \\ 0 & 0 & 1 & 0 \end{bsmallmatrix}}"]\\
         \& \mathbb C^2 \ar[ur, "{\begin{bsmallmatrix} 0 & 0 \\ 1 & 0 \\ 0 & 1 \\ 0 & 0 \end{bsmallmatrix}}", swap] \&
         \end{tikzcd}\\
    \hline
                
    \end{tabular}
    \caption{Support $\tau$-tilting pairs}
    \label{table: tau rigid pairs}
\end{table}

Notice that in $\rstau{\mathcal{P}(\kappa_0)}$, we have
\[
\mu_3^{\operatorname{AIR}}(\mathcal{M}_{1;0}\oplus\mathcal{M}_{2;0}\oplus\mathcal{M}_{3;0}) =  \mu_2^\mathrm{AIR}\mu_3^\mathrm{AIR}\mu_1^\mathrm{AIR}(0, \mathcal P(\kappa_0)) =\mathcal{M}_{1;0}\oplus\mathcal{M}_{2;0}\oplus\mathcal{N},
\]
where $\mathcal N = \begin{tikzcd}[column sep = small, ampersand replacement=\&]
\mathbb C^2 \ar[loop left, "{\begin{bsmallmatrix} 0 & 0 \\ 1 & 0\end{bsmallmatrix}}"] \ar[r] \& 0 \ar[r] \& 0 \ar[loop right]
\end{tikzcd}$. In view of \Cref{thm: FST and CS} and \Cref{thm: cc function induces bijection and isomorphism}, the two support $\tau$-tilting pairs $\mathcal M_{1;0}\oplus\mathcal M_{2;0} \oplus \mathcal M_{3;0} $ and $\mathcal M_{1;0}\oplus \mathcal M_{2;0} \oplus \mathcal N$ correspond to the triangulations $\kappa_4$ and $\kappa_3$ respectively.

Since
$$
0\rightarrow P_3(\kappa_0) \rightarrow \mathcal{M}_{3;0}\rightarrow 0 \quad \text{and} \quad
0\rightarrow P_2(\kappa_0)^2 \rightarrow P_3(\kappa_0) \rightarrow \mathcal{N} \rightarrow 0
$$
are the minimal projective resolutions of $\mathcal{M}_{3;0}$ and $\mathcal{N}$,
we have
$$
\mathbf{g}^{\mathcal{P}(\kappa_0)}(\mathcal{M}_{3;0})
=\left[\begin{array}{c}0\\ 0\\ -1\end{array}\right] \quad \text{and} \quad
\mathbf{g}^{\mathcal{P}(\kappa_0)}(\mathcal{N})=
\left[\begin{array}{c}0\\ 2\\ -1\end{array}\right].
$$
Furthermore, the $F$-polynomials of $\mathcal{M}_{3;0}$ and $\mathcal{N}$ are
\begin{align*}
    F_{\mathcal{M}_{3;0}}&=
    y_1^4y_2^2y_3^2+
\chi(\mathbb{P}^1(\mathbb{C}))y_1^3y_2^2y_3^2+
\chi(\{\star\}\sqcup(\mathbb{P}^1(\mathbb{C})\times\mathbb{C}^2))y_1^2y_2^2y_3^2+
\chi(\mathbb{P}^1(\mathbb{C}))y_1y_2^2y_3^2 
\\
& \quad 
+
y_2^2y_3^2  
+
\chi(\mathbb{P}^1(\mathbb{C}))y_1^2y_2y_3^2
+
\chi(\mathbb{P}^1(\mathbb{C}))y_1y_2y_3^2 +
\chi(\mathbb{P}^1(\mathbb{C}))y_2y_3^2 +
y_3^2+
y_1^2y_2y_3 
\\
& \quad +
y_1y_2y_3 
+
y_2y_3 
+
y_3 +
1\\
&=
    y_1^4y_2^2y_3^2+
2y_1^3y_2^2y_3^2+
3y_1^2y_2^2y_3^2+
2y_1y_2^2y_3^2 
\\
& \quad 
+
y_2^2y_3^2  
+
2y_1^2y_2y_3^2
+
2y_1y_2y_3^2 +
2y_2y_3^2 +
y_3^2+
y_1^2y_2y_3 
\\
& \quad +
y_1y_2y_3 
+
y_2y_3 
+
y_3 +
1,\\
    F_{\mathcal{N}}&= y_3^2+y_3+1.
\end{align*}
Hence, their Caldero--Chapoton functions are
\begin{align*}
    CC_{k_0}(\mathcal{M}_{3;0}) &=
x_1^{-2}x_2^2x_3+ 2x_1^{-2}x_2x_3 +3 x_1^{-2}x_3 + 2 x_1^{-2}x_2^{-1}x_3\\
& \quad 
+x_1^{-2}x_2^{-2}x_3 + 3x_1^{-1} + 2 x_1^{-1}x_2^{-1} + 2x_1^{-1}x_2^{-2}+x_2^{-2}x_3^{-1}+x_1^{-1}x_2\\
& \quad
+x_1^{-1}x_2^{-1}+ x_2^{-1}x_3^{-1}+x_3^{-1},
\\
CC_{\kappa_0}(\mathcal{N})&=x_3^{-1}+x_2x_3^{-1}+x_2^2x_3^{-1}.
\end{align*}

As for the $\tau$-rigid modules shared by both $\kappa_3$ and $\kappa_4$, we have
\begin{align*}
\mathbf{g}^{\mathcal{P}(\kappa_0)}(\mathcal{M}_{1;0})
&=\left[\begin{array}{c}-1\\ 0\\ 0\end{array}\right], & 
\mathbf{g}^{\mathcal{P}(\kappa_0)}(\mathcal{M}_{2;0})&=
\left[\begin{array}{c}0\\ 1\\ -1\end{array}\right],\\
    F_{\mathcal{M}_{1;0}}&=
    y_1^2+y_1+1, & 
    F_{\mathcal{M}_{2;0}}&=
    y_1^2y_2y_3^2 + y_1y_2y_3^2 + y_2y_3^2 +y_3^2+ y_3 +1.
\end{align*}
Hence, the Caldero--Chapoton functions are
\begin{align*}
    CC_{k_0}(\mathcal{M}_{1;0}) &= x_1^{-1}x_2^2+x_1^{-1}x_2+x_1^{-1},
\\
CC_{\kappa_0}(\mathcal{M}_{2;0})&=
x_1^{-1}x_2 + x_1^{-1} + x_1^{-1}x_2^{-1} + x_2^{-1}x_3^{-1} + x_3^{-1} + x_2x_3^{-1}.
\end{align*}

Direct computation shows that
$$
CC_{\kappa_0}(\mathcal{M}_{3;0})CC_{\kappa_0}(\mathcal{N})=
CC_{\kappa_0}(\mathcal{M}_{1;0})^2+CC_{\kappa_0}(\mathcal{M}_{1;0})CC_{\kappa_0}(\mathcal{M}_{2;0})+CC_{\kappa_0}(\mathcal{M}_{2;0})^2,
$$
which is precisely the equality predicted by the second row of \eqref{eq:CC-functions-satisfy-Chekhov-Shapiro-equation}.

With the aid of \Cref{table:r-vectors} if necessary, the reader can directly verify that the $\mathbf{g}$-vectors of $\mathcal{M}_{1;0}$, $\mathcal{M}_{2;0}$ and $\mathcal{M}_{3;0}$ coincide with the vectors $\mathbf{r}_1^{0}$, $\mathbf{r}_2^{0}$ and $\mathbf{r}_3^{0}$, as they should by Theorems \ref{thm: g-vectors are rays of stability chambers} and \ref{cor: g-vector equality}.

\section{Final considerations and future work}\label{section: final}

The \emph{$\tau$-tilting complex} of a finite-dimensional algebra $A$ over an algebraically closed field arose from the seminal work of Adachi--Iyama--Reiten \cite{adachi2014tau} as the abstract simplicial complex of $\tau$-rigid pairs. Later on, Demonet--Iyama--Jasso \cite{demonet2019tau} gave a concrete geometric realization of the $\tau$-tilting complex as a simplicial fan, and made the observation that the rays of this \emph{$\tau$-tilting fan} are precisely the rays spanned by the $\mathbf{g}$-vectors of the indecomposable $\tau$-rigid pairs of $A$. More recently, Br\"{u}stle--Smith--Treffinger \cite{brustle2019wall} showed that the stratification of the $\tau$-tilting fan is part of the stratification of the space of stability conditions of $\modu A$.

For the class of gentle algebras $\mathcal{P}(\kappa)$ arising from triangulations of surfaces with marked points on the boundary and orbifold points of order 3, in this paper we have shown that the reachable component of the $\tau$-tilting fan of $\mathcal{P}(\kappa)$ perfectly matches the $\mathbf{g}$-vector fan arising from the (generalized) cluster theory of the skew-symmetrizable matrix $B(\kappa)$. 
Furthermore, we have shown that the power series that dictate the wall-crossing automorphisms in the Bridgeland stability scattering diagram at the reachable facets of the $\tau$-tilting fan, are precisely Chekhov--Shapiro's generalized exchange polynomials, thus enhancing Br\"{u}stle--Smith--Treffinger's description of the wall-and-chamber structure. As a consequence, the reachable part of the stability scattering diagram coincides with the generalized cluster scattering diagram developed in \cite{mou2020wall, mou2021scattering} and \cite{cheung2021cluster}, providing the first class of them realized via representation theory.

In \Cref{section: tau tilting} and \Cref{section: caldero chapoton}, the ability to compare the $\tau$-tilting theory of $\mathcal{P}(\kappa)$ with that of $\mathcal{P}(\sigma)$ for different triangulations, and to interpret this comparison as the result of simply choosing different initial seeds in the same generalized cluster algebra, began to play an essential role in our arguments, even though, for instance, in \Cref{thm: cc function induces bijection and isomorphism} the statement of the desired result involves only a single triangulation~$\kappa$.

In the case of Jacobian algebras of arbitrary loop-free quivers with potential, if two of these are related by a QP-mutation of Derksen--Weyman--Zelevinsky \cite{derksen2008quivers}, a systematic 
comparison between their $\tau$-tilting theories is provided by mutations of decorated representations \cite{derksen2010quivers}. In the upcoming work \cite{LMpart2}, we will define such DWZ-type mutations of decorated representations of the gentle algebras $\mathcal{P}(\kappa)$ considered in this paper, and use them to parametrize $\tau$-reduced components by integral measured laminations of $\surf$, as well as to prove that the generic Caldero--Chapoton functions on $\tau$-reduced components form a $\mathbb{Z}$-linearly independent subset of the Caldero--Chapoton algebra.

\bibliographystyle{amsalpha-fi-arxlast}
\bibliography{reference.bib}

\end{document}